\documentclass{amsart}
\pdfoutput=1
\usepackage{dsfont}
\usepackage{helvet}
\usepackage[dvipsnames]{xcolor}
\usepackage{mathrsfs}
\usepackage{cancel}

\usepackage{hyperref}

\usepackage{enumerate,amsmath,amsxtra,amsthm,amssymb,fullpage,xr,soul,comment}
\usepackage[shortlabels]{enumitem}

\usepackage[all]{xy}
\usepackage{extarrows}

\usepackage[utf8]{inputenc}
\usepackage{verbatim}
\usepackage{graphicx}
\usepackage{tikz-cd}
\usepackage{bbm}
\usepackage{commath}
\usepackage{stackrel}
\usepackage[bbgreekl]{mathbbol}
\usepackage{stmaryrd}

\newcommand{\ZZ}{ {\mathbf{Z}} }
\newcommand{\QQ}{ {\mathbf{Q}} }

\newcommand{\CC}{ {\mathbf{C}} }
\newcommand{\NN}{ {\mathbf{N}} }

\newcommand{\Zp}{ {\ZZ_p} }
\newcommand{\Qp}{ {\QQ_p} }

\newcommand{\Cp}{ {\CC_p} }

\newcommand{\Ql}{ {\QQ_\ell} }

\newcommand{\frakp}{\mathfrak{p}}
\newcommand{\frakm}{\mathfrak{m}}
\newcommand{\fraka}{\mathfrak{a}}

\newcommand{\units}[1]{{#1}^\times}
\newcommand{\blank}{{-}}

\newcommand{\Iw}{{\rm Iw}}
\newcommand{\Frob}{{\rm Frob}}

\newcommand{\ad}{{{\rm Ad}^0}}
\newcommand{\phiGamma}{\varphi,\Gamma}

\newcommand{\Tor}{{\rm Tor}}

\usepackage{yfonts}

\newcommand{\End}{\mathrm{End}}
\newcommand{\Aut}{\mathrm{Aut}}

\newcommand{\Gal}{\mathrm{Gal}}
\newcommand{\Ad}{{ \mathrm{Ad}}}
\newcommand{\Sym}{{ \mathrm{Sym}}}

\newcommand{\HT}{\mathrm{HT}}

\newcommand{\cris}{\mathrm{cris}}
\newcommand{\dR}{\mathrm{dR}}

\newcommand{\Hom}{{\mathrm{Hom}}}

\newcommand{\Spm}{\mathrm{Spm}}

\newcommand{\hatotimes}{\widehat\otimes}

\newcommand{\Fil}{{\mathrm{Fil}}}

\newcommand{\lra}{\longrightarrow }

\newcommand{\cyc}{\chi_{\rm cyc}}
\newcommand{\RGamma}{\mathbf{R\Gamma}}
\newcommand{\RGammaIw}{\mathbf{R\Gamma}_{\rm Iw}}
\newcommand{\rank}{{\rm rank}}

\newcommand{\comm}[1]{}

\numberwithin{equation}{section}

\makeatletter
\newcommand{\mylabel}[2]{#2\def\@currentlabel{#2}\label{#1}}
\makeatother

\newtheorem{theorem}{Theorem}[section]
\newtheorem{conjecture}{Conjecture}[section]
\newtheorem{definition}[theorem]{Definition}
\newtheorem{lemma}[theorem]{Lemma}

\newtheorem{remark}[theorem]{Remark}

\newtheorem{corollary}[theorem]{Corollary}

\newtheorem{proposition}[theorem]{Proposition}

\newtheorem*{theorem*}{Theorem}
\newtheorem*{corollary*}{Corollary}
\newtheorem*{lemma*}{Lemma}
\newtheorem*{definition*}{Definition}
\newtheorem*{proposition*}{Proposition}
\newtheorem*{remark*}{Remark}
\newtheorem*{conjecture*}{Conjecture}
\newtheorem*{notation*}{Notation}
\newtheorem*{claim*}{Claim}
\newtheorem*{example*}{Example}

\title{Factorization of Algebraic $p$-adic $L$-functions Attached to Adjoint Representations of Coleman Families: Non-critical Case}
\author{F\i{}rt\i{}na K{üçü}k}
\date{September 2023}

\begin{document}

\begin{abstract}
The main objective of this article is to establish the $p$-adic Artin formalism for the algebraic $p$-adic $L$-functions attached to the adjoint representations of Coleman families of modular forms. In particular, we prove a factorization formula involving the determinants of appropriate Selmer complexes, using tools from rigid geometry, homological algebra, Euler systems, and $p$-adic Hodge theory. This work extends an earlier result of Palvannan to the $p$-non-ordinary setting.
\end{abstract}

\maketitle

\setcounter{tocdepth}{1}
\tableofcontents

\section{Introduction}\label{sect:introduction}
\subsection{Introduction and Background}\label{subsect:introbackground}
\par The main objective of this article is to prove the factorization of the algebraic $p$-adic $L$-function attached to the representation:
\begin{align*}
    \Ad(\mathbf{V}_\mathcal{F})(\psi\chi_{\rm cyc}^j)&:=\End_A(\mathbf{V}_\mathcal{F})(\psi\chi_{\rm cyc}^j),
\end{align*}
where $\mathcal{F}$ is a Coleman family over an affinoid algebra $A$ passing through a non-$\theta$-critical $p$-stabilization of a newform $f\in S_{k+2}(\Gamma_1(N),\varepsilon)$, $\mathbf{V}_\mathcal{F}$ is the big Galois representation attached to $\mathcal{F}$, $\psi$ is a non-trivial Dirichlet character and $\chi_{\rm cyc}$ is the $p$-adic cyclotomic character. Define the representation $ \ad(\mathbf{V}_\mathcal{F})$ as:
\begin{align*}
    \ad(\mathbf{V}_\mathcal{F})&:=\ker(\Ad(\mathbf{V}_\mathcal{F})\xrightarrow{{\rm Tr}}A).
\end{align*}
We then have the following decomposition of the representations: \[\Ad(\mathbf{V}_\mathcal{F})(\psi\chi_{\rm cyc}^j)\cong\ad(\mathbf{V}_\mathcal{F})(\psi\chi_{\rm cyc}^j)\oplus A(\psi\chi_{\rm cyc}^j).\]

\par Let $V_f$ denote the Deligne's (cohomological) representation attached to $f$ over a $p$-adic field $E$, and let $\Ad(V_f):={\rm End}_E(V_f)$ and $\Ad^0(V_f):=\ker(\Ad(V_f)\xrightarrow{{\rm Tr}}E)$. Artin formalism asserts that due to the definition of motivic $L$-functions, the $L$-function attached to $\Ad(f)$ factors as:
\[L(\Ad(V_f)(\psi),s)=L(\ad(V_f)(\psi),s)L(\psi,s).\] Establishing the $p$-adic Artin formalism is, however, far from obvious, as $\Ad(V_f)$ does not have any critical $L$-values (in the sense of Deligne). The first result in this direction (following the spirit of Gross' article \cite{gross}) is due to Dasgupta in \cite{dasgupta}, where the modular form of interest is $p$-ordinary:

\begin{theorem*}[Dasgupta]
    Let $f$ be a $p$-ordinary cuspidal eigenform and $\psi$ be a finite order Dirichlet character with conductor coprime to $p$. Then we have the factorization of $p$-adic $L$-functions
    \begin{align*}
        L_p(f\otimes f\otimes\psi,\sigma)=L_p(\Sym^2(f)\otimes\psi,\sigma)L_p(\varepsilon\psi,\sigma/(z\cdot\kappa)),
    \end{align*}
    where $\kappa(z)=z^k$.
\end{theorem*}
 
\par In \cite{palvannan2017selmer}, Palvannan proved that the analogous factorization holds for the associated algebraic $p$-adic $L$-functions (i.e. corresponding Greenberg Selmer groups and their divisors). In \cite{arlandini2021factorisation}, Arlandini and Loeffler generalized the result of Dasgupta in \cite{dasgupta} to the case where $f$ is $p$-non-ordinary and $\psi$ is trivial (see Theorem \ref{thm:Arlandini-Factorization-Result}). 

\par Our main goal in this article is to establish the $p$-adic Artin formalism for the algebraic counterparts of these $p$-adic $L$-functions, where $f$ is $p$-non-ordinary and $\psi$ is non-trivial. In order to accomplish this, we will appeal to Selmer complexes, which were developed by Nekovář in \cite{Nek06} and generalized by Pottharst (in \cite{Pottharst2013} and \cite{pottharst2012cyclotomic}) and Benois (see \cite{benois2014selmer},\cite{benois-p-adic-heights}) to the case where the local conditions at $p$ come from $(\phiGamma)$-modules, and the coefficients of the representations arise from affinoid algebras. In an upcoming work, following the spirit of this article, we will be interested in the case where $f$ is $\theta$-critical, which will require a further investigation.

\subsection{Statement of the main result}\label{subsect:statement-main}
\par Dasgupta proved his result by interpolating $f$ in a Hida family and stated the result by specializing the $p$-adic $L$-function attached to Rankin-Selberg product of Hida families. On the contrary, we deduce our theorem regarding the product of Coleman families by first proving the analogous theorem for a single specialization, then generalize it to families using control theorems for Selmer complexes. Therefore we will first state our main result for a single specialization. Let us present the basic notation and summarize our results below:

\par To simplify the notation, let 
\begin{align*}
    V_\Ad&:=\Ad(V_f)(\psi\chi_{\rm cyc}^j),\\
    V_\ad&:=\ad(V_f)(\psi\chi_{\rm cyc}^j),\\
    V_1:&=E(\psi\chi_{\rm cyc}^j).
\end{align*}
Let $D_\Ad$, $D_\ad$, and $D_1$ be certain $(\phiGamma)$-modules (see Section \ref{sect:phiGamma-and-p-adic-Hodge}, Section \ref{subsect:pottharstselmercomp} and Section \ref{sect:local-cond} for details) attached to $V_\Ad$, $V_\ad$, and $V_1$, respectively. These $(\phiGamma)$-modules are the choice of local conditions at $p$. For the data \[(V,D_V)\in\{(V_{\Ad},D_{\Ad}),(V_{\ad},D_{\ad}),(V_1,D_1)\},\] let $\RGammaIw(V,D_V)$ denote the Pottharst style Selmer complex attached to $(V,D_V)$. We will denote the cohomologies of these complexes by $H^i(V,D_V)$. Let $D_V^\perp$ denote the orthogonal complement of $D_V$ with respect to Grothendieck duality. Finally let $\mathcal{H}=\mathcal{H}_E$ denote the large Iwasawa algebra introduced by Perrin-Riou.
\begin{theorem*}[Theorem \ref{Main-Theorem-Specialization}]
    Under the hypotheses \ref{hyp:p-regularity}-\ref{hyp:big-img-IV} of Section \ref{subsect:assumptions-main-thm} we have:
    \begin{enumerate}[(1)]
        \item For the data \[(V,D_V)\in\{(V_{\Ad},D_{\Ad}),(V_{\ad},D_{\ad}),(V_1,D_1)\},\] the Selmer complexes $\RGammaIw(V,D_V)$ and $\RGammaIw(V^*(1),D_V^\perp)$ lie in the category of perfect $\mathcal{H}$-modules with perfect amplitude at $[1,2]$.
        \item The cohomologies $H^i_\Iw(V,D_V)$ and $H^i_\Iw(V^*(1),D_V^\perp)$ vanish except $i\neq 2$, and we have the following short exact sequences:
        \begin{align*}
            0 \rightarrow H^2_\Iw(V_{\ad},D_{\ad}) &\rightarrow H^2_\Iw(V_{\Ad},D_{\Ad})  \rightarrow H^2_\Iw(V_1,D_1) \rightarrow 0,\\
            0 \rightarrow H^2_\Iw(V_1^*(1),D_1^\perp) \rightarrow & H^2_\Iw(V_{\Ad}^*(1),D_{\Ad}^\perp)\rightarrow  H^2_\Iw(V_{\ad}^*(1),D_{\ad}^\perp)\rightarrow 0.
        \end{align*}
        \item For the pairs $(V,D_V)$ as above, $H^2_\Iw(V,D_V)$ and $H^2_\Iw(V^*(1),D_V^\perp)$ are torsion, and we obtain the following factorizations:
        \begin{align*}
            \det(\RGammaIw(V_{\Ad},D_{\Ad}))&=\det(\RGammaIw(V_{\ad},D_{\ad}))\hatotimes_{\mathcal{H}}\det(\RGammaIw(V_1,D_1)),\\
            \det(\RGammaIw(V_{\Ad}^*(1),D_{\Ad}^\perp))&=\det(\RGammaIw(V_{\ad}^*(1),D_{\ad}^\perp))\hatotimes_\mathcal{H}\det(\RGammaIw(V_1^*(1),D_1^\perp)).
        \end{align*}
    \end{enumerate}
\end{theorem*}

\par Let $\mathbf{V}_\Ad$ denote $\Ad(\mathbf{V}_\mathcal{F})(\psi\chi_{\rm cyc}^j)$, we define $\mathbf{V}_\ad$ and $\mathbf{V}_1$ similarly. For $\mathbf{V}$ one of those representations, let $\mathbf{D}_\mathbf{V}$ denote the local condition at $p$ attached to $\mathbf{V}$, and let $\mathbf{D}_\mathbf{V}^\perp$ denote the corresponding dual local condition. From the theorem above we deduce the following: 
\begin{theorem*}[Theorem \ref{Main-Theorem}]
    Under our hypotheses in Section \ref{subsect:assumptions-main-thm}, we have:
\begin{enumerate}[(1)]
    \item For the data \[(\mathbf{V},\mathbf{D}_\mathbf{V})\in\{(\mathbf{V}_{\Ad},\mathbf{D}_{\Ad}),(\mathbf{V}_{\ad},\mathbf{D}_{\ad}),(\mathbf{V}_1,\mathbf{D}_1)\},\] the Selmer complexes $\RGammaIw(\mathbf{V},\mathbf{D}_\mathbf{V})$ and $\RGammaIw(\mathbf{V}^*(1),\mathbf{D}_\mathbf{V}^\perp)$ lie in the category of perfect complexes of $\mathcal{H}_A:=\mathcal{H}\hatotimes_E A$-modules with perfect amplitude at $[1,2]$.

    \item The cohomologies $H^i_\Iw(\mathbf{V},\mathbf{D}_\mathbf{V})$ and $H^i_\Iw(\mathbf{V}^*(1),\mathbf{D}_\mathbf{V}^\perp)$ vanish except $i\neq 2$, and we have the short exact sequences:
    \begin{align*}
        0 \rightarrow H^2_\Iw(\mathbf{V}_{\ad},\mathbf{D}_{\ad}) &\rightarrow H^2_\Iw(\mathbf{V}_{\Ad},\mathbf{D}_{\Ad})  \rightarrow H^2_\Iw(\mathbf{V}_1,\mathbf{D}_1) \rightarrow 0,\\
        0 \rightarrow H^2_\Iw(\mathbf{V}_1^*(1),\mathbf{D}_1^\perp) \rightarrow & H^2_\Iw(\mathbf{V}_{\Ad}^*(1),\mathbf{D}_{\Ad}^\perp)\rightarrow  H^2_\Iw(\mathbf{V}_{\ad}^*(1),\mathbf{D}_{\ad}^\perp)\rightarrow 0.
    \end{align*} 
    \item We have the following factorizations:
    \begin{align*}
        \det(\RGammaIw(\mathbf{V}_{\Ad},\mathbf{D}_{\Ad}))&=\det(\RGammaIw(\mathbf{V}_{\ad},\mathbf{D}_{\ad}))\hatotimes_{\mathcal{H}_A}\det(\RGammaIw(\mathbf{V}_1,\mathbf{D}_1)),\\
        \det(\RGammaIw(\mathbf{V}_{\Ad}^*(1),\mathbf{D}_{\Ad}^\perp))&=\det(\RGammaIw(\mathbf{V}_{\ad}^*(1),\mathbf{D}_{\ad}^\perp))\hatotimes_{\mathcal{H}_A}\det(\RGammaIw(\mathbf{V}_1^*(1),\mathbf{D}_1^\perp)).
    \end{align*}
\end{enumerate}
\end{theorem*}

\par For the perfectness of the complexes above, we prove and use the following theorem, in which we generalize the arguments of \cite[\S 4.5]{danielebuyukboduksakamoto2023} to our setting:
\begin{theorem*}[Theorem \ref{thm:Perfectness-amplitude-1-2}]
    Assume that the following conditions hold:
\begin{enumerate}[(i)]
    \item $H^3_\Iw(\mathbf{V},\mathbf{D})=0$.
    \item For each prime ideal $\mathfrak{p}$ of $\Lambda_A[1/p]:=\Lambda[1/p]\hatotimes A$, we have $(\mathbf{V}\otimes_A(\Lambda_A[1/p]^\iota)/\frakp)^{G_{\QQ,S}}=0$, 
    \item For each $\ell\in S_f$, $\ell\neq p$, $H^1(I_\ell,\mathbf{V})$ is projective as $A$-module. 
\end{enumerate}
Then $\RGamma_{\rm Iw}(\mathbf{V},\mathbf{D})\in \mathscr{D}_{\rm parf}^{[1,2]}(\mathcal{H}_A)$.
\end{theorem*}

\par We also establish a partial result on comparison of algebraic $p$-adic $L$-functions (i.e. determinants of appropriate Selmer complexes) attached to $\mathbf{V}_\mathcal{F}\hatotimes_E\mathbf{V}_\mathcal{F}^*(\psi\chi_{\rm cyc}^j)$ and $\mathbf{V}_\Ad$, see Theorem \ref{thm:comparison-3-Ad} for the exact statement. 

\par We will end this section by pointing out some key differences between the result of Palvannan and our theorems. Firstly, Palvannan assumes that $f$ is $p$-ordinary, in this case there is a $\Gal(\overline{\QQ}_p/\Qp)$-stable subspace of $V_f$ (and $\mathbf{V}_\mathbf{f}$, where $\mathbf{f}$ is the Hida family passing through $f$) which satisfies the Panchishkin condition. Palvannan forms the local conditions attached to the relevant representations by using this subspace. In the $p$-non ordinary case, which we assume, $V_f$ and $\mathbf{V}_\mathcal{F}$ are irreducible as $\Gal(\overline{\QQ}_p/\Qp)$ representations, hence there is no such subspace. This is why we need to work with Pottharst style Selmer complexes, which allows us to choose our local conditions as $(\phiGamma)$-modules that comes from the ``triangulation'' of $\mathbf{V}_\mathcal{F}$ (see Theorem \ref{thm:triangulation} for details). We should note that the theory of $(\phiGamma)$-modules are very closely related to $p$-adic Hodge theory. We also use the advantage of appealing to the theory of Selmer complexes, since it is very convenient to use the language of derived categories. 

\par Another technical difference is that in the non-ordinary case, the $p$-adic $L$-functions are no longer integral, this is why we use $\mathcal{H}$, the algebra of locally analytic distributions on $\Gamma_0$, for cyclotomic deformations. Note that $\mathcal{H}$ is a non-Noetherian ring. Fortunately there is a rich theory of coadmissible modules over $\mathcal{H}$ thanks to Schneider and Teitelbaum \cite{Schneider-Teitelbaum}, which is applied to the theory of Selmer complexes by Pottharst and Benois. We also require some tools from rigid geometry, as our coefficient ring is an affinoid algebra and $\mathcal{H}$ is the global sections of the structure sheaf on rigid open unit ball.

\par Finally we utilize Beilinson-Flach Euler system in Section \ref{sect:ES-BF} to establish the analogue of the AD-TOR assumption of \cite{palvannan2017selmer} in our setting. In order to do this, we adapt the results of Loeffler and Zerbes in \cite{LZ16} to our setting, combining with the ideas from \cite{LoefflerZerbesSym2J_Reine_Angew_Math_2019} and \cite{BuyukbodukLeiVenkat2021Documenta}.

\subsection{Notation and Conventions}\label{subsect:notations}
\par Let $p$ always be an odd prime. For a field $F$ (with characteristic $0$) we will denote $\Gal(\overline{F}/F)$ by $G_F$, and when $F=\Qp$ we will write simply $G_p=G_{\Qp}$. If $F$ is a finite extension of $\QQ$ or $\Qp$, we will denote the ring of integers of $F$ by $\mathcal{O}_F$. 

\par Let $\QQ^{\rm cyc}=\bigcup_{n\geq 0}\QQ(\zeta_{p^{n}})$ be the compositum of the $p$-power cyclotomic extensions of $\QQ$, and let $\Gamma:=\Gal(\QQ^{\rm cyc}/\QQ)\cong \ZZ_p^\times$. Let $\Delta:=\Gal(\QQ(\zeta_p)/\QQ)$ be the torsion part of $\Gamma$, and for all $n\geq 0$, let $\QQ_n:=\QQ(\zeta_{p^{n+1}})^\Delta$, $\QQ_\infty:=\bigcup_{n\geq 0}\QQ_n=(\QQ^{\rm cyc})^\Delta$ be the cyclotomic $\Zp$ extension of $\QQ$, and $\Gamma_n:=\Gal(\QQ_\infty/\QQ_n)$. Then we have $\Gamma\cong\Gamma_0\times\Delta$. Under this isomorphism, we also have the decomposition of characters $\cyc=\omega\cdot\langle\cyc\rangle$, where $\omega$ is the Teichmüller character. Note that the $p$-adic cyclotomic character $\cyc$ gives a canonical isomorphism $\cyc:\Gamma\xrightarrow{\sim}\ZZ_p^\times$, and we have $\Gamma_0\cong (1+p\Zp)\cong \Zp$ and $\Delta\cong (\ZZ/p\ZZ)^\times$. 

\par For a finite extension $K$ of $\Qp$ we can also define $K^{\rm cyc}=K\cdot\QQ^{\rm cyc}=\bigcup_{n\geq 0}K(\zeta_{p^{n}})$, $\Gamma_K:=\Gal(K^{\rm cyc}/K)$, $\Delta_K:=\Gal(K(\zeta_p)/K)$, $K_n:=K(\zeta_{p^{n}})^\Delta$, $K_\infty:=(K^{\rm cyc})^\Delta$, and $\Gamma_K^0:=\Gal(K_\infty/K)$. For our purposes we will assume that $K$ is unramified over $\Qp$.

\par Note that $\Gamma\cong\Gamma_{\Qp}$, $\Delta\cong\Delta_{\Qp}$, $\Gamma_0\cong\Gamma_{\Qp}^0$. 

\par We define the Iwasawa algebra $\Lambda_E(\Gamma)$ as \[\Lambda_E(\Gamma)=\mathcal{O}_E[[\Gamma]]:=\varprojlim_{H}\mathcal{O}_E[\Gamma/H].\] We also set the large Iwasawa algebra, introduced by Perrin-Riou in \cite{perrin_riou} as \[\mathcal{H}_E(T):=\{f\in E[[T]]:f\textnormal{ converges in the open unit ball}\}.\]

\par If $V$ is a $p$-adic representation with coefficients in a finite extension of $\Qp$ which is Hodge-Tate, then we have the following Hodge-Tate decomposition: \[V\otimes\Cp\cong\bigoplus_{i\in\ZZ}\Cp(i)^{h_i},\] where $h_i=0$ for all but finitely many integers $i$. The indices $i$ for which $h_i\neq 0$ will be called the Hodge-Tate weights of $V$. In that case, the jumps appearing in the de Rham filtration of $\mathbf{D}_{\rm dR}(V)$ will be the negatives of the Hodge-Tate weights of $V$. In particular, the Hodge-Tate weight of the $p$-adic cyclotomic character will be $+1$.

\par Let $M$ be a pure motive, which is geometric. The $L$-function attached to $M$ is defined as: \[L(M,s):=\prod_\ell \det (1-\Frob^{-1}_\ell X|M_p^{I_\ell})_{X=\ell^{-s}}^{-1},\] where $\ell$ runs through all integer primes with $p\neq \ell$. Here the geometric Frobenius at $\ell$ is denoted by $\Frob^{-1}_\ell$. For instance, the $L$-function attached to a newform $f\in S_{k+2}(\Gamma_0(N))$, which is defined as \[L(f,s):=\sum_{n=1}^{\infty}\frac{a_n(f)}{n^s}\] will differ from the motivic $L$-function $L(M(f),s)$ up to a finitely many Euler factors, where $M(f)$ denotes the motive attached to $f$ by Scholl in \cite{scholl}. The $p$-adic realization $V_f:=M_p(f)$ of $M(f)$ is Deligne's cohomological representation attached to $f$, and with our conventions on the Hodge-Tate weights, $V_f$ has the Hodge-Tate weights $\{-1-k,0\}$.

\par To fix the possible confusion with our conventions when we can consider the twists by Dirichlet characters, we will assume that if $\psi$ denotes a Dirichlet character, then $\psi(p)=\psi(\Frob^{-1}_p)$, as done in \cite{arlandini2021factorisation}. We also note that we will write $\psi+j$ to mean the character $\psi\chi_{\rm cyc}^j$.

\subsection{Acknowledgements}
This article is derived from my PhD thesis \cite{mythesis}, which was supervised by Assoc. Prof. Kâzım Büyükboduk and funded by University College Dublin PhD Advance Scheme. Therefore, I would first like to thank my supervisor for suggesting this problem and for his incredible support during my Ph.D and the preparation of my thesis and this article. I would also like to thank to my colleagues Dr. Erman Işık and Dr. Daniele Casazza, and also to the external examiner for my viva voce, Prof. Matteo Longo, for their valuable suggestions.

\section{Affinoid Spaces and the Weight Space}\label{sect:affinoid}
In this section, we will review definitions and some properties of affinoid algebras, their spectra and the weight space. The details on affinoid algebras and affinoid spaces can be found in \cite{bosch}, and the details on the weight space can be found in \cite{eigenbook} and \cite{benoisbuyukboduk2020critical}.

\subsection{Affinoid Algebras}\label{subsect:Affinoid-alg}
\begin{definition}[Tate algebra] Let $E$ be a finite extension of $\Qp$. The $E$-algebra \[E\langle x_1,...,x_n\rangle:=\left\{\sum_{\alpha\in\NN^n}c_\alpha x^\alpha\in E[[x_1,...,x_n]]: \abs{c_\alpha}\rightarrow 0\textrm{\ as }\alpha\rightarrow \infty\right\}\] is called the Tate algebra.
\end{definition}

\par We will also denote $E\langle x_1,...,x_n\rangle$ by $T_n$, if $E$ is inferred from the context. Tate algebras are also characterized by the following property given in \cite[\S 2.2 Lemma 1]{bosch}:
\begin{proposition}
    A formal power series with $n$ variables \[f=\sum_{\alpha\in\NN^n}c_\alpha x^\alpha\in E[[x_1,...,x_n]]\] lies in $T_n$ iff $f$ converges on the closed unit ball \[B^n(\overline{E}):=\{x=(x_1,...,x_n)\in \overline{E}^n: \abs{x_i}\leq 1\}.\]
\end{proposition}

\par Some properties of the Tate algebras are as follows: For each $n$, $T_n$ is UFD, Noetherian, Jacobson, and regular of Krull dimension $n$. For each maximal ideal $\frakm$ of $T_n$, the residue field $T_n/\frakm$ is a finite extension of $E$. The Tate algebra $T_n$ is also complete with respect to the Gauss norm: \[\norm{f}:={\rm max}_\alpha\abs{c_\alpha},\] where \[f(x)=\sum_\alpha c_\alpha x^\alpha\in T_n.\] 

\par In particular, $T_1$ is an Euclidean domain, hence it is a PID. This last property will be useful when we discuss nice affinoid neighborhoods in Definition \ref{def:nice-affinoid-neigh}.

\begin{definition}[Affinoid algebra]
    An $E$-algebra $A$ is called an affinoid algebra over $E$ if there exists an $E$-algebra epimorphism $e:T_n\twoheadrightarrow A$ from a Tate algebra $T_n$ to $A$. 
\end{definition}

\par In other words, affinoid algebras are quotients of Tate algebras by some ideals. Affinoid algebras are Noetherian and Jacobson. There is a norm on affinoid algebras induced by the Gauss norm on Tate algebras, and all $E$-algebra morphisms between affinoid algebras are continuous. Affinoid algebras also satisfy Noether normalization theorem, i.e. for some $d$ there exists a finite monomorphism $T_d\rightarrow A$, and in particular the Krull dimension is $d$.

\par There exists a category of affinoid $E$-algebras with morphisms given by $E$-algebra maps which admits pushouts, where we give the universal property of pushouts in the following (see \cite[\S 3.1 Proposition 2, Appendix A Proposition 6]{bosch}):

\begin{proposition}[Completed tensor product]
    Let $R$, $A_1$, $A_2$ be affinoid $E$-algebras such that $A_1$ and $A_2$ are $R$-algebras. Then there exists a completed tensor product $A_1\hatotimes_R A_2$ of $A_1$ and $A_2$ over $R$ satisfying the following universal property: There are $R$-algebra homomorphisms $i_1:A_1\rightarrow A_1\hatotimes_R A_2$ and $i_2:A_2\rightarrow A_1\hatotimes_R A_2$, and for any $R$-algebra $A$ lying in this category and $R$-algebra homomorphisms $f_1:A_1\rightarrow A$ and $f_2:A_2\rightarrow A$, there exists a unique $R$-algebra homomorphism $f:A_1\hatotimes_R A_2\rightarrow A$ such that the following diagram commutes:
    \[ \begin{tikzcd}[column sep=large, row sep=large]
        & A & \\
        A_1 \arrow[ru, "f_1"] \arrow[r,"i_1"] & A_1\hatotimes_R A_2 \arrow[u, dashed, "\exists! f"'] & A_2 \arrow[l, "i_2"'] \arrow[lu, "f_2"'].
    \end{tikzcd}\]
\end{proposition}

\par The completed tensor product $A_1\hatotimes_R A_2$ can be constructed so that the maps $i_1:A_1\rightarrow A_1\hatotimes_R A_2$ and $i_2:A_2\rightarrow A_1\hatotimes_R A_2$ can be given as $a_1\mapsto a_1\otimes 1$ and $a_2\mapsto 1 \otimes a_2$, respectively.

\subsection{Affinoid and Rigid Spaces}\label{subsect:Affinoid-space}

\begin{definition}[Affinoid space]
$X$ is called an affinoid space over $E$ if $X=\Spm(A)$ for some affinoid algebra $A$ over $E$, where $\Spm(A)$ is the maximal spectra of $A$ (i.e. the set of maximal ideals of $A$). $X$ is equipped with the following data:
\begin{itemize}
    \item $X$ has the Zariski topology, where the closed subsets of $X$ are precisely the following: \[V(\fraka):=\{x\in\Spm(A):\fraka\subseteq\frakm_x\},\] where $\fraka$ is an ideal of $A$, and $\frakm_x$ is the corresponding maximal ideal in $A$ to the point $x$,
    \item $A$ is the ring of ``functions'' on $X$ viewed as follows: if $f\in A$ and $x\in\Spm(A)$, then $f(x)=f{\rm\ mod\ }\frakm_x\in A/\frakm_x$.
\end{itemize}
\end{definition}

\par Note that in this case, for an ideal $\fraka$ of $A$, the closed subset $V(\fraka)$ of $\Spm(A)$ is precisely \[V(\fraka):=\{x\in\Spm(A):f(x)=0{\rm\ for\ all\ }f\in\fraka\}.\] The morphisms of affinoid spaces are defined so that the category of affinoid spaces is the opposite category of the category of affioid algebras. For instance, the pullbacks (fibered products) are defined as $\Spm(A_1)\times_{\Spm(R)}\Spm(A_2):=\Spm(A_1\hatotimes_R A_2)$ and the natural maps $p_1:X_1\times_X X_2\rightarrow X_1$ and $p_2:X_1\times_X X_2\rightarrow X_2$ (which correspond to $i_1$ and $i_2$ in the opposite category) are the usual projections $(x_1,x_2)\mapsto x_1$ and $(x_1,x_2)\mapsto x_2$ respectively. 

\par Zariski topology on $\Spm(A)$ enjoys many similar properties to the topology on the prime spectrum ${\rm Spec}(A)$. For instance, the sets of the form \[D_f:=\{x\in\Spm(A):f(x)\neq 0\}\cong\Spm(A_f)\] for $f\in A$ forms a basis for this topology. There are more interesting topologies, such as canonical topology, weak and strong Grothendieck topologies, which are finer than the Zariski topology. The details can be found on Chapter 3 and 4 of \cite{bosch}.

\begin{definition}[Affinoid subdomain]
    A subset $U\subseteq X=\Spm(A)$ is called an affinoid subdomain of $X$ if there exists a morphism of affinoid spaces $\iota:X^\prime\rightarrow X$ such that $\iota(X^\prime)\subseteq U$ and this map is universal with this property, i.e. for any morphism $\varphi:Y=\Spm(B)\rightarrow X$ with $\varphi(Y)\subseteq U$, there exists a unique morphism $\varphi^\prime: Y\rightarrow X^\prime$ such that $\varphi=\iota\varphi^\prime$.
\end{definition}

\par It turns out that in this case, $\iota$ is a bijection from $X^\prime$ onto $U$, and $X^\prime$ is unique up to unique isomorphism, hence we can regard $U=X^\prime=\Spm(A_U)$. Also for any maximal ideal $\frakm_x\in U$, we have $A_{\iota(\frakm_x)}\cong {A_U}_{\frakm_x}$ and $A/{\iota(\frakm_x)}^n\cong {A_U}/{\frakm_x}^n$ for any $n\in\NN$ (see \cite[\S 3.3 Lemma 10]{bosch}). 

\par An important property of affinoid subdomains for our discussions is as follows (see \cite[\S 4.1 Corollary 5]{bosch}):
\begin{proposition}\label{prop:flatness_of_affinoid_subdomains}
    If $U\subseteq X=\Spm(A)$ is an affinoid subdomain of $X$, then the restriction map $A\rightarrow A_U$ is flat.
\end{proposition}

\begin{definition}[Admissible open subset]
    A subset $U\subseteq X$ is called admissible open if there exists a covering of $U$ by affinoid subdomains $\{U_i\}_i$ such that for each morphism of affinoid spaces $\varphi: Y\rightarrow X$ with $\varphi(Y)\subseteq U$, there is a finite refinement of the covering $\{\varphi^{-1}(U_i)\}_i$ of $Y$.
\end{definition}

\par For example, all Zariski open subsets of $X$ are admissible, in particular they are covered by finitely many affinoid subdomains of $X$. Another important example is the open unit ball. For $n\in\NN$, let \[W_n:=\Spm \left(E\langle X/p^{1/n}\rangle\right)\] be the closed ball with radius $p^{-1/n}$. Then the open unit ball is admissibly covered by these closed balls $W_n$.

\section{$\left(\phiGamma\right)$-Modules and $p$-adic Hodge Theory}\label{sect:phiGamma-and-p-adic-Hodge}

\subsection{Robba Rings and $(\varphi,\Gamma)$-modules}\label{subsect:robbaphigamma}

\begin{definition}\label{def:Robba-Ring}
	Let $r\geq 0$, $A(r,1):=\{z\in\Cp :p^{-\frac{1}{r}}\leq\abs{z}_p<1\}$, and $\pi$ denotes a formal variable. The Robba ring $\mathcal{R}_K^{(r)}$ is defined as the set of formal series $f(\pi)=\sum_{n\in\ZZ}a_n\pi^n$ with $a_n\in K$ such that $f(\pi)$ converges in the annulus $A(r,1)$. If $A$ is an affinoid algebra over $\Qp$, then we set $\mathcal{R}_{K,A}^{(r)}:=A\hatotimes_{\Qp}\mathcal{R}_K^{(r)}$ and $\mathcal{R}_{K,A}:=\cup_{r>0}\mathcal{R}_{K,A}^{(r)}$. If $K$ is clear from the context, we will denote those Robba rings simply by $\mathcal{R}_A^{(r)}$ and $\mathcal{R}_A$, respectively.
\end{definition}

\par Note that for $s\leq r$, we have $A(r,1)\subseteq A(s,1)$ which implies that $\mathcal{R}_K^{(s)}\subseteq\mathcal{R}_K^{(r)}$.

\begin{definition}\label{def:phiGamma-actions}
	Let $E$ be a finite extension of $\Qp$ and $f\in\mathcal{R}_E$. The continuous $E$-linear map $\varphi:=\mathcal{R}_E\to\mathcal{R}_E$ is defined as \[\varphi(f(\pi)):=f((1+\pi)^p-1).\] Similarly, for any $\gamma\in\Gamma$ we define the action of $\Gamma$ on $\mathcal{R}_E$ by \[\gamma(f(\pi)):=f((1+\pi)^{\cyc(\gamma)}-1).\]
\end{definition}

\par A $(\varphi,\Gamma)$-module over a Robba ring $\mathcal{R}_A$ is a finitely generated locally free $\mathcal{R}_A$-module equipped with a $\varphi$-semilinear map and a semilinear continuous action of $\Gamma$ satisfying some certain conditions (see \cite{benois-p-adic-heights}, \cite{KedlayaPottharstXiao} for the details). The $(\phiGamma)$-modules over $\mathcal{R}_A$ form a Tannakian category denoted by $\mathbf{M}_{\mathcal{R}_A}^{\phiGamma}$. 

Let $\mathbf{V}$ be a $p$-adic representation of $G_p$ with coefficients in an affinoid algebra $A$. To each such representation $\mathbf{V}$ one can attach a $(\phiGamma)$-module $\mathbf{D}_{{\rm rig},A}^\dagger(V)$ in a functorial way. If $A=E$ is a finite extension of $\Qp$, then we will write $V$ instead of $\mathbf{V}$, and we will denote the associated $(\phiGamma)$-module by $\mathbf{D}_{\rm rig}^\dagger(V)$. In particular, we have $\mathbf{D}_{\rm rig}^\dagger(\Qp(1))=\mathcal{R}_E(1)$.

\par If $x\in\Spm (A)$ with corresponding maximal ideal $\frakm_x$ and the residue field $E_x=A/\mathfrak{m}_x$, then the specialization at $x$ is a natural transformation from $\mathbf{D}_{{\rm rig},A}^\dagger$ to $\mathbf{D}_{\rm rig}^\dagger(V)$, i.e. $\mathbf{D}_{{\rm rig},A}^\dagger(V)_x=\mathbf{D}_{\rm rig}^\dagger(V_x)$. 

\par Note that not all $(\phiGamma)$-modules arise from $p$-adic representations in this way, and in the case $A=E$, the essential image of $\mathbf{D}_{\rm rig}^\dagger$ is the subcategory of étale (or slope zero) $(\phiGamma)$-modules.

\subsection{A brief introduction to $p$-adic Hodge Theory}\label{subsect:p-adic-Hodge}
\par In this subsection, we will give a summary of the contents of \cite{berger}, \cite{brinon-conrad}, and \cite[\S 2.2]{benois-p-adic-heights}.

\par First note that not all $G_p$-representations over $p$-adic fields are Hodge-Tate, see \cite[\S 2]{brinon-conrad} for definitions and details. However, the ones that we are interested in will be Hodge-Tate representations. Let $E$ be a finite extension of $\Qp$ and let $V$ be a representation of $G_\QQ$ over $E$ which is Hodge-Tate. It then has a Hodge-Tate decomposition of the following form:
\[ V\otimes_\Qp\Cp\cong\sum_{j\in\ZZ}\Cp(j)^{h_j},\] where $h_j$'s are non-negative integers such that $h_j=0$ for all but finitely many $j$. 

\par Recall that by our conventions on Hodge-Tate weights, integers $j$ for which the corresponding $h_j$'s are non-zero are called Hodge-Tate weights of $V$. In particular, the Hodge-Tate weight of $\Qp(1)$ is $1$. We will denote the multiset of Hodge-Tate weights of $V$ by ${\rm HT}(V)$.

\begin{definition}\label{def:Filtered-Modules}
    A filtered module $D$ is a finite dimensional $E$-vector space equipped with a decreasing filtration $({\rm Fil}^i D)_{i\in\ZZ}$. A filtered $(\varphi,N)$-module is a filtered module with a semilinear (w.r.t. the action of $G_p$) bijective map $\varphi: D\rightarrow D$ and a $E$-linear monodromy operator $N$ with $N\varphi=p\varphi N$. A filtered $\varphi$-module is a filtered $(\varphi,N)$-module with $N=0$. We will denote the Tannakian categories of these modules by $\mathbf{MF}_E$, $\mathbf{MF}_E^{\varphi,N}$, and $\mathbf{MF}_E^{\varphi}$, respectively.
\end{definition}

\par There are also filtered $(\varphi,N,G_p)$-modules, see \cite{benois-p-adic-heights} for the precise definition. The category of such modules are denoted by $\mathbf{MF}_{E}^{\varphi,N,G_p}$.

\par Fontaine's formalism gives period rings $\mathbf{B}_*$ for $*\in\{\rm dR,pst,st,cris\}$, and functors $\mathbf{D}_{*}:(\mathbf{B}_{*}\otimes_{\Qp}\blank)^{G_p}$ from the category of $p$-adic representations over $E$ to the categories of ``Filtered modules'', ``Filtered $(\varphi,N,G_p)$-modules'', ``Filtered $(\varphi,N)$-modules'' and ``Filtered $\varphi$-modules'' for $*={\rm dR,pst,st,cris}$ respectively. 

\begin{definition}\label{def:B-admissible-modules}
     We say that $V$ is $\mathbf{B}$-admissible if the dimension of $V$ is preserved under these functors. We say that $V$ is de Rham, potentially semistable, semistable, and crystalline if $V$ is $B_{\rm dR}$, $B_{\rm pst}$, $B_{\rm st}$, $B_{\rm cris}$ admissible, respectively. 
\end{definition}

\begin{theorem}[Berger]\label{thm:Berger-phiGamma}
    A $(\phiGamma)$-module $\mathbf{D}$ is potentially semistable if and only if it is de Rham. Let $*$ stand for ${\rm pst,st,cris}$, and $?$ for $\mathbf{MF}_{E}^{\varphi,N,G_p}$, $\mathbf{MF}_{E}^{\varphi,N}$, $\mathbf{MF}_{E}^{\varphi}$ respectively. The functors $\mathbf{D}_{*}$ factors from the category of $(\phiGamma)$-modules $\mathbf{M}_{\mathcal{R}_E}^{\phiGamma}$ such that the following diagrams commute: 

    \[ \begin{tikzcd}
       \mathbf{Rep}_E(G_p) \arrow[r, "\mathbf{D}_{\rm rig}^{\dagger}"] \arrow[rd, "\mathbf{D}_{*}"]
       & \mathbf{M}_{\mathcal{R}_E}^{\phiGamma} \arrow[d, dashed, "\exists!\mathscr{D}_{*}"] \\
         & ?
   \end{tikzcd}\]
\end{theorem}

\par Note that there is also a relation between de Rham filtration and the Hodge-Tate weights. If $V$ is de Rham, then it is Hodge-Tate, and the jumps of the filtration ${\rm Fil}^i\mathbf{D}_{\rm dR}(V)$ (i.e. the $i$'s for which ${\rm Fil}^i\neq {\rm Fil}^{i+1}$) are the additive inverses of the Hodge-Tate weights of $V$, and the dimensions of the graded pieces ${\rm Gr}^i(\mathbf{D}_{\rm dR}(V)):={\rm Fil}^i\mathbf{D}_{\rm dR}(V)/{\rm Fil}^{i+1}\mathbf{D}_{\rm dR}(V)$ are the multiplicities of the corresponding Hodge-Tate weights.

\par Let $V$ be a representation of $G_\QQ$, which is de Rham at $p$. Let $d^+(V):=V^{c=1}$ denote the dimension of the plus-eigenspace, where $c$ is any complex conjugation in $G_\QQ$. Let $\#{\rm HT}^+(V)$ denote the number of positive Hodge-Tate weights.

\begin{definition}\label{def:critical-Deligne-p-adic}
    We say that $V$ is critical if $d^+(V)=\#{\rm HT}^+(V)$. If $V$ is critical, then it is Panchishkin ordinary if there exists a $G_p$-equivariant exact sequence 
    \[0\rightarrow F^+(V) \rightarrow V \rightarrow F^-(V) \rightarrow 0\] with $\dim_E F^+(V)= d^+(V)$ and all Hodge-Tate weights of $F^+(V)$ are positive.
\end{definition}

\begin{definition}[Deligne's criticality condition]\label{def:critical-Deligne}
    We say that $s=n$ is a critical value of the $L$-series $L(V,s)$ associated to the representation $V$ (or to the motive $M$ such that $p$-adic realization of $M$ is $V$) if the product of the Euler factors at infinity \[L_\infty(V,s)L_\infty(V^*(1),-s)\] is holomorphic at $s=n$.
\end{definition} 

\begin{remark}
    Note that $s=n$ is critical (in the sense of Deligne) for $V$ if and only if $V(n)$ is critical in the sense of Definition \ref{def:critical-Deligne-p-adic}. First observe that $s=n$ is critical for $V$ if and only if $s=0$ is critical for $V(n)$ (see our conventions for $L$-functions attached to the twists of $V$ in the introduction). Hence we claim that $s=0$ is critical for $V$ if and only if $V$ is critical as in Definition \ref{def:critical-Deligne-p-adic}. For the equivalence between these two definitions see \cite[Lemma 2.3]{coates-perrin-riou}. Note also that in op. cit. what we call $V^*(-1)$ is denoted as $V^*$, hence Definition \ref{def:critical-Deligne} coincides with \cite[Definition 2.2]{coates-perrin-riou}.
\end{remark}

\par The following definition by \cite{bergerTrianguline} is a generalization of Panchishkin conditions for the representations attached to the $p$-non-ordinary forms.
\begin{definition}[Trianguline representations]\label{def:Berger-Trianguline}
    $V$ is called trianguline if there exists a finite extension $F$ of $E$ such that $\mathbf{D}_{{\rm rig},F}^\dagger(V\otimes_E F)$ is a successive extensions of rank 1 $(\phiGamma)$-modules as follows:
    \[0=:D_0\subseteq D_1 \subseteq D_2 \subseteq \dots \subseteq D_n:=\mathbf{D}_{{\rm rig},F}(V\otimes_E F).\]
\end{definition}

\section{Eigencurve, Coleman Families and Triangulations}\label{sect:families-triangulation}
\subsection{The Weight Space}\label{subsect:weight-space-hida}

\par Let $\Gamma$ be a profinite group, and let $H$ run through open subgroups of $\Gamma$. Recall that we define the Iwasawa algebra $\Lambda_E(\Gamma)$ as \[\Lambda_E(\Gamma)=\mathcal{O}_E[[\Gamma]]:=\varprojlim_{H}\mathcal{O}_E[\Gamma/H].\] The particular cases in which we will be most interested are $\Gamma\cong \ZZ_p^\times$ or $\ZZ_p$. Note that we have \[\Lambda_E(\ZZ_p^\times)\cong \mathcal{O}_E[\Delta]\otimes\Lambda_E(\Zp),\] and $\Lambda_E(\Zp)\cong \mathcal{O}_E[[T]]$, where the last isomorphism is given by sending a topological generator $\gamma\in\Zp$ to $1+T$.

\par The weight space $\mathcal{W}$ is defined as the Hom-set of continuous group homomorphism from $\ZZ_p^\times$ to the multiplicative group scheme, i.e. if $R$ is a Banach $E$-algebra then \[\mathcal{W}(R):=\Hom_{\rm cts,grp}(\ZZ_p^\times,R^\times).\] This definition is equivalent to saying that \[\mathcal{W}(R)=\Hom_{{\rm cts},\mathcal{O}_E{\rm -alg}}(\Lambda_E(\ZZ_p^\times),R).\] Applying the functor of points idea, we can see $\mathcal{W}$ as the generic fiber of ${\rm Spf}(\Lambda_E(\ZZ_p^\times))$. Note that this is not an affinoid space but it is a rigid analytic space: it is the disjoint union of $p-1$ open balls of radius $1$. Each of these balls are admissibly covered by affinoid disks $W_n$ defined in Section \ref{subsect:Affinoid-space}.

\par For every integer $k$, the character $\kappa_k:z\mapsto z^k$ is an element of the weight space, thus we have an injection $\ZZ\hookrightarrow\mathcal{W}$. Moreover, the integers form a dense subset of the weight space. Thus we can identify $\ZZ$ with the subset of the weight space $\mathcal{W}$ via $\kappa$. 

\par If we fix one of these components (i.e. a character $\chi:(\ZZ/p)^\times\rightarrow R^\times$, without loss of generality we can assume it is the trivial character) and call the component $\mathcal{W}_\chi$, it is an open ball of radius $1$, hence it is admissibly covered by closed affinoid balls of radius $p^{-1/n}$, each of which is isomorphic to \[W_n:=\Spm \left(E\langle X/p^{1/n}\rangle\right).\] Denote \[\mathcal{H}_n:=E\langle X/p^{1/n}\rangle=\Gamma(W_n,\mathcal{O}_{W_n}),\] so that we define the ring of locally analytic $E$-valued distributions on $\Zp$ as \[\mathcal{H}_E(\Zp):=\varprojlim_{n}\mathcal{H}_n=\Gamma(\mathcal{W}_\chi,\mathcal{O}_{\mathcal{W}_\chi}).\]

\par We also define \[\mathcal{H}_E(\ZZ_p^\times)\cong E[\Delta]\otimes\mathcal{H}_E(\ZZ_p),\] and note that $\mathcal{H}_E(\ZZ_p)$ contains the Iwasawa algebra $\Lambda_E(\Zp)\otimes E$. We will denote $\mathcal{H}_E(\Zp)$ by $\mathcal{H}$. If $A$ is an affinoid algebra over $E$, we will write $\mathcal{H}_A:=\mathcal{H}\hatotimes A$ and $\Lambda_A:=\Lambda\hatotimes A$.

If $\Gamma$ is a quotient of $G_\QQ$ or $G_p$ isomorphic to $\Zp$ or $\ZZ_p^\times$, and $\iota:\Gamma\rightarrow\Gamma$ is the involution that maps $\gamma\in\Gamma$ to $\gamma^{-1}$, then $\Lambda(\Gamma)^\iota$ will denote the rank one $\Lambda(\Gamma)$-module on which $G_\QQ$ acts via the inverse of the tautological character, see \cite[\S 8]{Nek06} and \cite[\S 2.2 ]{benois2014selmer} for details. This action naturally extends to $\mathcal{H}_A$, and we denote this rank 1 $\mathcal{H}_A$-module by $\mathcal{H}_A^\iota$.

\begin{definition}[Nice affinoid neighborhood, \cite{bellaiche2012}]\label{def:nice-affinoid-neigh} Let $W=\Spm (R)$ be an admissible affinoid neighborhood of the weight space $\mathcal{W}$, and $R^0$ denote the unit ball for the supremum norm for $R$. Then $W$ is called nice if $R$ and the residue ring $R^0/pR^0$ are PID, and the subset $\ZZ\cap W$ is dense in $W$.
\end{definition}

\par  For example, any closed ball with center $k$ and radius $p^{-n}$ is nice. This will be important when we discuss the perfectness of the Selmer complexes that we are interested in. 

\subsection{Eigencurve and Coleman Families}\label{subsect:eigencurve-coleman-fam}

\par In this subsection we will give the definitions of Coleman families, the big representations attached to these families and state the result that these representations are trianguline (due to Berger). We will also make a quick review of the eigencurve for the sake of completeness, even though we will not need this for this article, we will need this theory for the sequel article.

\par Let $U\subseteq\mathcal{W}$ be a wide open disk centered at a $\kappa(k_0)\in\mathcal{W}$ for some $k_0\geq 0$, i.e. a rigid analytic space $U=B(k_0,r)$ such that its $E$ valued points correspond to $\{x\in\mathcal{W}:\abs{x-\kappa(k_0)}<p^{-r}\}$ for some $r$. Let $\Lambda_U$ denote the power bounded analytic functions on $U$, i.e. \[\Lambda_U:=\{f\in E[[T]]:f\textnormal{ converges on }U,\norm{f}\leq 1\}.\] Note that $\Lambda_U\simeq \mathcal{O}_E[[T]]$, where the isomorphism is non-canonical (hence the notation $\Lambda_U$ is used). For $p\nmid N$, let $\Gamma_p:=\Gamma_1(N)\cap\Gamma_0(p)$. The following definition is from \cite[Definition 4.6.1]{LZ16} (also see \cite[Theorem 3.8]{Ochiai2018IwasawaMC}).

\begin{definition}[Coleman family]\label{def:coleman-family}
    A Coleman Family $\mathcal{F}$ over a wide open $U\subseteq\mathcal{W}$ is a power series \[\mathcal{F}=\sum_{n\geq 1}a_n(\mathcal{F})q^n\in\Lambda_U[[q]],\] where $a_1(\mathcal{F})=1$, and $a_p(\mathcal{F})$ is invertible in $\Lambda_U[1/p]$ such that for all but finitely many weights $k\in U\cap\NN$, the specialization $\mathcal{F}_k\in S_{k+2}(\Gamma_p)$ is a $p$-stabilized and normalized Hecke eigenform, and for all $k\in U\cap\NN$, the slope $v_p(a_p(\mathcal{F}_k))$ is constant.
\end{definition}

\begin{remark}
    Although we gave the definition of a Coleman family over a wide open disk $U$, we can also restrict $\mathcal{F}$ to an affinoid disk $X=\Spm(A)$ inside $U$ such that the same interpolation property holds. In that case, a Coleman family is a power series in $A[[q]]$. In this article we will work with Coleman families (and their associated representations) over an affinoid algebra $A$.
\end{remark}

\begin{theorem}\label{thm:unique-coleman-family}
 Suppose $f_\alpha\in S_{k+2}(\Gamma_p)$ is a Hecke eigenform which is not $\theta$-critical. Then there is a unique Coleman family $\mathcal{F}$ over some affinoid neighborhood $U$ of $k\in\mathcal{W}$ such that $\mathcal{F}_k=f_\alpha$.    
\end{theorem}
\begin{proof}
    See \cite[Theorem 4.6.4]{LZ16}.
\end{proof}

\par Now we give the definition (and the universal property) of the eigencurve, which is invented by Coleman and Mazur in \cite{coleman_mazur_1998} and generalized by Buzzard in \cite{buzzard_2007}. The following definition can be found in \cite[\S 2.1]{bellaiche2012}, also see \cite[Definition 3.5]{Loeffler_2018} and \cite[\S 5]{benoisbuyukboduk2020critical}.

\begin{definition}[Cuspidal Coleman-Mazur-Buzzard eigencurve]
    The cuspidal eigencurve with tame level $N$ $\mathcal{C}^0:=\mathcal{C}^0_N$ is a rigid analytic space of equidimension $1$, which is equipped with a weight morphism $\kappa:=\kappa_\mathcal{C}:\mathcal{C}^0\to\mathcal{W}$ that is locally finite and flat, and a universal eigenform $\mathcal{F}^{\rm univ}\in\mathcal{O}_{\mathcal{C}^0}[[q]]$ such that for all affinoid $U$ with a weight morphism $w:U\to\mathcal{W}$ and any Coleman family $\mathcal{F}$ over $U$, there is a unique morphism $\widetilde{w}:U\to\mathcal{C}^0$ such that $\kappa\circ\widetilde{w}=w$ and $\mathcal{F}=\widetilde{w}^*(\mathcal{F}^{\rm univ})$. 
\end{definition}

\par The eigencurve $\mathcal{C}$ is equipped with a finite flat weight map, which parametrizes modular forms, and contains $\mathcal{C}^0$ as a closed subset. $\mathcal{C}$ is admissibly covered by open affinoid subdomains of the form $\mathcal{C}_{W,\nu}$, where $W=\Spm(A)\subseteq\mathcal{W}$, and $\nu>0$, where $\mathcal{C}_{W,\nu}$ is the maximal spectrum of affinoid subalgebra $\mathcal{H}\otimes R$ of $\End_R(M^\dagger(\Gamma_p,W)^{\leq \nu})$ (and only here $\mathcal{H}$ denotes the Hecke algebra).
    
\par Note that $f_\alpha$ is $\theta$-critical if and only if the eigencurve is not étale over $\mathcal{W}$ at $x\in\mathcal{C}$, where $x$ is the point corresponding to $f_\alpha$ in the eigencurve.

\subsection{Motives and Galois Representations Attached to Modular Forms}\label{subsect:motives-newforms}

\par Let $f\in S_{k+2}(\Gamma_1(N),\varepsilon)$ be a normalised newform of weight $k+2$ and level $\Gamma_1(N)$ with nebentype $\varepsilon$. Suppose that $p\nmid N$. In \cite{scholl}, Scholl constructs a motive $M(f)$ such that its $p$-adic realization is the Deligne's cohomological $G_\QQ$ representation $V_f$, which is irreducible, crystalline at $p$ (since $p\nmid N$) and the action of $\varphi$ on $\mathbf{D}_{\rm cris}(V_f)$ is given by the characteristic polynomial $X^2-a_p(f)X+\varepsilon(p)p^{k+1}$. If $\chi$ is a Dirichlet character, we have \[V_{f_\chi}\cong V_f\otimes\chi.\] 

\par Note that for $g\in\{f_\alpha,f_\beta\}$ we have the isomorphisms of Galois representations $V_{g}\cong V_f$, and $V_{g^c}\cong V_{f^*}$. Note also that due to the Poincaré duality, we have the isomorphism: 
\begin{equation}\label{eq:poincare-duality-reps-forms}
    V_f^*\cong V_f(k+1)\otimes\varepsilon^{-1}\cong V_{f^*}(k+1).
\end{equation}

\par The Hodge-Tate weights of $V_f$ are $0$ and $-1-k$. Note also that the action of geometric Frobenius $\Frob^{-1}_p$ at $p$ on the $\ell$-adic realizations (for $\ell\neq p$) is given by the characteristic polynomial $X^2-a_p(f)X+\varepsilon(p)p^{k+1}$. Let $\alpha$ and $\beta$ denote the roots of this characteristic polynomial, and let $E$ be a finite extension of $\Qp$ which contains $\alpha$, $\beta$, and the Hecke field of $f$.

\par We say that $f$ is ordinary at $p$ if $v_p(a_p(f))=0$, otherwise we say $f$ is $p$-non-ordinary. In this article we are interested in the non-ordinary case. This implies that $v_p(\alpha), v_p(\beta)>0$, therefore none of the $p$-stabilizations of $f$ are at critical slope (i.e. $v_p(\alpha),v_p(\beta)\neq k+1$). Thus they are not $\theta$-critical, i.e. $V_f|_{G_p}$ does not split into the direct sum of two characters. Indeed, in the non-ordinary case, the Panchishkin ordinarity condition for $V_f^*$ fails (even though it satisfies the criticality condition in Definition \ref{def:critical-Deligne-p-adic} since $d^+(V_f^*)=1$ and $V_f^*$ has Hodge-Tate weights $0$ and $k+1$) as $V_f$ is irreducible as a $G_p$ representation. 

\par Let $\mathcal{F}$ be a Coleman family passing through $f_\alpha$ over a nice affinoid subdomain $U=\Spm(A)$ of the weight space. Let $x_0$ be the point corresponding to $f_\alpha$ in the eigencurve. Let $\mathcal{X}$ be the preimage of $U$ under the weight map $\kappa$ of the eigencurve, and let $\mathcal{X}^{\rm cl}$ denote the classical points of $\mathcal{X}$, i.e. the points which are sent to non-negative integers under the weight map $\kappa$ of the eigencurve.

\begin{theorem}[The big Galois representation attached to a Coleman family]\label{thm:rep-coleman-F}
    There exists a free $A$-module $\mathbf{V}:=\mathbf{V}_\mathcal{F}$ of rank $2$ equipped with Galois action $G_{\QQ,S}$ for some finite set of primes $S$ such that for each $x\in\mathcal{X}^{\rm cl}$, the following holds:
    \begin{itemize}
        \item $\mathbf{V}_x$ has Hodge-Tate weights $\{-\kappa(x)-1,0\}$,
        \item $\mathbf{V}_x$ is crystalline (semistable if $p\mid N$ were the case) as a $G_p$-representation,
        \item There is an $\mathbf{\alpha}\in A$ with $\mathbf{\alpha}(x_0)=\alpha$ such that $\mathbf{D}_{\rm cris}(\mathbf{V}_x)^{\varphi=\mathbf{\alpha}(x)}$ is $1$-dimensional,
        \item If $x\neq x_0$, then $\mathbf{D}_{\rm cris} (\mathbf{V}_x)^{\varphi=\mathbf{\alpha}(x)}\cap {\rm Fil}^{\kappa(x)+1}(\mathbf{D}_{\rm st}(\mathbf{V}_x))=0$.
    \end{itemize}
\end{theorem}

\par If the last condition holds also for $x_0$, then $f_\alpha$ is non-$\theta$-critical, otherwise it is $\theta$-critical. In the former case, $A_\mathcal{X}$ and $A$ will be isomorphic as rings, but the situation is more complicated when $f_\alpha$ is $\theta$-critical. See \cite[\S 3.3.2]{benoisbuyukboduk2020critical} for a more detailed explanation. Since we are interested in the non-$\theta$-critical case in this article, we will identify these rings and their spectra inside the weight space. 

\par The following theorem is due to \cite{KedlayaPottharstXiao} (see also \cite[\S 3.3]{benoisbuyukboduk2020critical}).
\begin{theorem}[Triangulation]\label{thm:triangulation}
    In the non $\theta$-critical case, there exists a rank-one $(\phiGamma)$-submodule $\mathbf{D}$ of $\mathbf{D}_{{\rm rig},A}^\dagger(\mathbf{V})$ such that 
    \begin{itemize}
        \item $\mathbf{D}=\mathbf{D}_\delta$ for $\delta:\QQ_p^\times\to\units{A}$ such that $\delta(p)=\alpha$ and $\delta(u)=1$ for $u\in\ZZ_p^\times$,
        \item $\mathscr{D}_{\rm cris}(\mathbf{D}_x)=\mathbf{D}_{\rm cris}(\mathbf{V}_x)^{\varphi=\mathbf{\alpha}(x)}$,
        \item For all classical points $x\in\mathcal{X}^{\rm cl}$ (including $x_0$ in the non-$\theta$-critical case), $\mathbf{D}_x$ is saturated in $\mathbf{D}_{\rm rig}^\dagger(\mathbf{V}_x)$.
    \end{itemize}
\end{theorem}

\section{Selmer Complexes}\label{sect:Selmer-Comp}

\subsection{Review of Galois Cohomology and Iwasawa Cohomology}\label{subsect:galoiscohomology}

\par Let $G$ be a Galois group of a Galois extension, and let $\mathbf{V}$ be a $p$-adic representation of $G$ with coefficients in an affinoid algebra $A$ over $E$. When $A=E$, we will denote $\mathbf{V}$ by $V$. Let $C^\bullet(G,\mathbf{V})$ denote the cochain complex of continuous cohomology of $G$, and let $\RGamma(G,\mathbf{V})$ denote the corresponding object in the derived category $\mathscr{D}(A)$ of complexes of $A$-modules. We will denote the cohomologies of $\RGamma(G,\mathbf{V})$ by $H^i(G,\mathbf{V})$. If $A=E$, and $T\subseteq V$ is a $G$-stable $\mathcal{O}_E$-lattice, we can also define $C^\bullet(G,T)$, $C^\bullet(G,V/T)$ and the corresponding derived objects $\RGamma(G,T)$ and $\RGamma(G,V/T)$ in $\mathscr{D}(\mathcal{O}_E)$ and $\mathscr{D}(E/\mathcal{O}_E)$ respectively in a similar manner. 

\par When $K$ denotes a finite extension of $\QQ$, $\Qp$ or $\Ql$, and $G=G_K$ is the absolute Galois group of $K$, we will denote the complex $C^\bullet(G_K,\mathbf{V})$ by $C^\bullet(K,\mathbf{V})$, the corresponding derived object $\RGamma(G_K,\mathbf{V})$ in $\mathscr{D}(A)$ by $\RGamma(K,\mathbf{V})$, and its cohomologies by $H^i(K,\mathbf{V})$. When $A=E$ and $T\subseteq V$ is a $G$-stable $\mathcal{O}_E$-lattice, we will denote $\RGamma(G_K,T)$ and $\RGamma(G_K,V/T)$ by $\RGamma(K,T)$ and $\RGamma(K,V/T)$, respectively.

\par From now on, let $\mathbf{V}$ be a representation of $G_F$ for a number field $F$ over an affinoid algebra $A$. Let $S$ be a finite set of places of $F$ containing the places $v$ for which $\mathbf{V}$ is ramified, places above $\infty$ and $p$, and let $S_f=S-\{v\mid\infty\}$. If $F_S$ denotes the maximal extension of $F$ unramified outside $S$, the Galois group $\Gal(F_S/F)$ of this extension will be denoted as $G_{F,S}$. Note that $\mathbf{V}$ is a representation of $G_{F,S}$. In this case we denote $C^\bullet(G_{F,S},\mathbf{V})$ by $C^\bullet_S(F,\mathbf{V})$, and the corresponding derived object in $\mathscr{D}(A)$ by $\RGamma_S(F,\mathbf{V})$ or simply by $\RGamma_S(\mathbf{V})$ when $F$ is clear from the context.

\par Let us state the definitions of Iwasawa theoretic versions of the complexes above. Let $\overline{\mathbf{V}}:=\mathbf{V}\otimes_A\mathcal{H}_A^\iota$ denote the cyclotomic deformation of $\mathbf{V}$. Then $\overline{\mathbf{V}}$ is also a representation of $G_{F,S}$ (and of $G_{F_v}$ for each $v\in S_f$), thus we can define the complexes
\begin{align*}
    C^\bullet_{\Iw,S}(F,\mathbf{V})&:=C^\bullet_S(F,\overline{\mathbf{V}}),\\
    C^\bullet_{\Iw}(F_v,\mathbf{V})&:=C^\bullet(F_v,\overline{\mathbf{V}}){\rm\ for\ }v\in S_f.
\end{align*}
We denote the corresponding objects in the derived category $\mathscr{D}(\mathcal{H}_A)$ by $\RGamma_{\Iw,S}(F,\mathbf{V})$ and $\RGammaIw(F_v,\mathbf{V})$, and their cohomologies by $H^i_{\Iw,S}(F,\mathbf{V})$ and $H^i_\Iw(F_v,\mathbf{V})$, respectively. Note that by Shapiro Lemma, when $v\mid p$, we have \[H^i_{\Iw}(F_v,\mathbf{V})\cong\varprojlim_n H^i(F_{v,n},\mathbf{V}),\] where the transition maps are given by corestriction. When $F=\QQ$, we will denote $\RGamma_{\Iw,S}(F,\mathbf{V})$ by $\RGamma_{\Iw,S}(\mathbf{V})$.

\subsubsection{Cohomology of $\left(\varphi,\Gamma\right)$-Modules}\label{subsubsect:cohomology-phigamma}

\begin{definition}[Fontaine-Herr complex]\label{def:Fontaine-Herr}
	Let $K$ be a finite extension of $\Qp$, and let $\mathbf{D}$ be a $(\varphi,\Gamma_K)$-module over $\mathcal{R}_A$. We define the Fontaine-Herr complex $C^{\bullet}_{\varphi,\gamma_n}(\mathbf{D})$, which is concentrated at degrees [0,2], as follows:
    \[
       \left[\mathbf{D}^{\Delta}\xrightarrow{d_0}\mathbf{D}^{\Delta}\oplus\mathbf{D}^{\Delta}\xrightarrow{d_1}\mathbf{D}^{\Delta}\right],
    \]
	where $d_0(z)=(\gamma_n(z) -z,\varphi(z)-z)$ and $d_1(x,y)=(\gamma_n(x) -x)+(y-\varphi(y))$.
\end{definition}
\par We denote the cohomologies of this complex by $H^i(K_n,\mathbf{D})$. For $n=0$, we will denote the complex by $C^\bullet_{\varphi,\Gamma_K}(\mathbf{D})$, and the corresponding object in the derived category $\mathscr{D}(A)$ is denoted by $\RGamma(K,\mathbf{D})$ (or simply by $\RGamma(\mathbf{D})$, when $K=\Qp$).

\begin{definition}[Iwasawa cohomology]\label{def:Iwasawa-Cohomology}
    Let $C^{\bullet}_{\psi}(\mathbf{D})$ be the following complex that is concentrated in degrees $1$ and $2$:
    \[\mathbf{D}^{\Delta}\xrightarrow{\psi -1} \mathbf{D}^{\Delta},\]
	where $\psi$ is the unique continuous map (and a left inverse of $\varphi$) $\psi:\mathcal{R}_A\to\mathcal{R}_A$ such that:
	\[\varphi(\psi(f(\pi))):=\frac{1}{p}\sum_{\xi\in\mu_p}f(\xi(1+\pi)-1).\]
\end{definition}
\par We denote the corresponding object in the derived category $\mathscr{D}(\mathcal{H}_A)$ by $\RGammaIw(K,\mathbf{D})$ (or simply by $\RGammaIw(\mathbf{D})$, when $K=\Qp$), and its cohomologies by $H^i_\Iw(K,\mathbf{D})$.

The following theorem relates the Galois cohomologies of a $G_K$-representation $\mathbf{V}$ and the cohomologies of its associated $(\phiGamma)$-module (see \cite[\S 2.2, \S 2.4]{BB16Exceptional}):
\begin{proposition}\label{prop:Galois-phiGamma-cohomology}
	If $\mathbf{V}$ is a $G_K$ representation over $A$ then we have 
    \begin{align*}
        H^i(K,\mathbf{V})&\cong H^i(K,\mathbf{D}_{{\rm rig},A}^{\dagger}(\mathbf{V})),\\
        H^{i}_{{\rm Iw}}(K,\mathbf{V})&\cong H^{i}_{{\rm Iw}}(K,\mathbf{D}_{{\rm rig},A}^{\dagger}(\mathbf{V})).
    \end{align*}
\end{proposition}

\subsection{Selmer Groups}\label{subsect:selmergrp}

\par In this subsection we set $A=E$. Let $V$ be a $G_F$ representation over $E$, let $T\subseteq V$ be a $G_F$-stable $\mathcal{O}_E$-lattice in $V$. Let $S$ be as in previous subsections. We first define the local conditions attached to $V$, i.e for each place $v\in S_f$, certain subspaces of $H^1(F_v,V)$.

\begin{definition}[Local conditions]\label{def:Local-Conds-Selmer-Grp}
     For each $v\in S_f$ and $v\nmid p$, let \[H^1_f(F_v,V):=\ker(H^1(F_v,V)\rightarrow H^1(F_v^{\rm ur},V)),\] where $F_v^{\rm ur}$ denotes the maximal unramified extension of $F_v$. For $v\mid p$, let \begin{align*}
        H^1_{\rm str}(F_v,V)&:=0,\\
        H^1_{(p)}(F_v,V)&:=H^1(F_v,V),\\
        H^1_f(F_v,V)&:=\ker(H^1(F_v,V)\rightarrow H^1(F_v,V\otimes\mathbf{B}_\cris))
    \end{align*}
    be the strict, relaxed, and Bloch-Kato local conditions at $p$ attached to $V$, respectively.
    For $v\in S_f$, we define $H^1_f(F_v,T)$ as the preimage of $H^1_f(F_v,V)$ under $H^1(F_v,T)\rightarrow H^1(F_v,V)$, and $H^1_f(F_v,V/T)$ as the image of $H^1(F_v,V)$ under $H^1(F_v,V)\rightarrow H^1(F_v,V/T)$. For strict and relaxed conditions at $p$, we define 
    \begin{align*}
        H^1_{\rm str}(F_v,T)&:=0=:H^1_{\rm str}(F_v,V/T),\\
        H^1_{(p)}(F_v,T)&:=H^1(F_v,T),\\
        H^1_{(p)}(F_v,V/T)&:=H^1(F_v,V/T).
    \end{align*}
\end{definition}

\begin{definition}[Selmer group]\label{def:Selmer-Grp}
    Let $M\in\{V,T,V/T\}$. For $?\in\{{\rm str},(p),f\}$, we define respectively the strict, relaxed, Bloch-Kato Selmer groups as follows:
    \[
    H^1_?(F,M):=\ker\left(H^1(F,M)\lra\bigoplus_{\substack{v\in S_f \\ v\nmid p }}\frac{H^1(F_v,M)}{H^1_f(F_v,M)}\oplus\bigoplus_{\substack{v\in S_f \\ v\mid p }}\frac{H^1(F_v,M)}{H^1_?(F_v,M)}\right).
    \]
\end{definition}

\par Let $D$ be a $(\varphi,\Gamma_K)$-module over the Robba ring $\mathcal{R}_E$. Let $H^1_f(K,D)\subseteq H^1(K,D)$ be the subspace of the crystalline extensions of $D$. Then $H^1_f(K,D)$ sits in the following exact sequence (see \cite[Proposition 2.9.2]{benois-p-adic-heights}):
\[
  0\rightarrow H^0(K,D)\rightarrow \mathscr{D}_\cris(D)\xrightarrow{1-\varphi,{\rm pr}}  \mathscr{D}_\cris(D)\oplus t_D(K)\rightarrow H^1_f(K,D)\rightarrow 0,
\]
where $t_D(K):=\mathscr{D}_{\dR}(D)/\Fil^0\mathscr{D}_\dR(D)$ is called the tangent space of $D$. When $V$ is potentially semistable, then we have \[H^1_f(K,V)=H^1_f(K,\mathbf{D}_{\rm rig}^\dagger(V)).\]

\subsection{Selmer Complexes}\label{subsect:nekovarselmercomp}

\par In this subsection we will give a brief summary of Nekovář's Selmer Complexes, details of which can be found in \cite{Nek06}. We will here focus on the rings $R$ which are complete local Noetherian rings with finite residue field, and give the generalizations to affinoid algebras in the following subsection. 

\par Let $F$ be a finite extension of $\QQ$, and let $S$ be a finite set of primes as before.

\begin{definition}[Selmer complex]\label{def:Selmer-Comp}
    Let $X$ be an admissible $R[G_F]$-module (\cite[Definition 3.2.1]{Nek06}), and for each $v\in S$, let $U_v^+(X)$ be a complex satisfying some finiteness conditions (see \cite[\S 4.2]{Nek06} for the details), and let $i_v^+:U_v^+(X)\rightarrow C^\bullet(G_v,X)$ be a morphism of complexes. The local conditions attached to $X$ is the collection of these morphisms $\Delta(X):=\{i_v^+\}$. Let $U_S^+(X):=\bigoplus_{v\in S_f}U_v^+(X)$.

    The Selmer complex attached to the data $(G_F,X,\Delta(X))$ is the following mapping cone:
    \[
    \widetilde{C}_f^{\bullet}(F,X,\Delta):= {\rm Cone}\left(C^\bullet_S(X)\oplus U_{S}^+(X) \xrightarrow{res_{S_f}-i_S^+} \bigoplus_{v\in S_f} C^\bullet(G_v,X)\right)[-1],
    \] where $i_S^+=\oplus i_v^+:U_{S_f}^+(X)\to\bigoplus_{v\in S_f} C^\bullet(G_v,X)$. The corresponding object in the derived category $\mathscr{D}_{\rm ft}^b(R)$ will be denoted by $\widetilde{\RGamma}_f(F,X,\Delta)$, or simply by omitting $F$ and $\Delta$ when they are clear from the context.
\end{definition}

\par We set \[U_S^-(F,X):={\rm Cone}\left(U_S^+(F,X)\to\bigoplus_{v\in S_f}C^\bullet(F_v,X)\right),\] and \[C^\bullet_{\rm c}(G_{F,S},X):={\rm Cone}\left(C^\bullet_S(F,X)\to \bigoplus_{v\in S_f}C^\bullet(F_v,X)\right)[-1].\]
	Then by definition of these complexes, we have the following exact triangles (see \cite[6.1.3.2]{Nek06}):
    \begin{align}
        U_S^+(F,X)\rightarrow \bigoplus_{v\in S_f} \RGamma(F_v,X) \rightarrow U_S^-(F,X) \rightarrow U_S^+(F,X)[1],\label{eq:Tr1}\tag{Tr1} \\ 
        U_S^+(F,X)[-1]\rightarrow \RGamma_{\rm c}(G_{F,S},X) \rightarrow \widetilde{\RGamma}_f(F,X) \rightarrow U_S^+(F,X), \label{eq:Tr2}\tag{Tr2} \\ 
        U_S^-(F,X)[-1] \rightarrow  \widetilde{\RGamma}_f(F,X) \rightarrow \RGamma_S(F,X) \rightarrow U_S^-(F,X). \label{eq:Tr3}\tag{Tr3}
    \end{align}

\par In particular, we will focus on the case where the local conditions (called Greenberg's local conditions) are as follows: for $v\nmid p$ we will set \[U_v^+(X):=C^{\bullet}(G_v/I_v,X^{I_v}),\] and for $v\mid p$ we choose a certain subrepresentation $F^+_vX\subseteq X$ and set $U_v^+(X):=C^\bullet(G_v,F^+_vX)$.

\par In the case where $F^+_vX=X$ (resp. $F^+_vX=0$) for all $v\mid p$ we will denote the associated Selmer complex by $\widetilde{C}^\bullet_{(p)}(X)$ (resp. $\widetilde{C}^\bullet_{\rm str}(X)$).

\begin{theorem}[Exact triangle of Selmer complexes]\label{thm:exact-triangle-nek-selmer-comp}
    Let $X_1$, $X_2$, $X_3$ be admissible $G_F$-representations forming a split exact sequence as follows:
    \begin{align}
        0\longrightarrow X_1\xlongrightarrow{f} X_2\xlongrightarrow{g} X_3 \longrightarrow 0.\label{eq:split-exact-reps-nek}
    \end{align}

    \par Suppose also that for each $v\mid p$, this exact sequence induce the following exact sequences:

    \begin{align}
        0\longrightarrow U_v^+(X_1)\xrightarrow{U_v^+(f)}U_v^+(X_2)\xrightarrow{U_v^+(g)} U_v^+(X_3) \longrightarrow 0.\label{eq:localcond-exact-reps}
    \end{align}

    Then we have the short exact sequence of the Selmer complexes:
    \begin{align*}
        0\longrightarrow\widetilde{C}^\bullet_f(F,X_1,\Delta(X_1))\longrightarrow \widetilde{C}^\bullet_f(F,X_2,\Delta(X_2))\longrightarrow \widetilde{C}^\bullet_f(F,X_3,\Delta(X_3))\longrightarrow 0.
    \end{align*}
    Moreover, if the splitting of \eqref{eq:split-exact-reps-nek} induces the splitting of \eqref{eq:localcond-exact-reps}, then the coboundary maps in the long exact sequence coming from the exact triangle above are $0$.
\end{theorem}
\begin{proof}
    We will follow the strategy that Nekovář used in the beginning of Chapter 11 of \cite{Nek06}, which lead to the definition of $p$-adic height pairing. First note that the splitness of the exact sequence in \eqref{eq:split-exact-reps-nek} yield the exact sequence \eqref{eq:localcond-exact-reps} also for $v\nmid p$. For each $v\in S$, let $U_v^-(X):={\rm Cone}(-i_v^+:U_v^+(X)\rightarrow C^\bullet(G_v,X))$ as in \cite{Nek06}. 
    
    Then we have a short exact sequence  
    \begin{align*}
        0\longrightarrow U_v^-(X_1)\xrightarrow{U_v^-(f)}U_v^-(X_2)\xrightarrow{U_v^-(g)} U_v^-(X_3) \longrightarrow 0
    \end{align*}
    obtained by applying 3x3 lemma to the following commutative diagram with exact columns, top and second rows:
    \[\begin{tikzcd}
        & 0 \arrow[d] & 0 \arrow[d] & 0 \arrow[d] & \\
        0 \arrow[r] & U_v^+(X_1) \arrow[r,"U_v^+(f)"]\arrow[d,"i_v^+"] & U_v^+(X_2) \arrow[r,"U_v^+(g)"]\arrow[d,"i_v^+"] & U_v^+(X_3) \arrow[r]\arrow[d,"i_v^+"] & 0\\
        0 \arrow[r] & C^\bullet(G_v,X_1) \arrow[r]\arrow[d,"{\rm Cone}(i_v^+)"] & C^\bullet(G_v,X_2) \arrow[r]\arrow[d,"{\rm Cone}(i_v^+)"] & C^\bullet(G_v,X_3) \arrow[r]\arrow[d,"{\rm Cone}(i_v^+)"] & 0\\
        0 \arrow[r,dotted] & U_v^-(X_1) \arrow[r,dotted] \arrow[d] & U_v^-(X_2) \arrow[r,dotted]  \arrow[d] & U_v^-(X_3) \arrow[r,dotted]  \arrow[d] & 0\\
        & 0 & 0  & 0 & \\
    \end{tikzcd}\]
    Hence we have the following exact sequence:
    \begin{align*}
        0\longrightarrow U_S^-(X_1)\xrightarrow{U_S^-(f)}U_S^-(X_2)\xrightarrow{U_S^-(g)} U_S^-(X_3) \longrightarrow 0.
    \end{align*}
    Recall the following exact sequence from \cite[6.1.3.1]{Nek06} inducing \eqref{eq:Tr3}:
    \begin{align*}
        0\longrightarrow U_S^-(X)[-1]\longrightarrow \widetilde{C}^\bullet_f(F,X,\Delta)\longrightarrow C^\bullet_S(X)\longrightarrow 0.
    \end{align*}
    Combining these and applying 3x3 lemma to the following diagram yields the desired exact sequence:
    \[\begin{tikzcd}
        & 0 \arrow[d] & 0 \arrow[d] & 0 \arrow[d] & \\
        0 \arrow[r] & U_S^-(X_1)[-1] \arrow[r,"U_S^-(f)"]\arrow[d] & U_S^-(X_2)[-1]  \arrow[r,"U_S^-(g)"]\arrow[d] & U_S^-(X_3)[-1]  \arrow[r]\arrow[d] & 0\\
        0 \arrow[r,dotted] & \widetilde{C}^\bullet_f(X_1) \arrow[r,dotted]\arrow[d] & \widetilde{C}^\bullet_f(X_2) \arrow[r,dotted]\arrow[d] & \widetilde{C}^\bullet_f(X_3) \arrow[r,dotted]\arrow[d] & 0\\
        0 \arrow[r] & C^\bullet_S(X_1) \arrow[r] \arrow[d] & C^\bullet_S(X_2) \arrow[r] \arrow[d] & C^\bullet_S(X_3) \arrow[r] \arrow[d] & 0\\
        & 0 & 0  & 0 & \\
    \end{tikzcd}\]
    Finally, for $v\mid p$ if the exact sequence \eqref{eq:localcond-exact-reps} also splits, then we have $U_S^-(X_2)=U_S^-(X_1)\oplus U_S^-(X_3)$, hence the first and third rows in the 3x3 diagram above split. Hence the second row induced by the 3x3 lemma also becomes split exact, and the coboundary map in the long exact sequence becomes $0$.
\end{proof}

\par Selmer complexes satisfy some duality theorems (via Grothendieck and Matlis dualities), control theorems, etc. We will be mainly interested in Pottharst style Selmer complexes, which generalizes the construction of Nekovář to the local conditions at $p$ arising from $(\phiGamma)$-modules and representations with coefficients in affinoid algebras. Thus we will state these theorems in terms of Pottharst style Selmer complexes in the following subsections.

\subsection{Pottharst Style Selmer Complexes}\label{subsect:pottharstselmercomp}

\par In this subsection, we will review results from \cite{pottharst2012cyclotomic}, \cite{Pottharst2013}, \cite{benois2014selmer}, \cite{BB16Exceptional} and \cite{benois-p-adic-heights}. Pottharst's Selmer Complexes are a direct generalization of Selmer Complexes introduced by Nekovář in \cite{Nek06} in a way that we are able to work with local conditions coming from triangulations and to work with coefficients in an affinoid algebra. In this subsection, for notational simplicity and our purposes, we will take $F=\QQ$ and consider $p$-adic representations $\mathbf{V}$ of $G_{\QQ,S}$ with coefficients in an affinoid algebra $A$ over $E$ (including the case $A=E$).

\par Let $\mathbf{D}\subseteq \mathbf{D}_{{\rm rig},A}^\dagger(\mathbf{V})$ be a $(\phiGamma)$-submodule.
Let \[U_p^+(\mathbf{V},\mathbf{D}):=C^\bullet_{\varphi,\Gamma}(\mathbf{D})\] denote the local condition at $p$, which is the Fontaine-Herr complex defined in Definition \ref{def:Fontaine-Herr}. If $\ell\neq p$, then $U_\ell^+(\mathbf{V},\mathbf{D})$ will denote the unramified local conditions 
\[U_\ell^+(\mathbf{V},\mathbf{D}):=\left[\mathbf{V}^{I_\ell}\xrightarrow{f_\ell-1}\mathbf{V}^{I_\ell}\right],\] which is a complex concentrated at degrees $[0,1]$.

Let $K_p(\mathbf{V})$ denote the total complex \[{\rm Tot}\left(C^\bullet(\Qp,\mathbf{D}_{{\rm rig},A}^\dagger(\mathbf{V}))\xrightarrow{\varphi -1}(C^\bullet(\Qp,\mathbf{D}_{{\rm rig},A}^\dagger(\mathbf{V})\right),\] and \[i_p^+:U_p^+(\mathbf{V},\mathbf{D})\to C_{\varphi,\Gamma}^\bullet(\mathbf{D}_{{\rm rig},A}^\dagger(\mathbf{V}))\xrightarrow{q.is}K_p(V).\] We finally set \[K_\ell(\mathbf{V}):=\RGamma(\Ql,\mathbf{V})\] and \[K_S(\mathbf{V}):=\bigoplus_{v\in S_f}K_v(\mathbf{V}),\] and also denote the corresponding objects to these complexes in the derived category $\mathscr{D}(A)$ exactly the same.

\begin{definition}\label{def:Selmer-Comp-Benois}
	The Selmer complex $\RGamma(\mathbf{V},\mathbf{D})$ attached to the data $(\mathbf{V},\mathbf{D})$ is defined as the mapping cone:
	\begin{align*}
		{\rm Cone}\left(\RGamma_S(\mathbf{V})\oplus\bigoplus_{v\in S_f}U_v^+(\mathbf{V},\mathbf{D})\to K_S(\mathbf{V})\right)[-1].
	\end{align*}
    If we work with representations of $G_F$ for general number fields $F$, and for each prime $\frakp$ of $F$ lying above $p$, if we have a local condition $\mathbf{D}_\frakp$, and $\mathbf{D}=(\mathbf{D}_\frakp)_\frakp$, then we define the corresponding Selmer complexes in a similar manner, and denote them by $\RGamma(F,\mathbf{V},\mathbf{D})$.
\end{definition}

\par Let $\widetilde{\mathbf{D}}:=\mathbf{D}_{{\rm rig},A}^\dagger(\mathbf{V})/\mathbf{D}$, and for $v\in S_f$, let \[
    U_v^-(\mathbf{V},\mathbf{D}):={\rm Cone}\left(U_v^+(\mathbf{V},\mathbf{D})\xrightarrow{-i_v^+} K_v(\mathbf{V})\right).
    \]
Then we have $U_p^-(\mathbf{V},\mathbf{D})=\RGamma(\widetilde{\mathbf{D}})$ and the following commutative diagram of exact triangles with vertical arrows being quasi-isomorphisms in the derived category $\mathscr{D}(A)$:
\[\begin{tikzcd}
	U_p^+(\mathbf{V},\mathbf{D}) \arrow[r] & K_p(\mathbf{V}) \arrow[r] & U_p^-(\mathbf{V},\mathbf{D}) \arrow[r] & U_p^+(\mathbf{V},\mathbf{D})[1]\\
	\RGamma(\mathbf{D}) \arrow[r]\arrow[u,"\sim"] & \RGamma(\mathbf{D}_{{\rm rig},A}^\dagger(\mathbf{V})) \arrow[r]\arrow[u,"\sim"] & \RGamma(\widetilde{\mathbf{D}}) \arrow[r]\arrow[u,dotted,"\sim"] & \RGamma(\mathbf{D})[1]\arrow[u,"\sim"]
\end{tikzcd}\]
Thus we have $H^i(\Qp,\widetilde{\mathbf{D}})\cong H^i(U_p^-(\mathbf{V},\mathbf{D}))$.

\par For the Iwasawa theoretic analogues of these Selmer complex, we define our local conditions as follows: First, let us look at the case $A=E$. Let $T$ be a $G_\QQ$-stable $\mathcal{O}_E$-lattice in $V$.
For any $\ell\neq p$, we define $U_{\Iw,\ell}^+(V,\mathbf{D})$ as
\begin{align*}
	U_{\Iw,\ell}^+(V,\mathbf{D})=\RGamma_{f,\Iw}(\Ql,T)\otimes_\Lambda\mathcal{H},
\end{align*}
where $\mathbf{R\Gamma}_{f,\Iw}(\Ql,T)$ is the corresponding derived object of the following complex concentrated at degrees $[0,1]$:
\begin{align*}
	\left[T^{I_\ell}\otimes\Lambda^\iota\xrightarrow{{\rm Fr}_\ell -1}T^{I_\ell}\otimes\Lambda^\iota\right].
\end{align*}
For $v=p$, we define $U^+_{\Iw,p}(V,\mathbf{D})$ as $\RGammaIw(\mathbf{D})$. We define $U^+_{\Iw,S}(V,\mathbf{D})$ as before, and for each $v\in S_f$, $K_{\Iw,v}(V):=K_v(T\otimes_\Lambda\mathcal{H}^\iota)\simeq\RGammaIw(\QQ_v,V)$
Then we can define the Iwasawa theoretic Selmer complex $\RGammaIw(V,\mathbf{D})$ associated to this local data as follows:
\begin{align*}
	{\rm Cone}\left( \left(\RGamma_{\Iw,S}(T)\otimes_\Lambda^\mathbf{L}\mathcal{H}\right)\oplus U_{{\rm Iw},S}^+(V,\mathbf{D})\to\bigoplus_{v\in S_f}\RGamma_\Iw(\QQ_v,V)\right)[-1].
\end{align*}

\par Note that we could simply define the Iwasawa theoretic Selmer complex as \[\RGammaIw(V,\mathbf{D}):=\RGamma(\overline{V},\overline{\mathbf{D}}),\] where we recall that $\overline{V}=V\otimes\mathcal{H}^\iota$ and $\overline{\mathbf{D}}=\mathbf{D}\otimes\mathcal{H}^\iota$.

\par If $\mathbf{V}$ is a $p$-adic representation over an affinoid algebra $A$, we replace $T$ in the definition above by $\mathbf{V}$. In this case, the Iwasawa Selmer complex we are interested in is the Selmer Complex $\RGamma(\overline{\mathbf{V}},\overline{\mathbf{D}})$, where \[\overline{\mathbf{V}}=\mathbf{V}{\otimes}_A\mathcal{H}_A^\iota\] and \[\overline{\mathbf{D}}:=\mathbf{D}{\otimes}_A\mathcal{H}_A^\iota\] are the cyclotomic deformations of $\mathbf{V}$ and $\mathbf{D}$, respectively. More explicitly, our local conditions are given as follows: For $\ell\neq p$ we have 
\begin{align*}
    U_{\ell,\Iw}^+(\mathbf{V},\mathbf{D})=U_\ell^+(\overline{\mathbf{V}},\overline{\mathbf{D}}):=\left[(\mathbf{V}\otimes_A\Lambda_A^\iota)^{I_\ell}\xrightarrow{\rm{Fr}_\ell -1}(\mathbf{V}\otimes_A\Lambda_A^\iota)^{I_\ell}\right]\otimes^\mathbf{L}_{\Lambda_A}\mathcal{H}_A,
\end{align*}
and at $p$ we have
\begin{align*}
    U_{p,\Iw}^+(\mathbf{V},\mathbf{D})=U_p^+(\overline{\mathbf{V}},\overline{\mathbf{D}}):=\left[\mathbf{D}^{\Delta}\xrightarrow{\psi-1}\mathbf{D}^{\Delta}\right].
\end{align*}
\par We will also denote this Selmer complex by $\RGamma_{\rm Iw}(\mathbf{V},\mathbf{D})$. Now we will state some useful properties of these Selmer complexes.

\begin{proposition}[Twisting property]\label{prop:even_cyclotomic_twist}
    Let $\chi:=\left\langle\cyc\right\rangle$ be the restriction of the cyclotomic character to $\Gamma_0$, i.e. $\left\langle\cyc\right\rangle=\cyc\omega^{-1}$. Then we have 
    \[
    \RGammaIw(\mathbf{V}(\chi),\mathbf{D}(\chi))\simeq\RGammaIw(\mathbf{V},\mathbf{D})(\chi). 
    \]
\end{proposition}
\begin{proof}
    \par By the definition of Selmer complexes, it is sufficient to show that the twisting property holds for $\RGamma_{\Iw, S}(\mathbf{V})$, $\RGammaIw(\QQ_v,\mathbf{V})$ and $U^+_{\Iw,v}(\mathbf{V},\mathbf{D})$ for each $v\in S_f$. If $\ell\neq p$, then \[U_{\Iw,\ell}^+(\mathbf{V}(\chi),\mathbf{D}(\chi))=U_{\Iw,\ell}^+(\mathbf{V},\mathbf{D})(\chi)\] since $\chi$ is unramified for all primes $\ell\neq p$. Same argument goes for $\RGammaIw(\Ql,V)$.
    \par For the prime $p$, recall the definition of \[U_{\Iw,p}^+(\mathbf{V},\mathbf{D})=\RGammaIw(\mathbf{D}),\] and \[\RGammaIw(\Qp,\mathbf{V})=\RGammaIw(\mathbf{D}_{{\rm rig},A}^\dagger(\mathbf{V})).\] Since $\Delta$ acts trivially on $\chi$, we can take $\chi$ outside of these complexes. 
    \par Finally for $\RGamma_{\Iw,S}(\mathbf{V})$, we will refer to \cite[\S VI Proposition 2.1(i)]{Ru00}. 
\end{proof}

\begin{theorem}\label{thm:exact-triangle-selmer-comp}
Let $\mathbf{V}_1$, $\mathbf{V}_2$, $\mathbf{V}_3$ be finite free $G_K$-representations with coefficients in an affinoid algebra $A$ over a $p$-adic field $E$ forming a split exact sequence as follows:
\begin{align}
    0\longrightarrow \mathbf{V}_1\xlongrightarrow{f} \mathbf{V}_2\xlongrightarrow{g} \mathbf{V}_3 \longrightarrow 0.\label{eq:split-exact-reps}
\end{align}
For $i\in\{1,2,3\}$, let $\mathbf{D}_i$ be a $(\phiGamma)$-submodule of $\mathbf{D}_{{\rm rig},A}^\dagger(\mathbf{V}_i)$ such that $\mathbf{D}_1$ is the preimage of $\mathbf{D}_2$ under $\mathbf{D}_{{\rm rig},A}^\dagger(f)$, and $\mathbf{D}_3$ is the image of $\mathbf{D}_2$ under $\mathbf{D}_{{\rm rig},A}^\dagger(g)$. In other words suppose there is an exact sequence of the following form:
\begin{align}
    0\longrightarrow \mathbf{D}_1\xrightarrow{\mathbf{D}_{{\rm rig},A}^\dagger(f)} \mathbf{D}_2\xrightarrow{\mathbf{D}_{{\rm rig},A}^\dagger(g)} \mathbf{D}_3 \longrightarrow 0.\label{eq:phigamma-exact-reps}
\end{align}
Then we have the exact triangles of the Selmer complexes:
\begin{align}
    \RGamma_{?}(\mathbf{V}_1,\mathbf{D}_1)\longrightarrow &\RGamma_{?}(\mathbf{V}_2,\mathbf{D}_2)\longrightarrow \RGamma_{?}(\mathbf{V}_3,\mathbf{D}_3),
\end{align}
where $?\in\{\emptyset,{\Iw}\}$.
Moreover, if the splitting of \eqref{eq:split-exact-reps} induces the splitting of \eqref{eq:phigamma-exact-reps}, then the coboundary maps in the long exact sequences coming from the exact triangles above are $0$.
\end{theorem}
\begin{proof}
This is essentially Theorem \ref{thm:exact-triangle-nek-selmer-comp} adapted to this setting.
\end{proof}

\begin{proposition}[Control theorems]\label{prop:control-theorems-pottharst} 
    Let $\mathbf{V}$ and $\mathbf{D}$ be as before, $x\in\Spm(A)$ with corresponding maximal ideal $\frakm_x$ of $A$. Suppose that $A/\frakm_x\cong E$. Let ${{\rm sp}_{x,\Iw}}:\mathcal{H}_A\rightarrow\mathcal{H}$ be induced by the specialization map ${\rm sp}_x:A\rightarrow A/\frakm_x$, and let $V_x:=\mathbf{V}\otimes_A A/\mathfrak{m}_x$. Finally for an affinoid algebra $B$, let $f:A\rightarrow B$ a morphism of affinoid algebras. Then we have:
    \begin{align}
        \RGammaIw(\mathbf{V},\mathbf{D})&\otimes_{\mathcal{H}_A}^{\mathbf{L}}A\simeq \RGamma(\mathbf{V},\mathbf{D}),\label{eq:control-HA-A}\tag{CTRL-1}\\
        \RGammaIw(\mathbf{V},\mathbf{D})&\otimes_{{\rm sp}_{x,\Iw}}^{\mathbf{L}}\mathcal{H}\simeq \RGammaIw(V_x,\mathbf{D}_x),\label{eq:control-HA-H}\tag{CTRL-2}\\
        \RGamma(\mathbf{V},\mathbf{D})&\otimes_A^{\mathbf{L}}{A/\mathfrak{m}_x}\simeq \RGamma(V_x,\mathbf{D}_x),\label{eq:control-A-E}\tag{CTRL-3}\\
        \RGamma(\mathbf{V},\mathbf{D})&\otimes_f^{\mathbf{L}}B\simeq \RGamma(\mathbf{V}\otimes_f B,\mathbf{D}\otimes_f B)\label{eq:control-A-B}\tag{CTRL-4}.
    \end{align}
\end{proposition}
\begin{proof}
    All of these situations are a special case of \cite[Theorem 1.12 - Situation (3)]{Pottharst2013}. Here it is crucial that the $p$-adic open unit ball $W:=\mathcal{W}_\chi$ where $\chi=\omega^j$ for some $j\in(\ZZ/p\ZZ)^\times$ (recall Section \ref{subsect:Affinoid-space}) is admissibly covered by affinoid subdomains (since $\mathcal{H}=\Gamma(W,\mathcal{O}_W)$).
\end{proof}

\par Now we state a general theorem on spectral sequences, which will help us to compute the cohomologies above.

\begin{proposition}[Tor spectral sequence]\label{prop:Tor-spectral-seq}
    Let $R$ be a ring, $M$ be an $R$-module, and $K^\bullet$ be a bounded cochain complex of $R$-modules. Then there exists a page-2 spectral sequence
    \[
    E_2^{i,j}:\Tor_i^R(H^{-j}(K^\bullet),M)\Rightarrow H^{-i-j}(K^\bullet\otimes_R^\mathbf{L}M),
    \]
    where $d_2^{i,j}:E_2^{i+2,j-1}\rightarrow E_2^{i,j}$. If $\Tor_i^R(H^{-j}(K^\bullet),M)=0$ for $i\geq 2$, then this spectral sequence degenerates to the following exact sequences for each $n\in\ZZ$:
    \[0\rightarrow H^n(K^\bullet)\otimes_R M \rightarrow H^n(K^\bullet\otimes_R^\mathbf{L}M) \rightarrow \Tor_1^R(K^\bullet,M) \rightarrow 0,\] i.e.
\end{proposition}
\begin{proof}
    See \cite[\href{https://stacks.math.columbia.edu/tag/061Z}{Example 061Z}]{stacks_project}. Note that in Stacks, homological notation is used and the roles of $i$ and $j$ were swapped, hence our differential maps $d_r$ are of degree $(-r,r-1)$ contrary to the case in Stacks. Also note that the actual result is on computing the homologies of of a chain complex $K_\bullet$, here we used the re-indexing convention by setting $K^i:=K_{-i}$.
    When $\Tor_i(K^\bullet,M)=0$ for $i\geq 2$, then we have $E_\infty^{i,j}=E_2^{i,j}$, therefore we get the short exact sequence above for each $n\in\ZZ$.
\end{proof}

\par We will now focus on the case $A=E$ and state structure theorem of coadmissible $\mathcal{H}$-modules (in the sense of \cite{pottharst2012cyclotomic} and \cite{Schneider-Teitelbaum}), and the duality theorems. Recall that a coadmissible $\mathcal{H}$-module $M$ is the inverse limit of finitely generated $\mathcal{H}_n$-modules $M_n$ such that $M_n\otimes_{\mathcal{H}_n}\mathcal{H}_{n-1}\cong M_{n-1}$, where the isomorphism is induced by the transition maps. According to \cite[Remark 1.2]{pottharst2012cyclotomic}, conjecturally these theorems hold also for coadmissible $\mathcal{H}_A$-modules. For our purposes, these versions are sufficient thanks to the control theorems.

\begin{proposition}[Structure theorem, \cite{pottharst2012cyclotomic} Proposition 1.1, also \cite{benois2014selmer} Proposition 3]\label{structure-thm-pottharst}
Let $M$ be a coadmissible $\mathcal{H}$-module. Then the torsion submodule $M_{\rm tors}$ of $M$ is coadmissible, $M/M_{\rm tors}$ is finitely generated free $\mathcal{H}$-module. Moreover, $M_{\rm tors}$ is isomorphic to $\prod_{\alpha\in I}\mathcal{H}/{\mathfrak{p}_\alpha^{n_\alpha}}$, where each $\mathfrak{p}_\alpha$ is a prime in $\mathcal{H}_n[1/p]$ for some $n\in\NN$, and for each $n\in\NN$ there are finitely many prime ideals $\mathfrak{p}_\alpha$ with $n_\alpha=0$.
\end{proposition}

\par Let $\mathcal{K}$ be the field of fractions of $\mathcal{H}$, and let \[\omega^\bullet:={\rm Cone}\left[\mathcal{K}\rightarrow\mathcal{K}/\mathcal{H}\right][-1]\] be the dualizing complex for $\mathcal{H}$. For a coadmissible $\mathcal{H}$-module $M$, let \[\mathscr{D}(M):=\Hom_\mathcal{H}([M],\omega^\bullet),\] and let $\mathscr{D}^i(M)$ denote $H^i(\mathscr{D}(M))$. Then we have the following theorem (see \cite[\S 1]{pottharst2012cyclotomic}, also \cite[\S 1.6]{benois2014selmer})

\begin{proposition}[Dualizing complex]\label{prop:dualizing-comp-pottharst} 
    We have
    \begin{align*}
        \mathscr{D}^0(M)={\rm Hom}_\mathcal{H}(M/M_{\rm tors},\mathcal{H}),\\
        \mathscr{D}^1(M)=\mathscr{D}^1(V_{\rm tors})=\prod_{\alpha\in I}{\mathfrak{p}_\alpha^{-n_\alpha}}/\mathcal{H},
    \end{align*}
    and $\mathscr{D}^i(M)=0$ for $i\neq 0,1$.
\end{proposition}

\begin{proposition}[Grothendieck duality]\label{prop:Pottharst-Duality-Iwasawa-SelmerComp}
    Let $V$ be a $p$-adic representation over $E$, for a $(\phiGamma)$-module $\mathbf{D}\subseteq\mathbf{D}_{\rm rig}^\dagger(V)$, let $\mathbf{D}^\perp:=\Hom_{\mathcal{R}_E}(\mathbf{D}_{\rm rig}^\dagger(V)/\mathbf{D},\mathcal{R}_E(1))$. Then we have 
\begin{align}
    \mathscr{D}\RGammaIw(V,\mathbf{D})\simeq \RGammaIw(V^*(1),\mathbf{D}^\perp)^\iota[3],
\end{align}
and we get the following (split) exact sequence:
\begin{align*}
    0\rightarrow \mathscr{D}^1H^{4-i}_{\Iw}(V,\mathbf{D}) \rightarrow H^i_{\Iw}(V^*(1),\mathbf{D}^\perp)^\iota \rightarrow \mathscr{D}^0H^{3-i}_{\Iw}(V,\mathbf{D}) \rightarrow 0.
\end{align*}
\end{proposition}
\begin{proof}
    See \cite[Theorem 4.1]{pottharst2012cyclotomic}.
\end{proof}

\par We also give the analogous duality statement for Selmer complexes of representations over affinoid algebras (see \cite[Theorem 3.1.9]{benois-p-adic-heights} and \cite[Theorem 1.16]{Pottharst2013}):

\begin{proposition}[Duality over affinoids]\label{prop:Duality-Affinoid-SelmerComp}
    Let $\mathbf{V}$ be a $p$-adic representation over an affinoid algebra $A$, for a $(\phiGamma)$-module $\mathbf{D}\subseteq\mathbf{D}_{{\rm rig},A}^\dagger(\mathbf{V})$, let $\mathbf{D}^\perp:=\Hom_{\mathcal{R}_A}(\mathbf{D}_{{\rm rig},A}^\dagger(\mathbf{V})/\mathbf{D},\mathcal{R}_A(1))$. Let $\mathscr{D}_{\rm parf}^{[a,b]}(A)$ denote the category of perfect complexes of $A$-modules with perfect amplitude $[a,b]$ (see next section for details).
    Suppose that the local conditions attached to $\RGamma(\mathbf{V},\mathbf{D})$ and $\RGamma(\mathbf{V}^*(1),\mathbf{D}^\perp)$ lie in the category $\mathscr{D}_{\rm parf}^{[0,2]}(A)$. Then these complexes lie in $\mathscr{D}_{\rm parf}^{[0,3]}(A)$, and we have the following duality in $\mathscr{D}_{\rm parf}^{[0,3]}(A)$:
\begin{align}
    \mathbf{R}{\rm Hom}_A(\RGamma(\mathbf{V},\mathbf{D}),A)\simeq \RGamma(\mathbf{V}^*(1),\mathbf{D}^\perp)[3].
\end{align}
\end{proposition}

\begin{corollary}[Pottharst]\label{cor:pottharst-proposition-hi-ranks}
For $i\neq 1,2,3$ we have \[H^i_{\Iw}(V,\mathbf{D})=0,\] and \[H^3_{\Iw}(V,\mathbf{D})=(T^*(1)^{\Gal(\QQ_S/\QQ_\infty)})^*\otimes_\Lambda\mathcal{H}.\] When the local conditions are (strict) ordinary (in the sense of \cite{pottharst2012cyclotomic}), we have \[{\rm rank}_\mathcal{H}H^1_\Iw(V,\mathbf{D})={\rm rank}_\mathcal{H}H^2_\Iw(V,\mathbf{D}).\]
\end{corollary}
\begin{proof}
See \cite[Corollary 4.2.]{pottharst2012cyclotomic} and \cite[Proposition 6]{benois2014selmer}
\end{proof}

\begin{corollary}\label{cor:vanishing_of_H3}
Let $V$ be an absolutely irreducible representation of $G_{\QQ,S}=\Gal(\QQ_S/\QQ)$ over $E$, and set $H_{\QQ,S}:=\Gal(\QQ_S/\QQ_\infty)$, so that $\Gamma_0=G_{\QQ,S}/H_{\QQ,S}$. If $V^*(1)$ is not a character of $\Gamma_0$, then $H^3_\Iw(V,\mathbf{D})=0$.
\end{corollary}
\begin{proof}
We will show that $(V^*(1))^H_{\QQ,S}=0$. First note that $V^*(1)$ is also absolutely irreducible. Since $H_{\QQ,S}$ is a normal subgroup of $G_{\QQ,S}$, $V^*(1)^{H_{\QQ,S}}$ is a subset of $V^*(1)$ which is stable under the action of $G_{\QQ,S}$. By irreducibility of $V^*(1)$, either $V^*(1)^{H_{\QQ,S}}=V^*(1)$ or $V^*(1)^{H_{\QQ,S}}=0$. Suppose the former scenario, note that this representation factors through $\Gamma_0\cong\Zp$, which is an abelian group. By Schur's lemma, nontrivial irreducible representations of an abelian group over an algebraically closed field must be a character of that group. This gives us $V^*(1)\otimes_E\overline{E}\cong\overline{E}$ as vector spaces, which gives that $V^*(1)$ is also a $1$-dimensional representation. By our assumption this is not the case, thus we get $V^*(1)^{H_{\QQ,S}}=0$, which implies that $T^*(1)^{H_{\QQ,S}}\subseteq V^*(1)^{H_{\QQ,S}}=0$. Therefore the results follows by Proposition \ref{cor:pottharst-proposition-hi-ranks}.
\end{proof}

\par We finally relate the Selmer complex associated to $(V,\mathbf{D})$ to the Bloch-Kato Selmer group under certain conditions:

\begin{proposition}\label{prop:Comparison-with-BK}
    Let $V$ be a $p$-adic representation with coefficients in $E$, and $\mathbf{D}\subseteq \mathbf{D}_{\rm cris}(V)$ be a crystalline submodule coming from the triangulation of $V$. Suppose that $V$ and $\mathbf{D}$ enjoy the following properties:
    \begin{itemize}
        \item ${\rm Fil}^0\mathbf{D}=0$,
        \item ${\rm Fil}^0\mathbf{D}_{\rm cris}(V)/\mathbf{D}=\mathbf{D}_{\rm cris}(V)/\mathbf{D}$,
        \item $(\mathbf{D}_{\rm cris}(V)/\mathbf{D})^{\varphi=1}=0$,
        \item $\mathbf{D}/(p\varphi -1)\mathbf{D}=0$.
    \end{itemize}
    Then the Bloch-Kato local condition coincides with the local condition at $p$ and the first cohomology of the Selmer Complex gives the Bloch-Kato Selmer group. More precisely,
    \[H^1_f(\Qp,V)=H^1(\mathbf{D}),{\textrm{ and }}H^1_f(\QQ,V)=H^1(V,\mathbf{D}).\]
\end{proposition}
\begin{proof}
    See \cite[Proposition 6 (v)]{benois2014selmer} and \cite[Theorem 4.1 (4)]{pottharst2012cyclotomic}.
\end{proof}

\subsection{Perfectness of Selmer Complexes}\label{subsect:perfectselmercomp}

\par In this subsection we will prove that under certain conditions, the Selmer complexes that we are interested in are perfect complexes with perfect amplitude $[1,2]$. 

\begin{definition}[Perfect complex]\label{def:Perfect-Comp}
    We say a cochain complex $K^\bullet$ of $R$-modules is perfect if $K^\bullet$ is quasi isomorphic to a bounded complex consisting of finite projective $R$-modules. We denote the derived category of perfect complexes of $R$-modules by $\mathscr{D}_{\rm parf}(R)$. If $K^\bullet$ can be represented by finite projective $R$-modules concentrated at degrees $[a,b]$, then we write $K^\bullet\in\mathscr{D}_{\rm parf}^{[a,b]}(R)$.
\end{definition}

\par As usual let $\mathbf{V}$ denote a $p$-adic representation of $G_{\QQ,S}$ with coefficients in an affinoid algebra $A$ over a finite extension $E$ of $\Qp$, and let $\mathbf{D}\subseteq \mathbf{D}_{{\rm rig},A}^\dagger(\mathbf{V})$ be a $(\phiGamma)$-module. We first state the main theorem of this section:

\begin{theorem}\label{thm:Perfectness-amplitude-1-2}
    Assume that the following conditions hold:
\begin{enumerate}[(i)]
    \item $H^3_\Iw(\mathbf{V},\mathbf{D})=0$. \label{cond:h3Iw-vanishes}
    \item For each prime ideal $\mathfrak{p}$ of $\Lambda_A[1/p]:=\Lambda[1/p]\hatotimes A$, we have $(\mathbf{V}\otimes_A(\Lambda_A[1/p]^\iota)/\frakp)^{G_{\QQ,S}}=0$, \label{cond:h0Iw-residual-vanishes}
    \item For each $\ell\in S_f$, $\ell\neq p$, $H^1(I_\ell,\mathbf{V})$ is projective as $A$-module. \label{cond:tamagawa-proj}
\end{enumerate}
Then $\RGamma_{\rm Iw}(\mathbf{V},\mathbf{D})\in \mathscr{D}_{\rm parf}^{[1,2]}(\mathcal{H}_A)$.
\end{theorem}

\begin{remark}
    In \cite[Proposition 4.46]{danielebuyukboduksakamoto2023} the perfectness of Nekovář's Selmer complexes of $R$-modules with perfect-amplitude $[1,2]$ is established, where $R$ is a complete Noetherian local ring with finite residue field of characteristic $p$. In Theorem \ref{thm:Perfectness-amplitude-1-2}, we generalize this theorem to the Pottharst style Selmer complexes of $\mathcal{H}_A$-modules. The key difference in our theorem is that $\mathcal{H}_A$ is neither a Noetherian nor a local ring. 
\end{remark}

\par Now we state several lemmas that we will use for the proof of Theorem \ref{thm:Perfectness-amplitude-1-2}. For $\ell\neq p$ and $\ell\in S_f$, let $I_\ell^w$ denote the wild inertia subgroup of $I_\ell$, and let $t_\ell$ be a topological generator of $I_\ell/I_\ell^w$.
\begin{lemma}\label{lemma:wild-inertia-invariant-projective}
    The following are valid:
    \begin{itemize}
        \item $V^{I_\ell^w}$ is finite projective submodule of $V$,
        \item $\RGamma(I_\ell,V)\in\mathscr{D}_{\rm parf}^{[0,1]}(A)$, and can be represented by the following perfect complex:
        \[
            \left[V^{I_\ell^w}\xrightarrow{t_\ell -1}V^{I_\ell^w}\right],
        \]
        \item If $f:A\rightarrow R$ is any ring homomorphism, then $V^{I_\ell^w}\otimes_A R\cong (V\otimes_A R)^{I_\ell^w}$, and $\RGamma(I_\ell,V)\otimes_A^\mathbf{L}R\simeq\RGamma(I_\ell,V\otimes_A R)$. In particular, $\RGamma(I_\ell,V)\otimes_A^\mathbf{L}\mathcal{H}_A^{\iota}\simeq\RGamma_\Iw(I_\ell,V)$.
    \end{itemize}
\end{lemma}
\begin{proof}
    See \cite[Lemma 4.21]{danielebuyukboduksakamoto2023}. The only difference is that in the setting of \cite{danielebuyukboduksakamoto2023}, $V$ is a representation over a complete local Noetherian ring with finite residue field of characteristic $p$. The key point to the proof is to show that the image of the wild inertia subgroup $I_\ell^w$ in $\Aut_A(V)$ is finite. However this is also true for representations over affinoid algebras (see (55) in \cite[\S 3.1.6]{benois-p-adic-heights}). The rest of the proof is exactly the same. 
\end{proof}

\begin{lemma}[Perfectness of Selmer complexes]\label{lemma:perfectness-unramified-tamagawa}
Suppose that the condition \eqref{cond:tamagawa-proj} in Theorem \ref{thm:Perfectness-amplitude-1-2} holds. Then the Iwasawa theoretic analytic Selmer complex $\RGammaIw(\mathbf{V},\mathbf{D})$ lies in $\mathscr{D}_{\rm{parf}}^{[0,3]}(\mathcal{H}_A)$.
\end{lemma}
\begin{proof}
    By \cite[Theorem 1.11]{Pottharst2013}, it is sufficient to check that for each $v\in S_f$, $U_{\Iw ,v}^+(\mathbf{V},\mathbf{D})\in\mathscr{D}_{\rm parf}^{[0,2]}(\mathcal{H}_A)$. For $v=p$ we have $U_{\Iw ,p}^+(\mathbf{V},\mathbf{D})=\RGammaIw(\mathbf{D})\in\mathscr{D}_{\rm parf}^{[0,2]}(\mathcal{H}_A)$ by \cite[Theorem 4.4.6]{KedlayaPottharstXiao}. Thus it is sufficient to show that for all $\ell\neq p$ in $S_f$, $U_{\Iw,\ell}^+(\mathbf{V},\mathbf{D})\in\mathscr{D}_{\rm{parf}}^{[0,1]}(\mathcal{H}_A)$.

    \par Recall the definition of $U_{\Iw,\ell}^+$:
    \[
        U_{\Iw,\ell}^+(\mathbf{V},\mathbf{D}):=\left[(\mathbf{V}\otimes_A\mathcal{H}_A)^{I_\ell}\xrightarrow{\Frob_\ell -1}(\mathbf{V}\otimes_A\mathcal{H}_A)^{I_\ell}\right].
    \]
    Since $I_\ell$ acts trivially on $\mathcal{H}_A^\iota$, by \cite[Proposition 4.24 (3)]{danielebuyukboduksakamoto2023} (note that this is also true in our setting thanks to \ref{lemma:wild-inertia-invariant-projective}) we have:
    \[
        U_{\Iw,\ell}^+(\mathbf{V},\mathbf{D})\simeq\left[\mathbf{V}^{I_\ell}\otimes_A\mathcal{H}_A\xrightarrow{\Frob_\ell -1}\mathbf{V}^{I_\ell}\otimes_A\mathcal{H}_A\right]\simeq\left[\mathbf{V}^{I_\ell}\xrightarrow{\Frob_\ell -1}\mathbf{V}^{I_\ell}\right]\otimes_A\mathcal{H}_A,
    \]
    where the last isomorphism follows from the definition of tensor product of complexes.
    
    \par Now we claim that $\mathcal{H}_A$ is (faithfully) flat over $A$. By the base-change property of faithfully flatness, it is sufficient to show that $\mathcal{H}$ is faithfully flat over $E$. Notice that $\Lambda_E\cong\mathcal{O}_E[[T]]$ is faithfully flat over $\mathcal{O}_E$, hence $\Lambda_E[1/p]:=\Lambda_E\otimes_{\mathcal{O}_E}E$ is faithfully flat over $E$. Since we know that $\mathcal{H}=\mathcal{H}_E$ is faithfully flat over $\Lambda_E[1/p]$, we obtain the result. Therefore,
    \[
        U_{\Iw,\ell}^+(\mathbf{V},\mathbf{D})\simeq U_\ell^+(\mathbf{V},\mathbf{D})\otimes^\mathbf{L}_A\mathcal{H}_A.
    \]
    \par By our assumption, thanks to \cite[3.1.7]{benois-p-adic-heights}, $U_\ell^+(\mathbf{V},\mathbf{D})$ is a perfect complex of $A$-modules with perfect amplitude $[0,1]$, therefore $U_\ell^+(\mathbf{V},\mathbf{D})\otimes^\mathbf{L}_A\mathcal{H}_A$ lies in $\mathscr{D}_{\rm parf}^{[0,1]}(\mathcal{H}_A)$.
\end{proof}

\begin{lemma}\label{lemma:tamagawa-holds-when-A-is-PID}
    If the affinoid algebra $A$ is a PID, then $H^1(I_\ell,\mathbf{V})$ is projective, and Lemma \ref{lemma:perfectness-unramified-tamagawa} holds true.
\end{lemma}
\begin{proof}
    Since $A$ is a Dedekind domain, every submodule of projective $A$-module is projective. Hence $H^0(I_\ell,\mathbf{V})$ and $(t_\ell -1)\mathbf{V}^{I_\ell^w}$ are projective $A$-modules. Now we will use the arguments of \cite[Theorem 3.1.7 (ii)]{benois-p-adic-heights} to show that $H^1(I_\ell,\mathbf{V})$ is projective. Consider the following exact sequence:
    \begin{align}\label{eq:PID-exact-seq-1}
        0\rightarrow \mathbf{V}^{I_\ell}\rightarrow \mathbf{V}^{I_\ell^w}\xrightarrow{t_\ell -1} (t_\ell -1)\mathbf{V}^{I_\ell^w}\rightarrow 0.       
    \end{align}
    This sequence splits since $(t_\ell -1)\mathbf{V}^{I_\ell^w}$ is projective. Dualizing \eqref{eq:PID-exact-seq-1} we get 
    \begin{align}\label{eq:PID-exact-seq-2}
        0\rightarrow  (t_\ell -1)\mathbf{V}^*(1)^{I_\ell^w} \rightarrow \mathbf{V}^*(1)^{I_\ell^w}\rightarrow (\mathbf{V}^{I_\ell})^*(1)\rightarrow 0,
    \end{align}
    which splits again since $(\mathbf{V}^{I_\ell})^*(1)$ is projective. Hence we get
    \[
        H^1(I_\ell,\mathbf{V}^*(1))(1)\cong \mathbf{V}^*(1)^{I_\ell^w}/(t_\ell -1)\mathbf{V}^*(1)^{I_\ell^w} \cong (\mathbf{V}^{I_\ell})^*(1).
    \]
    Hence (replacing the roles of $\mathbf{V}$ and $\mathbf{V}^*(1)$), we conclude that $H^1(I_\ell,\mathbf{V})$ is projective.
\end{proof}

\par We will prove that our Selmer complexes can be represented by a complex of projectives concentrated at degrees $[1,2]$. In the following definitions and lemmas, let $R$ denote a ring and $K^\bullet$ denote a complex of $R$-modules in $\mathscr{D}(R)$.

\begin{definition}\label{def:Tor-Amplitude}
    We say that a complex $K^\bullet\in \mathscr{D}(R)$ has finite $\Tor$ dimension with $\Tor$ amplitude $[a,b]$ if for every $R$-module $M$, we have $H^i(K^\bullet\otimes^{\mathbf{L}}_RM)=0$ if $i\notin[a,b]$.
\end{definition}

\begin{lemma}\label{lemma:stacks-perfect-tor-amplitude}
    $K^\bullet$ is perfect iff $K^\bullet$ is pseudo-coherent and has finite $\Tor$ dimension. Moreover if one of these properties holds and $K^\bullet$ has $\Tor$ amplitude $[a,b]$, then $K^\bullet\in\mathscr{D}_{\rm parf}^{[a,b]}(R)$.
\end{lemma}
\begin{proof}
    See \cite[\href{https://stacks.math.columbia.edu/tag/0658}{Lemma 0658}]{stacks_project}.
\end{proof}

\begin{lemma}\label{lemma:perfect-last-cohomology-vanish}
Let $K^\bullet\in\mathscr{D}_{\rm parf}^{[a,b]}(R)$ such that $H^b(K^\bullet)=0$. Then $K^\bullet\in\mathscr{D}_{\rm parf}^{[a,b-1]}(R)$.  
\end{lemma}
\begin{proof}
    Let $K^\bullet\simeq \left[P^a\rightarrow\dots\rightarrow P^b\right]$ (with $P^i$ finite projective for $i\in[a,b]$) such that $H^b(K^\bullet)=0$. Then the map $d^{b-1}:P^{b-1}\rightarrow P^b$ is surjective. Note that by definition, \[K^\bullet\otimes^{\mathbf{L}}_RM\simeq\left[P^a\otimes_RM\xrightarrow{d^a\otimes 1_M}\dots\xrightarrow{d^{b-1}\otimes 1_M} P^b\otimes_R M\right],\] therefore the right-exactness of $\blank\otimes_RM$ implies that $d^{b-1}\otimes 1_M:P^{b-1}\otimes_R M\rightarrow P^{b}\otimes_R M$ is surjective. This completes the proof.
\end{proof}

\begin{lemma}\label{lemma:stacks-exact-triangle-perfect-amplitude}
    Let \[K^\bullet\rightarrow L^\bullet\rightarrow M^\bullet\rightarrow K^\bullet[1]\] be an exact triangle in $\mathscr{D}(R)$. If any two of $\{K^\bullet,L^\bullet,M^\bullet\}$ are perfect, so is the third. If $L^\bullet$ has $\Tor$ amplitude $[a+1,b+1]$ and $M^\bullet$ has $\Tor$ amplitude $[a,b]$, then $K^\bullet$ has $\Tor$ amplitude $[a+1,b+1]$.
\end{lemma}
\begin{proof}
    See \cite[\href{https://stacks.math.columbia.edu/tag/066R}{Lemma 066R-(3)}]{stacks_project} and \cite[\href{https://stacks.math.columbia.edu/tag/0655}{Lemma 0655}]{stacks_project}. 
\end{proof}

\begin{lemma}\label{lemma:base-change-perfectness-tor-amplitude}
    Let $R\rightarrow R^\prime$ be a ring map. If $K^\bullet$ has $\Tor$ amplitude $[a,b]$ as a complex of $R$-modules, then $K^\bullet\otimes_R^\mathbf{L}R^\prime$ has $\Tor$ amplitude $[a,b]$ as a complex of $R^\prime$-modules. If $R\rightarrow R^\prime$ is faithfully flat, and if $K^\bullet\otimes_R R^\prime\in \mathscr{D}_{\rm parf}(R^\prime)$, then $K^\bullet\in \mathscr{D}_{\rm parf}(R)$.
\end{lemma}
\begin{proof}
    See \cite[\href{https://stacks.math.columbia.edu/tag/068T}{Lemma 068T}]{stacks_project} and \cite[\href{https://stacks.math.columbia.edu/tag/066L}{Lemma 066L}]{stacks_project}.
\end{proof}

\begin{proof}[Proof of Theorem \ref{thm:Perfectness-amplitude-1-2}]
It remains to show that $\RGamma_{\rm Iw}(\mathbf{V},\mathbf{D})$ has the perfect amplitude $[1,2]$. If we can show that $\RGamma_{\rm Iw}(\mathbf{V},\mathbf{D})$ has $\Tor$ amplitude in $[1,2]$, then we are done by Lemma \ref{lemma:stacks-perfect-tor-amplitude}. Suppose that we showed $\RGamma_{\rm Iw}(\mathbf{V},\mathbf{D})$ has $\Tor$ amplitude in $[1,3]$. Then combining with our running assumption that $H^3_{\rm Iw}(\mathbf{V},\mathbf{D})=0$, we are done by Lemma \ref{lemma:perfect-last-cohomology-vanish}.

Consider the following exact triangle (see \cite[\S 7.2]{harronpottharst2014iwasawa}):
\begin{align*}
    \RGamma_{\rm Iw}(\mathbf{V},\mathbf{D})\longrightarrow \RGamma_{(p),{\rm Iw}}(\QQ,\mathbf{V}) \longrightarrow \RGammaIw(\widetilde{\mathbf{D}}),
\end{align*}
where $\widetilde{\mathbf{D}}:= \mathbf{D}_{{\rm rig},A}^\dagger(\mathbf{V})/\mathbf{D}$ and $\RGamma_{(p),{\rm Iw}}(\QQ,\mathbf{V})$ denotes the $p$-relaxed Selmer complex. We know from \cite[Theorem 4.4.6]{KedlayaPottharstXiao} that \[\RGammaIw(\widetilde{\mathbf{D}})\in\mathscr{D}_{\rm parf}^{[0,2]}(\mathcal{H}_A).\] Thus by Lemma \ref{lemma:stacks-exact-triangle-perfect-amplitude} it suffices to show that \[\RGamma_{(p),{\rm Iw}}(\QQ,\mathbf{V})\in \mathscr{D}_{\rm parf}^{[1,3]}(\mathcal{H}_A).\]

\par Let \[K^\bullet:=\RGamma_{(p)}(\QQ,\mathbf{V}\otimes_A \Lambda_A[1/p]^\iota),\] then we have \[K^\bullet\otimes_{\Lambda_A[1/p]}\mathcal{H}_A\simeq \RGamma_{(p),{\rm Iw}}(\QQ,\mathbf{V}).\] Note that $\mathcal{H}$ is faithfully flat over $\Lambda[1/p]$, therefore $\mathcal{H}_A$ is faithfully flat over $\Lambda_A[1/p]$ by the base-change property of faithfully flat maps. Hence by Lemma \ref{lemma:base-change-perfectness-tor-amplitude}, $K^\bullet$ is perfect and we reduce the problem to showing that $H^0(K^\bullet\otimes_{\Lambda_A[1/p]}^\mathbf{L}M)=0$ for all $M$.

\par Note that $\Lambda_A[1/p]$ is a Noetherian ring, therefore the rest of the argumentation is reduced to the what is shown in \cite[\S 4.5]{danielebuyukboduksakamoto2023}. 

Since $K^\bullet$ is perfect $K^\bullet\simeq \left[P^0\rightarrow P^1\rightarrow P^2 \rightarrow P^3 \right]$, where $P^i$'s are all projective. Then we have \[H^0(K^\bullet\otimes^\mathbf{L}M):=H^0([P^0\otimes M\rightarrow P^1\otimes M])\cong\Tor_1^{\Lambda[1/p]}(P^1/P^0, M).\]

\par Note that the last isomorphism is not true in general, however our assumption \ref{cond:h0Iw-residual-vanishes} ensures that $H^0(K^\bullet)=0$, and in this case this isomorphism holds. To show that for any $M$ the module $\Tor_1^{\Lambda_A[1/p]}(P^1/P^0, M)=0$, we show for any prime $\mathfrak{p}$ of $\Lambda_A[1/p]$, the localization at $\mathfrak{p}$ is 0, i.e. \[\left(\Tor_1^{\Lambda[1/p]}(P^1/P^0, M)\right)_\mathfrak{p}=\Tor_1^{(\Lambda_A[1/p])_\mathfrak{p}}((P^1/P^0)_\mathfrak{p}, M_\mathfrak{p})=0.\] If we show that $(P^1/P^0)_\mathfrak{p}$ is flat over $\Lambda_A[1/p]_\mathfrak{p}$, then the result follows. This can be shown by using the local criterion for flatness: Let $X:=(P^1/P^0)_\frakp$. Note that $X$ is finitely generated and $X/\frakp$ is flat over $(\Lambda_A[1/p])_\frakp/\frakp$, thus we only need to show that $\Tor_1^{\Lambda[1/p]_\frakp}((\Lambda_A[1/p]_\frakp)/\frakp,X)=0$. Note that $K^\bullet \simeq [P^0\rightarrow P^1 \rightarrow P^2 \rightarrow P^3]$ gives $K^\bullet_\frakp \simeq [P^0_\frakp \rightarrow P^1_\frakp \rightarrow P^2_\frakp \rightarrow P^3_\frakp]$. Hence \[\Tor_1^{(\Lambda_A[1/p])_\frakp}((\Lambda_A[1/p]_\frakp)/\frakp,X)=H^0(K^\bullet_\frakp\otimes^\mathbf{L}(\Lambda_A[1/p]_\frakp)/\frakp).\] By the definition of $K^\bullet$, the last term is \[H^0_{(p)}(\QQ,\mathbf{V}\otimes (\Lambda_A[1/p])_\frakp^\iota/\frakp)\cong H^0_{(p)}(\QQ,\mathbf{V}\otimes \Lambda_A[1/p]^\iota/\frakp)\otimes_{\Lambda_A[1/p]}(\Lambda_A[1/p])_\frakp,\] and \[H^0_{(p)}(\QQ,\mathbf{V}\otimes \Lambda_A[1/p]^\iota/\frakp)\subseteq (\mathbf{V}\otimes (\Lambda_A[1/p])_\frakp^\iota/\frakp)^{G_{\QQ,S}}.\] By our assumption \ref{cond:h0Iw-residual-vanishes}, we get $(\mathbf{V}\otimes \Lambda_A[1/p]/\frakp)^{G_{\QQ,S}}=0$, hence the result follows.
\end{proof}

\begin{remark}\label{remark:perfectness-1-2-using-duality}
    If the following duality theorem in $\mathscr{D}_{\rm coad}(\mathcal{H})$ for Iwasawa theoretic Selmer complexes attached to representations $V$ over $E$ 
    \[\RGammaIw(V^*(1),\mathbf{D}^\perp)\simeq\RGammaIw(V,\mathbf{D})^*[-3]\] 
    could be generalized to the Selmer complexes attached to representations $\mathbf(V)$ over $A$ lying in $\mathscr{D}_{\rm parf}^{[0,3]}(\mathcal{H}_A)$, then if the assumptions \ref{cond:h3Iw-vanishes} and \ref{cond:tamagawa-proj} of Theorem \ref{thm:Perfectness-amplitude-1-2} hold for $(\mathbf{V},\mathbf{D})$ and $(\mathbf{V}^*(1),\mathbf{D}^\perp)$, we can still show that \[\RGamma_{\rm Iw}(\mathbf{V},\mathbf{D}),\ \RGamma_{\rm Iw}(\mathbf{V}^*(1),\mathbf{D}^\perp)\in\mathscr{D}_{\rm parf}^{[1,2]}(\mathcal{H}_A)\] by arguing as in \cite[Proposition 9.7.4-(ii)]{Nek06}. 
    \par In particular when $A=E$, the following assumption \[H^3_\Iw(V,\mathbf{D})=0=H^3_\Iw(V^*(1),\mathbf{D}^\perp)\] is sufficient to show that \[\RGamma_{\rm Iw}(V,\mathbf{D}),\ \RGamma_{\rm Iw}(V^*(1),\mathbf{D}^\perp)\in\mathscr{D}_{\rm parf}^{[1,2]}(\mathcal{H})\]
\end{remark}

\begin{corollary}\label{cor:algebraic-p-adic-L-function}
    Assume that \[H^1_\Iw(\mathbf{V},\mathbf{D})=0=H^1_\Iw(\mathbf{V}^*(1),\mathbf{D}^\perp).\] We have \[\RGammaIw(\mathbf{V},\mathbf{D})\simeq \left[P^1\xrightarrow{f}P^2\right],\] where $P^1$ and $P^2$ are finite projective $\mathcal{H}_A$-modules of same rank. Then \[\det(\RGammaIw(\mathbf{V},\mathbf{D}))=\det(f)=:{\rm Fitt}_0(H^2_\Iw(\mathbf{V},\mathbf{D})).\] In particular when $A=E$ is a finite extension of $\Qp$, we have \[\det(\RGammaIw(V,\mathbf{D}))={\rm char}_{\mathcal{H}}(H^2_\Iw(V,\mathbf{D})).\]
\end{corollary}
\begin{definition}[Algebraic $p$-adic $L$-function]
    With the above motivation, we define the module of algebraic $p$-adic $L$-functions attached to $(\mathbf{V},\mathbf{D})$ as $\det(\RGammaIw(\mathbf{V},\mathbf{D}))$.
\end{definition}

\section{Rankin-Selberg Products and Beilinson-Flach Elements}\label{sect:ES-BF}

\subsection{Rankin-Selberg Products and Adjoint Motives}\label{subsect:rankin-selberg}

\par Let $f\in S_{k_f+2}(N_f,\varepsilon_f)$ and $g\in S_{k_g+2}(N_g,\varepsilon_g)$ be cuspidal normalized newforms of weights $k_f+2$ and $k_g+2$, levels $\Gamma_1(N_f)$ and $\Gamma_1(N_g)$ with nebentypes $\varepsilon_f$ and $\varepsilon_g$ respectively. Let $\psi$ be a finite order Dirichlet character with conductor $N_\psi$. We will assume that $p\nmid N_fN_gN_\psi$. Let $\alpha_f$ and $\beta_f$ (resp. $\alpha_g$ and $\beta_g$) be the roots of the Hecke polynomial $X^2-a_p(f)X+\varepsilon_fp^{k_f+1}$ (resp. $X^2-a_p(g)X+\varepsilon_gp^{k_g+1}$). We will always assume that $\alpha_f\neq\beta_f$ and $\alpha_g\neq\beta_g$ (called $p$-regularity hypothesis), which is conjecturally true and known to be true when $f$ (and $g$) is of weight $2$.

\par The imprimitive Rankin-Selberg $L$-function is defined as \[L(f,g,\psi,s):=L_{(N_fN_gN_\psi)}(\varepsilon_f\varepsilon_g\psi^2,2s-2-k_f-k_g)\sum_{n\geq 1}\frac{a_n(f)a_n(g)\psi(n)}{n^{-s}},\] where for $S\in\NN^+$ and a character $\chi$, $L_S(\chi,s)$ denotes the Dirichlet $L$-series without Euler factors at primes dividing $S$ (see \cite[\S 2]{arlandini2021factorisation} and \cite[\S 2.1]{Loeffler_2018} for further details). Note that this impritimitve $L$-function differs from the primitive $L$-function attached to $f\otimes g\otimes\psi$ up to finitely many Euler factors.

\par Recall that $V_f$ is the $2$-dimensional (cohomological) Galois representation attached to $f$ with coefficients in $E$, which is a finite extension of $\Qp$ and which contains the Hecke field $K_f$ and $U_p$-eigenvalues $\alpha_f$ and $\beta_f$. Let $T_f\subseteq V_f$ denote a $G_\QQ$ stable rank $2$ $\mathcal{O}_E$-lattice inside $V_f$. Let $\Sym^2f$ denote the symmetric-square motive, whose $p$-adic realization is a 3-dimensional subrepresentation $\Sym^2(V_f)$ of $V_f\otimes V_f$ consisting of symmetric tensors. Following \cite[Proposition 2.1.4]{LoefflerZerbesSym2J_Reine_Angew_Math_2019} we can define the imprimitive $L$-function for $\Sym^2f$ as \[L^{\rm imp}(\Sym^2f,\psi,s):=\frac{L(f,f,\psi,s)}{L_(N_fN_\psi)(\varepsilon\psi,s-k-1)}.\]

\par Note that the representation $V_f\otimes_E V_f\otimes\psi$ splits as \[V_f\otimes V_f\otimes\psi\cong {\rm Sym}^2(V_f)(\psi)\oplus\wedge^2V_f(\psi),\] where \[\wedge^2V_f:=\det(V_f)\cong\varepsilon_f\chi_{\rm cyc}^{-1-k}.\]

\par We also define \[\Ad(V_f):=V_f\otimes V_f^*,\] and \[\ad(V_f):=\ker\left(\Ad(V_f)\xrightarrow{\rm Tr}E\right).\] Then we have the splitting of representations \[\Ad(V_f)(\psi)\cong \ad(V_f)(\psi)\oplus E(\psi).\]Note that we have \[\ad(V_f)\cong \Sym^2V_f(k_f+1)\otimes\varepsilon^{-1}.\] This follows using the splittings above and \eqref{eq:poincare-duality-reps-forms}, see also \cite[\S 1.3]{dasgupta}.

\par Let us see that $\Ad(V_f)(\psi)$ (or equivalently, $V_f\otimes V_f\otimes\psi$) does not have any critical values, by applying the criticality criteria given in Definition \ref{def:critical-Deligne-p-adic}. 

\par To see this, note that $d^+(V_f)=1$.  Recall that we defined $d^+(V):=\dim V^{c=1}$ just before Definition \ref{def:critical-Deligne-p-adic}. Hence for all $n\in\ZZ$, we have 
\begin{align*}
    &d^+(\Ad(V_f)(\psi+n))=2,\\
    &d^+(E(\psi+n))=\frac{1+(-1)^n\psi(-1)}{2},\\
    &d^+(\ad(V_f)(\psi+n))=d^+(\Ad(V_f)(\psi+n))-d^+(E(\psi+n))=\frac{3-(-1)^n\psi(-1)}{2}.
\end{align*}

\par Let us also compute the Hodge-Tate weights of twists of these motives: $V_f$ has Hodge-Tate weights $\HT(V_f)=\{-1-k_f,0\}$, hence we get
\begin{align*}
    &\HT(\Ad(V_f)(\psi+n))=\{-1-k_f+n,n,n,1+k_f+n\},\\
    &\HT(E(\psi+n))=\{n\},\\
    &\HT(\ad(V_f)(\psi+n))=\{-1-k_f+n,n,1+k_f+n\}.
\end{align*} 
From this computation, one sees that there are no critical values for $\Ad(V_f)(\psi)$. We can also compute the critical values for $\ad(V_f)(\psi)$ as \[\{n\in\ZZ:\psi+n\textnormal{ is even and }-k_f-1< n \leq 0 \}\cup \{n\in\ZZ:\psi+n\textnormal{ is odd and }0< n \leq 1+k_f\},\] and the critical values for $E(\psi)$ as \[\{n\in\ZZ:\psi+n\textnormal{ is even and }n \geq 1\}\cup \{n\in\ZZ:\psi+n\textnormal{ is odd and }n\leq 0\}.\] Note that the critical values for $\ad(V_f)(\psi)$ and $E(\psi)$ do not coincide, which reflects the fact that there are no critical values for $\Ad(V_f)(\psi)$.

\par Hence it is not trivial to establish $p$-adic Artin formalism for the $p$-adic $L$-functions of these representations. Therefore one needs to work with $p$-adic $L$-functions attached to the deformations of these representations, i.e. replace $f$ by a suitable family of $p$-adic modular forms passing through $f$. In this article we are interested in the case where $f$ is $p$-non-ordinary, hence we will use Coleman families (see Definition \ref{def:coleman-family}) to interpolate $f$.

\par From now on, let $f$ be $p$-non-ordinary, so both of its $p$-stabilizations $f_{\alpha_f}$ and $f_{\beta_f}$ are non-$\theta$-critical. We fix a $p$-stabilization $f_\alpha:=f_{\alpha_f}$. Since $f$ is a cusp form, $f_\alpha$ corresponds to a point $x_0\in\mathcal{C}^0_N$ such that $\kappa(x_0)=k_f\in\mathcal{W}$, where $\kappa:\mathcal{C}^0_N\rightarrow \mathcal{W}$ is the weight map. Moreover, since $f_\alpha$ is non critical, $\kappa$ is étale at $x_0$, so a nice affinoid neighborhood of $k_f\in\mathcal{W}$ can be identified with a nice affinoid neighborhood $U=\Spm(A)$ of $x_0$ in the eigencurve. 

\par Let $\mathcal{F}$ be the unique Coleman family over $U$ passing through $f_\alpha$. Note that in this case the coefficient field $E$ of $V_f$ is isomorphic to $A/\frakm_{x_0}$, where $\frakm_{x_0}$ is the maximal ideal of $A$ corresponding to the point $x_0\in U$. Let $\mathbf{V}_\mathcal{F}$ denote the big $G_\QQ$ representation (with coefficients in $A$) attached to $\mathcal{F}$ by Theorem \ref{thm:rep-coleman-F}. In particular we have \[(\mathbf{V}_\mathcal{F})_{x_0}:=\mathbf{V}_\mathcal{F}\otimes_A A/\mathfrak{m}_{x_0}\cong V_f.\] 

\par If $\Delta:U\rightarrow U\times_E U$ denotes the diagonal map, then the corresponding map of affinoid algebras $\Delta^*:A\hatotimes_E A\rightarrow A$ is given by the multiplication map sending $a_1\otimes a_2$ to $a_1a_2$. This map pulls back the representation  
$\mathbf{V}_\mathcal{F}\hatotimes_E\mathbf{V}_\mathcal{F}$ to \[\Delta^*(\mathbf{V}_\mathcal{F}\hatotimes_E\mathbf{V}_\mathcal{F})\cong \mathbf{V}_\mathcal{F}\hatotimes_A\mathbf{V}_\mathcal{F}.\] 
Thus we have the following decomposition of representations in families:
\[\Delta^*(\mathbf{V}_\mathcal{F}\hatotimes_E\mathbf{V}_\mathcal{F})(\psi)\cong\Sym^2(\mathbf{V}_\mathcal{F})\oplus\wedge^2\mathbf{V}_\mathcal{F}(\psi),\] where \[\wedge^2\mathbf{V}_\mathcal{F}=\varepsilon_{\mathcal{F}}\chi_{\rm cyc}^{-1}\mathbf{k}_\mathcal{F}^{-1}.\] Here $\varepsilon_{\mathcal{F}}$ is the nebentype of $\mathcal{F}$ and $\mathbf{k}_\mathcal{F}$ is the weight character of $\mathcal{F}$ (with values in $A$) such that each character in $U\subseteq\mathcal{W}$ factors through $\mathbf{k}_\mathcal{F}$.
\par Let 
\begin{align*}
    \mathbf{V}_\mathcal{F}^*&:=\Hom_A(\mathbf{V}_\mathcal{F},A)\cong\mathbf{V}_{\mathcal{F}^c}(\mathbf{k}_\mathcal{F}+1),\\
    \Ad(\mathbf{V}_\mathcal{F})&:=\Hom_A(\mathbf{V}_\mathcal{F},\mathbf{V}_\mathcal{F})\cong \mathbf{V}_\mathcal{F}\hatotimes_A\mathbf{V}_\mathcal{F}^*,\\
    \ad(\mathbf{V}_\mathcal{F})&:=\ker(\Ad(\mathbf{V}_\mathcal{F})\xrightarrow{\rm{Tr}}A).
\end{align*}
Note that we have the decomposition 
\[
    \Delta^*(\mathbf{V}_\mathcal{F}\hatotimes_E\mathbf{V}_\mathcal{F}^*)(\psi)\cong\Ad(\mathbf{V}_\mathcal{F})(\psi)\cong\ad(\mathbf{V}_\mathcal{F})(\psi)\oplus A(\psi).
\]
\par The following theorem on factorization of $p$-adic $L$-functions, which generalizes the main theorem of Dasgupta in \cite{dasgupta} is due to Arlandini and Loeffler in \cite[Theorem A, Theorem B]{arlandini2021factorisation}.

\begin{theorem}[Arlandini-Loeffler]\label{thm:Arlandini-Factorization-Result}
    Let \[L_p^{\rm geom}(\mathcal{F},\mathcal{G}):=L_p^{\rm geom}(\mathcal{F},\mathcal{G})(\mathbf{k}_\mathcal{F},\mathbf{k}_\mathcal{G},\mathbf{j})\] denote the three variable geometric Rankin-Selberg $p$-adic $L$-function defined in \cite[\S 9]{LZ16}. There exists a unique two variable meromorphic $p$-adic $L$-function \[L_p^{\rm imp}(\Sym^2(\mathbf{V}_\mathcal{F}))\] interpolating the critical values of $L^{\rm imp}(\Sym^2(f),s)$ (see \cite[Theorem A]{arlandini2021factorisation} for the exact interpolation properties) such that the following factorization formula holds for all $(\sigma_1,\sigma_2)\in U\times\mathcal{W}$:
    \begin{align*}
        L_p^{\rm geom}(\mathcal{F},\mathcal{F})(\sigma_1,\sigma_1,\sigma_2)=L_p^{\rm imp}(\Sym^2(\mathbf{V}_\mathcal{F}))(\sigma_1,\sigma_2)L_p(\varepsilon_f,\sigma_2-\sigma_1-1).
    \end{align*}
\end{theorem}

\par In the following theorem \cite[Theorem 9.3.2.]{LZ16} of Loeffler and Zerbes, the interpolation property of $L_p^{\rm geom}$ is given as follows:
\begin{theorem}[Interpolation property of $L_p^{\rm geom}$]\label{thm:Lp-geom-interpolation}
    For Coleman families $\mathcal{F}$ and $\mathcal{G}$ over (nice affinoid domains) $U_\mathcal{F}$ and $U_\mathcal{G}$, the geometric $p$-adic $L$-function $L_p^{\rm geom}$ has the following interpolation property: If $(k,k^\prime,j)\in U_\mathcal{F}\times_E U_\mathcal{G}\times_E \mathcal{W}\cap\NN^3$ with $k\geq 0$, $k^\prime\geq -1$, and $\frac{k+k^\prime+1}{2}\leq j\leq k$, and if $\mathcal{F}_k$ and $\mathcal{G}_{k^\prime}$ are the $p$-stabilizations of the newforms $f_k$ and $g_{k^\prime}$ with $U_p$ eigenvalues $\alpha_{f_k}$ and $\alpha_{g_{k^\prime}}$ respectively, then we have
    \begin{align*}
        L_p^{\rm geom}(\mathcal{F},\mathcal{G})(k,k^\prime,j)=\frac{E(f_k,g_{k^\prime},1+j)}{E(f_k)E^*(f_k)}\frac{j!(j-k^\prime -1)!}{\pi^{2j-k^\prime+1}(-1)^{k-k^\prime}2^{2j+2+k-k^\prime}\langle f_k,f_k\rangle}L(f_k,g_{k^\prime},1+j),
    \end{align*}
    where 
    \begin{align*}
        E(f_k)&:=\left(1-\frac{\beta_{f_k}}{p\alpha_{f_k}}\right),\\
        E^*(f_k)&:=\left(1-\frac{\beta_{f_k}}{\alpha_{f_k}}\right),\\
        E(f_k,g_{k^\prime},1+j)&:=\left(1-\frac{p^j}{\alpha_{f_k}\alpha_{g_{k^\prime}}}\right)\left(1-\frac{p^j}{\alpha_{f_k}\beta_{g_{k^\prime}}}\right)\left(1-\frac{\beta_{f_k}\alpha_{g_{k^\prime}}}{p^{1+j}}\right)\left(1-\frac{\beta_{f_k}\beta_{g_{k^\prime}}}{p^{1+j}}\right).
    \end{align*}
\end{theorem}

\begin{remark}
    We should note that Arlandini and Loeffler's $p$-adic $L$-function of $\Sym^2\mathcal{F}$ is defined only at integral weights in the cyclotomic variable, and the analogous $p$-adic $L$-function for the twists of $\Sym^2\mathcal{F}$ by Dirichlet characters $\psi$ is not defined. Therefore they do not have a factorization result for $p$-adic $L$-function attached to $\mathcal{F}\otimes\mathcal{F}\otimes\psi$ for a non-trivial character $\psi$. However (as I learned from Loeffler by a personal conversation at Iwasawa 2023 conference) the twisted version of the above theorem is expected to follow applying the same methods. Therefore we will state below the twisted version of the above theorem as a conjecture, and assume it as an input for our results.
\end{remark}
\begin{conjecture}[Arlandini-Loeffler, twisted version]\label{conj:arlandini-twisted}
    For a finite order Dirichlet character $\psi$, there exists a unique two variable meromorphic $p$-adic $L$-function \[L_p^{\rm imp}(\Sym^2(\mathbf{V}_\mathcal{F}),\psi)\] interpolating the critical values of $L^{\rm imp}(\Sym^2(f),\psi,s)$ such that the following factorization formula holds for all $(\sigma_1,\sigma_2)\in U\times\mathcal{W}$:
    \begin{align*}
        L_p^{\rm geom}(\mathcal{F},\mathcal{F}_\psi)(\sigma_1,\sigma_1,\sigma_2)=L_p^{\rm imp}(\Sym^2(\mathbf{V}_\mathcal{F}),\psi)(\sigma_1,\sigma_2)L_p(\varepsilon_f\psi,\sigma_2-\sigma_1-1),
    \end{align*}
    where $\mathcal{F}_\psi$ denotes the twist $\mathcal{F}\otimes\psi$ of $\mathcal{F}$ by $\psi$.
\end{conjecture}

\subsection{Beilinson-Flach Elements}\label{sect:BF-elts}

\par In this section, we introduce the Beilinson-Flach elements and Euler system coming from these elements, following \cite{LZ16} and \cite{BuyukbodukLeiVenkat2021Documenta}. We will not give the definition of these elements and how they are constructed (interested reader may look at op. cit. also the survey article \cite{trilogies} and \cite{BDR1,BDR2}, and \cite{LLZ} for the construction of Euler systems coming from these elements). However we present here the explicit reciprocity law (\cite[Theorem 7.1.5]{LZ16}) which relates these elements to Rankin-Selberg $p$-adic $L$-functions via Perrin-Riou's regulator (big logarithm) map introduced in \cite{Perrin-Riou95}. We will explain how to use this to bound the Bloch-Kato Selmer group attached to $\ad(V_f)$.

\par Let $\mathcal{F}$ and $\mathcal{G}$ be Coleman families over nice affinoid domains $U_\mathcal{F}$ and $U_\mathcal{G}$ passing through $f_{\alpha_f}$ and $g_{\alpha_g}$ respectively. Let $c>1$ be an integer coprime to $6pN_fN_g$, and let $m\geq 1$ be an integer. Then there exist cohomology classes \[_c\mathcal{BF}^{[\mathcal{F},\mathcal{G}]}_{m,1}\in H^1(\QQ(\mu(m)),\mathbf{V}_\mathcal{F}^*\hatotimes_E\mathbf{V}_\mathcal{G}^*\hatotimes_E\mathcal{H}(\Gamma)^\iota),\] called Beilinson-Flach elements, which interpolates the elements 
\[_c\mathcal{BF}^{[k,k^\prime,j]}_{m,1}\in H^1(\QQ(\mu(m)),V_{\mathcal{F}_{k}}^*\hatotimes V_{\mathcal{G}_{k^\prime}}^*(-j))\]
up to a nonzero factor (if $p^j\neq a_p(\mathcal{F}_k)a_p(\mathcal{G}_{k^\prime})$) for \[(k,k^\prime,j)\in U_\mathcal{F}\times_E U_\mathcal{G}\times_E\mathcal{W}\cap\NN^3\] with $0\leq j\leq {\rm min}(k,k^\prime)$. See \cite[Definition 3.5.2, Theorem 5.4.2]{LZ16} for precise details.

\par The following Explicit Reciprocity Law is due to \cite[Theorem 7.1.5]{LZ16}:
\begin{theorem}[Loeffler-Zerbes]\label{thm:explicit-reciprocity-law}
    For $p$-stabilizations $f_\alpha$ and $g_\alpha$ of $f$ and $g$ respectively, let $\mathcal{F}$ and $\mathcal{G}$ denote the Coleman Families passing through $f_\alpha$ and $g_\alpha$ respectively.
    Then we have
    \begin{align*}
        \left\langle\mathcal{L}(_c\mathcal{BF}^{[\mathcal{F},\mathcal{G}]}_{1,1}),\eta_\mathcal{F}\otimes\omega_\mathcal{G}\right\rangle=\nu_c(\mathbf{k}_\mathcal{F},\mathbf{k}_\mathcal{G},1+\mathbf{j})(-1)^{1+\mathbf{j}}\lambda_N(\mathcal{F})^{-1}L_p^{\rm geom}(\mathcal{F},\mathcal{G})(\mathbf{k}_\mathcal{F},\mathbf{k}_\mathcal{G},1+\mathbf{j}),
    \end{align*}
where \[\nu_c(\mathbf{k}_\mathcal{F},\mathbf{k}_\mathcal{G},\mathbf{j}):=(c^2-c^{-\mathbf{k}_\mathcal{F}-\mathbf{k}_\mathcal{G}+2\mathbf{j}-2}\varepsilon_\mathcal{F}(c)^{-1}\varepsilon_\mathcal{G}(c)^{-1}),\] $\mathbf{k}_\mathcal{G}$ and $\mathbf{k}_\mathcal{G}$ are weights of $\mathcal{F}$ and $\mathcal{G}$ respectively, $\lambda_N(\mathcal{F})$ is the Atkin-Lehner pseudo-eigenvalue (see \cite{Atkin} and \cite[Proposition 10.1.2]{KLZ17} for the definition and properties), $\eta_\mathcal{F}$ and $\omega_\mathcal{G}$ are given as in \cite[\S 6.4]{LZ16}, and $\mathcal{L}$ is Perrin-Riou's regulator map.
\end{theorem}
\begin{corollary}\label{cor:non-vanishing-of-BF}
    Let $\mathcal{F}$ and $\widetilde{\mathcal{F}}$ be Coleman families (over the same nice affinoid neighborhood) passing through the $p$-stabilizations $f_{\alpha_f}$ and $f_{\beta_f}$ of $f$, respectively. Suppose that the product \[L_p^{\rm imp}({\rm Sym}^2(\mathbf{V}_\mathcal{F}),\varepsilon^{-1}\psi^{-1})(k_f,2+k_f-j)L_p(\psi^{-1},1-j)\] of $p$-adic $L$-functions does not vanish for $j$ with $-k_f-1<j\leq 0$ and $\psi(-1)=(-1)^j$. Then there exists a $c$ coprime to $6pN_fN_\psi$ such that the specializations of \[_c\mathcal{BF}^{[\mathcal{F},\mathcal{F}^c_{\psi^{-1}}]}_{1,1}\ \textnormal{and}\ _c\mathcal{BF}^{[\mathcal{F},\widetilde{\mathcal{F}}^c_{\psi^{-1}}]}_{1,1}\] at $(k_f,k_f,1-j+k_f)$ are nonzero, where $-k_f-1<j\leq 0$ and $\psi(-1)=(-1)^j$.
\end{corollary}
\begin{proof}
    Using the factorization of Arlandini and Loeffler (twisted version, see Conjecture \ref{conj:arlandini-twisted}), we have \[L_p^{\rm geom}(\mathcal{F},\mathcal{F}^c_{\psi^{-1}})(k_f,k_f,2+k_f-j)=L_p^{\rm imp}({\rm Sym}^2(\mathbf{V}_\mathcal{F}),\varepsilon^{-1}\psi^{-1})(k_f,2+k_f-j)L_p(\psi^{-1},1-j).\]
    \par Note that setting $\mathcal{G}:=\mathcal{F}^c_{\psi^{-1}}$ gives us $\mathbf{k}_\mathcal{G}=\mathbf{k}_\mathcal{F}$ and $\varepsilon_\mathcal{G}=\varepsilon_\mathcal{F}^{-1}\psi^{-2}(c)$, thus \[\nu_c(\mathbf{k}_\mathcal{F},\mathbf{k}_\mathcal{G},1+\mathbf{j})=(c^2-c^{-2\mathbf{k}_\mathcal{F}+2\mathbf{j}}\psi(c)^2).\] Specializing $\nu_c(\mathbf{k}_\mathcal{F},\mathbf{k}_\mathcal{G},1+\mathbf{j})$ at $(\mathbf{k}_\mathcal{F},\mathbf{k}_\mathcal{G},\mathbf{j})=(k_f,k_f,1-j+k_f)$ gives 
    \[\nu_c(k_f,k_f,1-j+k_f)=(c^2-c^{2-2j}\psi(c)^2),\]
    thus we can choose a $c$ such that this is nonzero when $-k_f-1<j\leq 0$ and $\psi(-1)=(-1)^j$. Recall that those $j$'s are critical for $\ad(V_f)(\psi)$. Combining with our assumption, the RHS of the explicit reciprocity law (Theorem \ref{thm:explicit-reciprocity-law}) is nonzero modulo $(k_f,k_f,1-j+k_f)$. This shows that the specialization of $_c\mathcal{BF}^{[\mathcal{F},\mathcal{F}^c_{\psi^{-1}}]}_{1,1}$ at $(k_f,k_f,1-j+k_f)$, which is an element of $H^1(\QQ(\mu(m)),\Ad(V_f^*)(\psi+j))$, must also be nonzero.
    \par For $_c\mathcal{BF}^{[\mathcal{F},\widetilde{\mathcal{F}}^c_{\psi^{-1}}]}_{1,1}$, applying the explicit reciprocity law (with $\mathcal{G}=\widetilde{\mathcal{F}}^c_{\psi^{-1}}$) concludes the proof, as we have $L_p^{\rm geom}(\mathcal{F},\widetilde{\mathcal{F}}^c_{\psi^{-1}})=L_p^{\rm geom}(\mathcal{F},\mathcal{F}^c_{\psi^{-1}})$ (see \cite[Proposition 3.6.3]{BLLV}).
    
\end{proof}

\par For the sake of notational simplicity, let $V_\Ad$ and $V_\ad$ denote $\Ad(V_f)(\psi+j)$ and $\ad(V_f)(\psi+j)$, and let $T_\Ad$ and $T_\ad$ be $\mathcal{O}_E$-lattices inside $V_\Ad$ and $V_\ad$, respectively. Our aim is to show that the Bloch-Kato Selmer groups $H^1_f(\QQ,V_{\ad})$ and $H^1_f(\QQ,V_{\ad}^*(1))$ vanish. To be able to prove this result, we will utilize the Beilinson-Flach Euler system. 

\begin{lemma}\label{lemma:BF-Euler-System}
    Let $\psi+j$ be even, and $j$ be in the critical range of $\ad(V_f)(\psi)$. Suppose that the hypotheses of Corollary \ref{cor:non-vanishing-of-BF} hold true, and in addition suppose that the following hold true:
    \begin{itemize}
        \item $p\geq 7$ and $E\cong\Qp$,
        \item ${\rm Im}(G_\QQ\rightarrow {\rm Aut}(V_f))$ contains a conjugate of ${\rm SL}_2(\Zp)$,
        \item There exists $u\in(\ZZ/N_fN_\psi)^\times$ such that $\psi(u)\neq\pm 1\ ({\rm mod\ \mathfrak{P}})$ and $\varepsilon_f^{-1}\psi(u)$ is a square modulo $\mathfrak{P}$, where $\mathfrak{P}$ is the prime of $E$.
    \end{itemize}
    
    \par Let $\mathcal{P}_\psi$ be the set of primes $\ell\nmid pNN_\psi$ such that 
    \begin{itemize}
        \item $\ell\equiv$ 1 (${\rm mod}$ $p$),
        \item $T_\ad/(\Frob_\ell -1)$ is a free $\mathcal{O}$-module of rank $1$,
        \item $\Frob_\ell -1$ is bijective on $\mathcal{O}_E(\psi+j)$.
    \end{itemize}
    Let $\mathcal{R}_\psi$ be the set of square free products of primes in $\mathcal{P}_\psi$. For each $r \in \mathcal{R}_\psi$, there exist two cohomology classes 
    \[d_r^{\alpha,\alpha}, d_r^{\alpha,\beta} \in H^1_{(p)}(\QQ(r),V_{\ad}),\]
    where $\QQ(r)$ is the maximal $p$-extension of $\QQ$ in $\QQ(\mu(r))$, satisfying the following properties:
    \begin{enumerate}[(i)]
        \item There exists a constant $D$ independent of $r$ such that 
        \[Dd_r^{\alpha,\alpha}\,,\, Dd_r^{\alpha,\beta} \in H^1(\QQ(r),T_{\ad})\,.\]
        \item For $ r\ell \in \mathcal{R}_\psi$ and $\mu\in \{\alpha,\beta\}$ we have 
        \[\textup{cor}_{\QQ(r\ell)/\QQ(r)}\left(d_{r\ell}^{\alpha,\mu}\right)=P_\ell(\ell^{j-1}\Frob_\ell^{-1})\cdot d_r^{\alpha,\mu},\]
        where $P_\ell(X)$ is the Euler polynomial for $\ad(V_f^*)(\psi^{-1})$ at $\ell$ and $\Frob_\ell^{-1}$ is the geometric Frobenius.  
        \item The classes $d_1^{\alpha,\alpha}, d_1^{\alpha,\beta} \in H^1_{(p)}(\QQ,V_{\ad})$ are linearly independent.
    \end{enumerate}
\end{lemma}
\begin{proof}
    \par This theorem essentially \cite[Theorem 3.2.12]{BuyukbodukLeiVenkat2021Documenta}, \cite[Theorem 8.1.4]{LZ16}, and \cite[Theorem 5.3.3]{Loeffler_2018} adapted to our setting. We will follow the proofs of these theorems:
    Let $\widetilde{c}_r^{\alpha,\alpha}$ and $\widetilde{c}_r^{\alpha,\beta}$ be the specializations of Beilinson-Flach elements \[_c{\mathcal{BF}_{r,1}^{[\mathcal{F},\mathcal{F}^c_{\psi^{-1}}]}}\ \textnormal{and}\ _c{\mathcal{BF}_{r,1}^{[\mathcal{F},\widetilde{\mathcal{F}}^c_{\psi^{-1}}]}}\] at $(k_f,k_f,1-j+k_f)$, respectively, where $\mathcal{F}$, $\widetilde{\mathcal{F}}$, and $c$ are as in Corollary \ref{cor:non-vanishing-of-BF}. 
    \par Let $c_r^{\alpha,\alpha}$ and $c_r^{\alpha,\beta}$ be appropriate modifications of $\widetilde{c}_r^{\alpha,\alpha}$ and $\widetilde{c}_r^{\alpha,\beta}$ as in \cite[\S 7.3]{LLZ} so that they satisfy the Euler system norm relations for $\Ad(V_f)(\psi+j)$, i.e. for $r\ell\in\mathcal{R}_\psi$ we have \[\textup{cor}_{\QQ(r\ell)/\QQ(r)}\left(c_{r\ell}^{\alpha,\mu}\right)=\widetilde{P}_\ell(\ell^{j-1}\Frob_\ell^{-1})\cdot d_r^{\alpha,\mu},\] where $\widetilde{P}_\ell(X)$ is the Euler factor of $\Ad(V_f^*)(\psi^{-1})$.
    \par Now we need to modify $c_r^{\alpha,\mu}$ so that they satisfy the Euler system norm relations for $\ad(V_f)(\psi+j)$. Let $m=r\ell\in\mathcal{R}_\psi$. Note that the Euler factor $1-\ell^{j-1}\psi^{-1}(\ell)\sigma_\ell^{-1}$ for $\psi^{-1}+1-j$ is invertible in $\Zp[[\Gal(\QQ(m)\QQ_{\rm cyc}/\QQ)]]$ since $\ell\equiv 1\ ({\rm mod}\ p)$, where $\sigma_\ell\in\Gal(\QQ(m)\QQ_{\rm cyc}/\QQ)$ is the unique element acting trivially on $\QQ(\ell)$ and maps to $\Frob_\ell\in\Gal(\QQ(r)\QQ_{\rm cyc}/\QQ)$. Hence for any $m\in\mathcal{R}_\psi$ let
    \[d_m^{\alpha,\mu}:=\prod_{\ell\mid m}(1-\ell^{1-j}\psi^{-1}(\ell)\sigma_\ell^{-1})^{-1}c_m^{\alpha,\mu}.\] This proves (ii), and (i) holds true due to \cite[Theorem 8.1.4-(ii)]{LZ16}.
    \par Note that $d_1^{\alpha,\mu}\neq 0$ iff $\widetilde{c}_1^{\alpha,\mu}\neq 0$, which is true thanks to Corollary \ref{cor:non-vanishing-of-BF}. Hence by \cite[Theorem 8.1.4-(iv) and (v)]{LZ16}, the images of $d_r^{\alpha,\alpha}$ and $d_r^{\alpha,\beta}$ in $H^1_{(p)}(\QQ(r),V_{\Ad})$ are linearly independent. On the other hand, using the exact triangles in Theorem \ref{thm:exact-triangle-selmer-comp} for $p$-relaxed Selmer complexes (and noticing that the local conditions split w.r.t. the splitting of the representations), we obtain 
    \[
        H^1_{(p)}(\QQ(r),V_{\Ad})= H^1_{(p)}(\QQ(r),V_{\ad})\oplus H^1_{(p)}(\QQ(r),E(\psi+j)),
    \]
    and $H^1_{(p)}(\QQ,E(\psi+j))$ vanishes since $\psi+j$ is even and the validity of Leopoldt conjecture for abelian number fields holds true (see \cite[Proposition 3.2.3]{BuyukbodukLeiVenkat2021Documenta}).
    Thus we get \[H^1_{(p)}(\QQ(r),V_{\Ad})= H^1_{(p)}(\QQ(r),V_{\ad}),\] which concludes (iii).
\end{proof}

\begin{corollary}\label{cor:Vanishing-of-BK-Selmer-grp}
Suppose that $\psi(-1)=(-1)^j$ and $j\in(-1-k,0]$. If \[L_p^{{\rm imp}}(\Sym^2(f),\varepsilon_f^{-1}\psi^{-1},-j+k_f+2)L_p(\psi,1-j)\neq 0,\]then the Bloch-Kato Selmer groups \[H^1_f(\QQ,\ad(V_f)(\psi+j))\] and \[H^1_f(\QQ,\ad(V_f)(\psi^{-1}+1-j))\] vanish.
\end{corollary}
\begin{proof}
    First of all, we notice that Lemma \ref{lemma:BF-Euler-System} gives us a (horizontal) Euler system (see \cite[Chapter 2]{Ru00}), hence we have \[\dim_E H^1_{(p)}(\QQ,\ad(V_f)(\psi+j))=2\textnormal{ and }\dim_E H^1_{{\rm str}}(\QQ,\ad(V_f)(\psi^{-1}+1-j))=0,\] where the latter is thanks to \cite[Theorem 2.2(iii)]{Ru00} combined with the big image hypothesis, and the former is due to the Euler-Poincaré characteristic formula applied to $\RGamma_{(p)}(\QQ,\ad(V_f)(\psi+j))$ and the duality theorem (see Proposition \ref{prop:Duality-Affinoid-SelmerComp}) between this complex and $\RGamma_{\rm str}(\QQ,\ad(V_f)(\psi^{-1}+1-j))$.
    
    We will follow the arguments of \cite[Theorem 8.2.1]{LZ16} to conclude that $H^1_f(\QQ,\ad(V_f)(\psi+j))=0$. We have the following exact sequence:
    \begin{align*}
        0\longrightarrow H^1_f(\QQ,\ad(V_f)(\psi+j))\longrightarrow H^1_{(p)}(\QQ,\ad(V_f)(\psi+j))\\ \xrightarrow{{\rm loc}_p} H^1(\Qp,\ad(V_f)(\psi+j))/H^1_f(\Qp,\ad(V_f)(\psi+j)).
    \end{align*}
    By Lemma \ref{lemma:BF-Euler-System} we know that the images of the elements $d_1^{\alpha,\alpha}$ and $d_1^{\alpha,\beta}$ are linearly independent. Since \[\dim_E(H^1(\Qp,\ad(V_f)(\psi+j))/H^1_f(\Qp,\ad(V_f)(\psi+j)))=2,\] we conclude that the map ${\rm loc}_p$ above is surjective, hence the above exact sequence can be completed to a short exact sequence:
    \begin{align*}
        0\longrightarrow H^1_f(\QQ,\ad(V_f)(\psi+j))\longrightarrow H^1_{(p)}(\QQ,\ad(V_f)(\psi+j))\\ \xrightarrow{{\rm loc}_p} H^1(\Qp,V_{\ad})/H^1_f(\Qp,\ad(V_f)(\psi+j))\longrightarrow 0.
    \end{align*}
    This gives us that $\dim_E(H^1_f(\QQ,\ad(V_f)(\psi+j)))=0$, i.e. $H^1_f(\QQ,\ad(V_f)(\psi+j))=0$.
    The duality arguments in the proof of \cite[Corollary 8.3.2]{LZ16} gives $H^1_f(\QQ,\ad(V_f)(\psi^{-1}+1-j))=0$.
\end{proof}

\section{Algebraic Factorization}\label{sect:factorization}

\par This section is where our main theorem is stated and proved. First we will present the hypotheses that our objects of interest need to satisfy, then state our main result, and prove it using the theorems in previous sections. In the last subsection, we will relate the algebraic $p$-adic $L$-function (recall Corollary \ref{cor:algebraic-p-adic-L-function} for the definition) attached to certain twists of $\Ad(\mathbf{V}_\mathcal{F})$ to the algebraic $p$-adic $L$-function attached to certain twists of $\mathbf{V}_\mathcal{F}\hatotimes_E\mathbf{V}_\mathcal{F}^*$.

\subsection{Setup and Assumptions}\label{subsect:assumptions-main-thm}

\par Let $f\in S_{k_f+2}(\Gamma_1(N_f),\varepsilon_f)$ be a $p$-non-ordinary normalized newform such that $(N_f,p)=1$ with roots $\alpha:=\alpha_f$ and $\beta:=\beta_f$ of the Hecke polynomial $X^2-a_p(f)X+\varepsilon_f(p)p^{k_f+1}$ at $p$. Let $\psi$ be a finite order Dirichlet character with conductor $N_\psi$ coprime to $pN_f$. Let $\mathcal{F}$ be the unique Coleman family over the affinoid space $U=\Spm(A)\subseteq\mathcal{W}$ passing through $f_\alpha$, and $\mathbf{V}_\mathcal{F}$ be the big Galois representation attached to $\mathcal{F}$ as in Theorem \ref{thm:rep-coleman-F}.

\begin{remark}
    Throughout this section we will assume that $\psi+j$ is an even character, i.e. $\psi(-1)=(-1)^j$. When $\psi+j$ is odd, the character $\psi^{-1}+(1-j)$ is an even character, thus we can replace the roles of $\psi$ by $\psi^{-1}$ and $j$ by $1-j$ in our assumptions below. Indeed we will present our main result so that these two cases can be switch by duality.
\end{remark}

\par Recall that in the previous section we assumed the following hypotheses on our objects:

\begin{enumerate}[(i)]
    \item $\alpha\neq\beta$, \label{hyp:p-regularity}
    \item $-k_f-1<j\leq 0$ if $\psi(-1)=(-1)^j$,\label{hyp:j-critical}
    \item $L_p^{{\rm imp}}(\Sym^2(f),\varepsilon_f^{-1}\psi^{-1},-j+k_f+2)L_p(\psi,1-j)\neq 0$,\label{hyp:Lp-nonzero}
    \item $\psi(p)\neq p^j$, $\psi(p)\frac{\beta}{\alpha}\neq p^j$, $\frac{\alpha}{\beta}\psi(p)\neq p^{j-1}$,\label{hyp:no-local-zero}
    \item ${\rm Im}(G_\QQ\rightarrow {\rm Aut}(V_f))$ contains a conjugate of ${\rm SL}_2(\Zp)$,\label{hyp:big-img-I}
    \item The prime above $p$ in the Hecke field of $f$ has degree $1$, i.e. $E\cong \Qp$, \label{hyp:big-img-II}
    \item There exists $u\in(\ZZ/N_fN_\psi)^\times$ such that $\psi(u)\neq\pm 1\ ({\rm mod\ \mathfrak{P}})$ and $\varepsilon_f^{-1}\psi(u)$ is a square modulo $\mathfrak{P}$, where $\mathfrak{P}$ is the prime of $E$, \label{hyp:big-img-III}
    \item $p\geq7$. \label{hyp:big-img-IV}
    
\end{enumerate}

\par Let us briefly recall the hypotheses above. \ref{hyp:p-regularity} is the $p$-regularity condition which is conjectured to be always true. This is required to be able to decompose $\mathbf{D}_{{\rm cris}}(V_f^*)$ into rank $1$ distinct eigenspaces. \ref{hyp:j-critical} is needed so that $j$ is in the critical range of $\ad(V_f)(\psi+j)$ (in the sense of Deligne, see Definition \ref{def:critical-Deligne} and Definition \ref{def:critical-Deligne-p-adic}). \ref{hyp:Lp-nonzero} is imposed in the corollary of the explicit reciprocity law (Corollary \ref{cor:non-vanishing-of-BF}) so that the specializations of Beilinson-Flach elements are nonzero. \ref{hyp:no-local-zero} is the no local zero hypothesis of \cite[\S 8.1]{LZ16} adapted to our setting (see also the remark below). The hypotheses \ref{hyp:big-img-I} - \ref{hyp:big-img-IV} are for verifying the \textbf{big image hypothesis}, which is crucial for being able to apply the Euler system arguments (see \cite[Remark 2.4]{Ru00} for the big image hypothesis, and see \cite[\S 5.2]{LoefflerZerbesSym2J_Reine_Angew_Math_2019} and \cite{BuyukbodukLeiVenkat2021Documenta} for the verification of this hypothesis via those assumptions). We would like to remark that in \ref{hyp:big-img-III}, we replaced $\psi$ by $\varepsilon_f^{-1}\psi$ in Loeffler and Zerbes' assumption in \cite[5.1]{LoefflerZerbesSym2J_Reine_Angew_Math_2019}.

\begin{remark}
    Note that the whole point of the assumptions above is to prove that the Bloch-Kato Selmer groups $H^1_f(\QQ,\ad(V_f)(\psi+j))$ and $H^1_f(\QQ,V_\ad(\psi^{-1}+1-j))$ vanish. If one can show this vanishing result by relaxing these conditions, our factorization theorem will be true without requiring those hypotheses (if we also assume the hypotheses below) except (i) and (v) above. Those two assumptions are required so that we can relate the Selmer complex attached to $\ad(V_f)(\psi+j)$ to the Selmer group $H^1_f(\QQ,\ad(V_f)(\psi+j))$ via Proposition \ref{prop:Comparison-with-BK}.
\end{remark}

\par In addition, we will use the following hypotheses to prove the factorization theorems:

\begin{enumerate}[label=({\alph*})]
    \item For each prime ideal $\mathfrak{p}$ of $\Lambda_A[1/p]:=\Lambda[1/p]\hatotimes A$, we have \[(\ad(\mathbf{V}_\mathcal{F})(\psi+j)\otimes_A(\Lambda_A[1/p]^\iota)/\frakp)^{H_{\QQ,S}}=0,\] and \[(\ad(\mathbf{V}_\mathcal{F})(\psi^{-1}+1-j)\otimes_A(\Lambda_A[1/p]^\iota)/\frakp)^{H_{\QQ,S}}=0,\] where $H_{\QQ,S}:=\Gal(\QQ_S/\QQ_\infty).$\label{hyp:parf-1-2-ad0}
    \item $H^1(I_\ell,\mathbf{V}_\mathcal{F}\hatotimes_E \mathbf{V}_\mathcal{F}^*(\psi+j))$ and $H^1(I_\ell,\mathbf{V}_\mathcal{F}\hatotimes_E \mathbf{V}_\mathcal{F}^*(\psi^{-1}+1-j))$ are projective $A\hatotimes_E A$-modules.\label{hyp:tamagawa-assumption-FxF} 
    \item For each prime ideal $\mathfrak{p}$ of $\Lambda_{A\hatotimes_E A}[1/p]:=\Lambda[1/p]\hatotimes_E{A\hatotimes_E A}$, we have \[(\mathbf{V}_\mathcal{F}\hatotimes_E \mathbf{V}_\mathcal{F}^*(\psi+j)\otimes_{A\hatotimes_EA}(\Lambda_{A\hatotimes_EA}[1/p]^\iota)/\frakp)^{H_{\QQ,S}}=0,\] and \[(\mathbf{V}_\mathcal{F}\hatotimes_E \mathbf{V}_\mathcal{F}^*(\psi^{-1}+1-j)\otimes_{A\hatotimes_EA}(\Lambda_{A\hatotimes_EA}[1/p]^\iota)/\frakp)^{H_{\QQ,S}}=0.\]\label{hyp:parf-1-2-FxF} 
\end{enumerate}

\begin{remark}
    Note that conditions \ref{hyp:parf-1-2-ad0} and \ref{hyp:parf-1-2-FxF} imply conditions \ref{cond:h3Iw-vanishes} and \ref{cond:h0Iw-residual-vanishes} of Theorem \ref{thm:Perfectness-amplitude-1-2} for the representations considered in these conditions (we will explain it in more detail later). Also note that if for each specialization $x\in U=\Spm(A)$ the representation $\ad(V_{\mathcal{F}_x})$ is absolutely irreducible as a $\Gal(\QQ_S/\QQ_\infty)$ representation (which we know for all classical specializations if $U$ is sufficiently small, see \cite[Note 3.2.2]{LoefflerZerbesSym2J_Reine_Angew_Math_2019} and \cite[\S 1]{hida_tilouine_urban}), the condition \ref{hyp:parf-1-2-ad0} is automatically satisfied. Condition \ref{hyp:tamagawa-assumption-FxF} is required for the perfectness of the Selmer complexes attached to the twists of $\mathbf{V}_\mathcal{F}\hatotimes_E \mathbf{V}_\mathcal{F}^*$.
\end{remark}

\subsection{The ``Correct'' Local Conditions}\label{sect:local-cond}

\par Before defining the local conditions at $p$, for the sake of notational simplicity we set
\begin{align*}
    \mathbf{V}_{\Ad}&:=\Ad(\mathbf{V}_\mathcal{F})(\psi+j),\\
    \mathbf{V}_{\ad}&:=\ad(\mathbf{V}_\mathcal{F})(\psi+j),\\
    \mathbf{V}_{1}&:=A(\psi+j)\cong\mathcal{R}_A(\psi+j),
\end{align*}
and denote the specializations of these representations at $x_0\in U$ by $V_{\Ad}$, $V_{\ad}$ and $V_1$, i.e. 
\begin{align*}
    V_{\Ad}&:=\Ad(V_f)(\psi+j),\\
    V_{\ad}&:=\ad(V_f)(\psi+j),\\
    V_{1}&:=E(\psi+j)\cong\mathcal{R}_A(\psi+j).
\end{align*}

\par Recall that we assumed (without loss of generality) that $\psi+j$ is even. Then we have the following exact sequence of $G_\QQ$-representations which splits via ${\rm Tr}^*/2\otimes{\rm Id}$:
\begin{align}\label{eq:Trace-exact-sequence}
    0\longrightarrow \ad(V_f)(\psi+j) \longrightarrow \Ad(V_f)(\psi+j) \xrightarrow{\rm{Tr}\otimes{\rm Id}} E(\psi+j)\longrightarrow 0.  
\end{align}
Dually, the following exact sequence splits which splits via ${\rm Tr}/2\otimes{\rm Id}$:
\begin{align}\label{eq:Trace-dual-exact-sequence}
    0\longrightarrow V_1^*(1) \xrightarrow{\rm{Tr^*}\otimes{\rm Id}} V_\Ad^*(1) \longrightarrow V_\ad^*(1)\longrightarrow 0.  
\end{align}

\par Recall also from Section \ref{subsect:motives-newforms} that there exists a $(\varphi,\Gamma)$-submodule $\mathbf{D}\subseteq\mathbf{D}_{\rm{rig},A}^\dagger(\mathbf{V}_\mathcal{F})$  of rank $1$ over the Robba ring $\mathcal{R}_A$ (see Section \ref{subsect:robbaphigamma}), which satisfies the properties in Theorem \ref{thm:triangulation}. Let $D:=\mathbf{D}_{x_0}\subseteq \mathbf{D}_{\rm rig}^\dagger(V_f)$ denote its specialization at $x_0\in U$, which is a rank $1$ $(\phiGamma)$-module over the Robba ring $\mathcal{R}_E$ that is saturated in $\mathbf{D}_{\rm rig}^\dagger(V_f)$ (as $f_\alpha$ is non-$\theta$-critical).

\par We set \[D_{\Ad}:=\Hom_E(\mathbf{D}_{\rm rig}^\dagger(V_f),D)(\psi+j).\] Then $D_{\Ad}\subseteq\mathbf{D}_{\rm rig}^\dagger(V_\Ad)$ is a rank $2$ $(\phiGamma)$-submodule over $\mathcal{R}_E$. This is the most natural choice for the local condition at $p$ attached to $V_\Ad$. Indeed $\rank_{\mathcal{R}_E}D_\Ad=2=d^+(V_\Ad)$, hence $(V_\Ad,D_\Ad)$ satisfies the weakly Panchishkin condition of \cite[Definition 4.19]{danielebuyukboduksakamoto2023} (in \cite{Pottharst2013,pottharst2012cyclotomic}, the local conditions satisfying this condition is called (strict) ordinary), even though $V_\Ad$ is not Panchishkin (or critical in the sense of Definition \ref{def:critical-Deligne-p-adic}).

\par This choice of local condition $D_\Ad\subseteq\mathbf{D}_{\rm rig}^\dagger(V_\Ad)$ at $p$ induces the local conditions $D_\ad\subseteq\mathbf{D}_{\rm rig}^\dagger(V_\Ad)$ and $D_1\subseteq\mathbf{D}_{\rm rig}^\dagger(V_\Ad)$ at $p$ via the maps in the exact sequence \eqref{eq:Trace-exact-sequence}. More precisely we have 
\begin{align*}
    D_{\ad}&:=\Hom(\mathbf{D}_{\rm rig}^\dagger(V_f)/D,D)(\psi+j),\\
    D_{1}&:\Hom(D,D)(\psi+j)\cong\mathcal{R}_E(\psi+j).
\end{align*}

\par In this case $(V_\ad,D_\ad)$ and $(V_1,D_1)$ also satisfy the weakly Panchishkin condition, i.e. \[\rank_{\mathcal{R}_E}D_\ad=1=d^+(V_\ad)\] and \[\rank_{\mathcal{R}_E}D_1=1=d^+(V_1),\] even though only $V_\ad$ is critical (as we assumed $\psi+j$ is even and $-1-k_f<j\leq 0$) and $V_1$ is not.

\par We also define the local conditions attached to $\mathbf{V}_\Ad$, $\mathbf{V}_\ad$, and $\mathbf{V}_1$ in a similar manner, i.e.
\begin{align*}
    \mathbf{D}_{\Ad}&:=\Hom(\mathbf{D}_{{\rm rig},A}^\dagger(\mathbf{V}_\mathcal{F}),\mathbf{D})(\psi+j),\\
    \mathbf{D}_{\ad}&:=\Hom(\mathbf{D}_{{\rm rig},A}^\dagger(\mathbf{V}_\mathcal{F})/\mathbf{D},\mathbf{D})(\psi+j),\\
    \mathbf{D}_{1}&:=\Hom(\mathbf{D},\mathbf{D})(\psi+j)\cong\mathcal{R}_A(\psi+j).
\end{align*}
In this case we have $D_{\Ad}\cong(\mathbf{D}_{\Ad})_{x_0}$, $D_{\ad}\cong(\mathbf{D}_{\ad})_{x_0}$, and $D_{1}\cong(\mathbf{D}_{1})_{x_0}$.

\par Let $\mathbf{D}_\Ad^\perp$, $\mathbf{D}_\ad^\perp$, $\mathbf{D}_1^\perp$ be the (dual) local conditions (see Proposition \ref{prop:Duality-Affinoid-SelmerComp} for the definitions) attached to $\mathbf{V}_\Ad^*(1)$, $\mathbf{V}_\ad^*(1)$, $\mathbf{V}_1^*(1)$, respectively. Also let $D_\Ad^\perp\cong(\mathbf{D}_\Ad^\perp)_{x_0}$, $D_\ad^\perp\cong(\mathbf{D}_\ad^\perp)_{x_0}$, and $D_1^\perp\cong(\mathbf{D}_1^\perp)_{x_0}$ denote the local conditions attached to $V_\Ad^*(1)$, $V_\ad^*(1)$, $V_1^*(1)$, respectively. Note that the local conditions $D_\ad^\perp$ and $D_1^\perp$ are induced by $D_\Ad^\perp$ via the maps in the exact sequence \eqref{eq:Trace-dual-exact-sequence}.

\par Note that $(V_\Ad^*(1),D_\Ad^\perp)$, $(V_1^*(1),D_1^\perp)$, and $(V_1^*(1),D_1^\perp)$ are weakly Panchishkin, since we have
\begin{align*}
    \rank_{\mathcal{R}_E}D_\Ad^\perp&=2=d^+(V_\Ad^*(1)),\\
    \rank_{\mathcal{R}_E}D_\ad^\perp&=2=d^+(V_\ad^*(1)),\\
    \rank_{\mathcal{R}_E}D_1^\perp&=0=d^+(V_1^*(1)).
\end{align*}

\par By Theorem \ref{thm:exact-triangle-selmer-comp}, we have exact triangles of the following (Pottharst style) Selmer Complexes:
\begin{align}
    \RGamma_?(\mathbf{V}_{\ad},\mathbf{D}_{\ad}) \rightarrow &\RGamma_?(\mathbf{V}_{\Ad},\mathbf{D}_{\Ad}) \rightarrow \RGamma_?(\mathbf{V}_1,\mathbf{D}_1),\label{eq:exact-triangle-families}\\
    \RGamma_?(V_{\ad},\mathbf{D}_{\ad}) \rightarrow &\RGamma_?(V_{\Ad},\mathbf{D}_{\Ad}) \rightarrow \RGamma_?(V_1,\mathbf{D}_1),\label{eq:exact-triangle-specialization}\\
    \RGamma_?(\mathbf{V}_{\ad}^*(1),\mathbf{D}_{\ad}^\perp) \rightarrow &\RGamma_?(\mathbf{V}_{\Ad}^*(1),\mathbf{D}_{\Ad}^\perp) \rightarrow \RGamma_?(\mathbf{V}_1^*(1),\mathbf{D}_1^\perp),\label{eq:exact-triangle-dual-families}\\
    \RGamma_?(V_{\ad}^*(1),\mathbf{D}_{\ad}^\perp) \rightarrow &\RGamma_?(V_{\Ad}^*(1),\mathbf{D}_{\Ad}^\perp) \rightarrow \RGamma_?(V_1^*(1),\mathbf{D}_1^\perp),\label{eq:exact-triangle-dual-specialization}
\end{align}
where $?\in\{\emptyset,{\rm Iw}\}$.

\begin{proposition}\label{prop:Euler-Poincare-}
 Euler-Poincaré characteristics of the Selmer Complexes \[\RGamma(V_\Ad,D_\Ad),\ \RGamma(V_\ad,D_\ad),\ \RGamma(V_1,D_1),\] and \[\RGamma(V_\Ad^*(1),D_\Ad^\perp),\ \RGamma(V_\ad^*(1),D_\ad^\perp),\ \RGamma(V_1^*(1),D_1^\perp)\] are all $0$.
\end{proposition}
\begin{proof} 
Due to the computations of ranks above, all of those representations and local conditions appearing are weakly Panchishkin. By Corollary \ref{cor:pottharst-proposition-hi-ranks} the result follows.
\end{proof}
\subsection{Proof of the Main Results}\label{sect:main-result-proof}

\par In this section we state and prove our factorization theorems. First we will present our factorization result for the Selmer complexes attached to the representations $V_\Ad$, $V_\ad$, and $V_1$. After proving that result, we will use control theorems and the structures of coadmissible modules to generalize it to Selmer complexes attached to $\mathbf{V}_\Ad$, $\mathbf{V}_\ad$, and $\mathbf{V}_1$.

\begin{theorem}[Factorization for single newform]\label{Main-Theorem-Specialization} Under the hypotheses \ref{hyp:p-regularity}-\ref{hyp:big-img-IV} we have:
    \begin{enumerate}[(1)]
        \item For the data \[(V,D_V)\in\{(V_{\Ad},D_{\Ad}),(V_{\ad},D_{\ad}),(V_1,D_1)\},\] the Selmer complexes $\RGammaIw(V,D_V)$ and $\RGammaIw(V^*(1),D_V^\perp)$ lie in the category $\mathcal{D}_{\rm parf}^{[1,2]}(\mathcal{H})$.\label{enum:main-thm-spec-1}
        \item The cohomologies $H^i_\Iw(V,D_V)$ and $H^i_\Iw(V^*(1),D_V^\perp)$ vanish except $i\neq 2$. Therefore the exact triangles \eqref{eq:exact-triangle-specialization} and \eqref{eq:exact-triangle-dual-specialization} degenerate to the following short exact sequences:
        \begin{align}
            0 \rightarrow H^2_\Iw(V_{\ad},D_{\ad}) &\rightarrow H^2_\Iw(V_{\Ad},D_{\Ad})  \rightarrow H^2_\Iw(V_1,D_1) \rightarrow 0, \label{eq:SES-Main-Thm-even-specialization}\\
            0 \rightarrow H^2_\Iw(V_1^*(1),D_1^\perp) \rightarrow & H^2_\Iw(V_{\Ad}^*(1),D_{\Ad}^\perp)\rightarrow  H^2_\Iw(V_{\ad}^*(1),D_{\ad}^\perp)\rightarrow 0.\label{eq:SES-Main-Thm-odd-specialization}
        \end{align}\label{enum:main-thm-spec-2}
        \item For the pairs $(V,D_V)$ as above, $H^2_\Iw(V,D_V)$ and $H^2_\Iw(V^*(1),D_V^\perp)$ are torsion, and we have 
        \begin{align*}
            \det(\RGammaIw(V,D_V))&={\rm char}_\mathcal{H}H^2_\Iw(V,D_V),\\
            \det(\RGammaIw(V^*(1),D_V^\perp))&={\rm char}_\mathcal{H}H^2_\Iw(V^*(1),D_V^\perp),
        \end{align*}
            and we obtain the following factorizations:
        \begin{align}
            \det(\RGammaIw(V_{\Ad},D_{\Ad}))&=\det(\RGammaIw(V_{\ad},D_{\ad}))\hatotimes_{\mathcal{H}}\det(\RGammaIw(V_1,D_1)),\label{eq:Det-Factorization-Even-specialization}\\
            \det(\RGammaIw(V_{\Ad}^*(1),D_{\Ad}^\perp))&=\det(\RGammaIw(V_{\ad}^*(1),D_{\ad}^\perp))\hatotimes_\mathcal{H}\det(\RGammaIw(V_1^*(1),D_1^\perp)).\label{eq:Det-Factorization-Odd-specialization}
        \end{align}\label{enum:main-thm-spec-3}
    \end{enumerate}
\end{theorem}

\begin{lemma}[Vanishing of $H^3$ for specialization]\label{lemma:vanishing-H3-specialization}
    For $(V,D_V)$ as in Theorem \ref{Main-Theorem-Specialization} - \ref{enum:main-thm-spec-1} we have \[H^3_\Iw(V,D_V)=0=H^3_\Iw(V^*(1),D_V^\perp).\]
\end{lemma}
\begin{proof}
    \par First note that $\ad(V_f)$ is irreducible, see \cite[Note 3.2.2]{LoefflerZerbesSym2J_Reine_Angew_Math_2019} and \cite[\S 1]{hida_tilouine_urban}. Hence $V_\ad$ and $V_\ad^*(1)$ are also irreducible as $\Gal(\QQ_S/\QQ_\infty)$-representations, where $S$ is the finite set of primes dividing $pN_\psi N_f$. This and $\dim_E V_\ad$ being bigger than 1 imply that \[(V_\ad^*(1))^{\Gal(\QQ_S/\QQ_\infty)}=0,\] and \[V_\ad^{\Gal(\QQ_S/\QQ_\infty)}=0.\] By Corollary \ref{cor:vanishing_of_H3} we obtain \[H^3_\Iw(V_\ad,D_\ad)=0=H^3_\Iw(V_\ad^*(1),D_\ad^\perp).\]
    \par On the other hand, $V_1=E(\psi+j)$ and $V_1^*(1)=E(\psi^{-1}+1-j)$ are one dimensional (hence irreducible) and not characters of $\Gamma_0\cong\Gal(\QQ_\infty/\QQ)$ (since $\psi$ is nontrivial with conductor coprime to $p$), hence \[(V_1^*(1))^{\Gal(\QQ_S/\QQ_\infty)}=0,\] and \[V_1^{\Gal(\QQ_S/\QQ_\infty)}=0.\] Again by Corollary \ref{cor:vanishing_of_H3} we obtain \[H^3_\Iw(V_1,D_1)=0=H^3_\Iw(V_1^*(1),D_1^\perp).\]
    The exact triangles \eqref{eq:exact-triangle-specialization} and \eqref{eq:exact-triangle-dual-specialization} yields exact sequences
    \begin{align*}
        H^3_\Iw(V_\ad,D_\ad)\rightarrow &H^3_\Iw(V_\Ad,D_\Ad)\rightarrow H^3_\Iw(V_1,D_1),\\
        H^3_\Iw(V_1^*(1),D_1^\perp)\rightarrow  &H^3_\Iw(V_\Ad^*(1),D_\Ad^\perp) \rightarrow H^3_\Iw(V_\ad^*(1),D_\ad^\perp),
    \end{align*}
    where we computed the leftmost and rightmost terms as $0$. Thus we obtain \[H^3_\Iw(V_\Ad,D_\Ad)=0=H^3_\Iw(V_\Ad^*(1),D_\Ad^\perp).\]
\end{proof}
\begin{proof}[Proof of Theorem \ref{Main-Theorem-Specialization}]
    Let $(V,D_V)$ be as in Theorem \ref{Main-Theorem-Specialization} - \ref{enum:main-thm-spec-1}. Recall that in Lemma \ref{lemma:vanishing-H3-specialization} we calculated \[H^3_\Iw(V,D_V)=0=H^3_\Iw(V^*(1),D_V^\perp).\] By Remark \ref{remark:perfectness-1-2-using-duality}, we obtain \ref{enum:main-thm-spec-1}.

    \par The exact sequence in Proposition \ref{prop:Pottharst-Duality-Iwasawa-SelmerComp} (applied to the case $i=3$) gives an isomorphism between torsion part of $H^1_\Iw(V,D_V)$ and $H^3_\Iw(V^*(1),D^\perp)^\iota$, which we calculated as $0$ in Lemma \ref{lemma:vanishing-H3-specialization}. Thus we see that $H^1_\Iw(V,D_V)$ and $H^1_\Iw(V^*(1),D_V^\perp)$ must be free $\mathcal{H}$-modules, and of finite rank due to the structure theorem of coadmissible $\mathcal{H}$-modules (see Theorem \ref{structure-thm-pottharst}). \par When $V=V_1$, we know that \[H^1_\Iw(\QQ,V_1)=0\] when $\psi+j$ is an even character (see \cite[Corollary 3.2.4, Remark 3.2.5]{BuyukbodukLeiVenkat2021Documenta}). Since \[H^1_\Iw(V_1,D_1)\subseteq H^1_\Iw(\QQ,V_1),\] we conclude that \[H^1_\Iw(V_1,D_1)=0.\]
    \par Now consider the case $V=V_\ad$. By Proposition \ref{prop:control-theorems-pottharst} (The control theorem \eqref{eq:control-HA-A} applied to $A=E$) gives \[H^1_\Iw({V}_\ad,{D}_\ad)_{\Gamma_0}\subseteq H^1(V_\ad,D_\ad),\] and by Proposition \ref{prop:Comparison-with-BK} we have \[H^1(V_\ad,D_\ad)\cong H^1_f(\QQ,V_\ad)\cong0,\] where the latter isomorphism is due to Corollary \ref{cor:Vanishing-of-BK-Selmer-grp}. Thus we obtain \[H^1_\Iw({V}_\ad,{D}_\ad)_{\Gamma_0}=0.\] Since $H^1_\Iw({V}_\ad,{D}_\ad)$ is free $\mathcal{H}$-module of with finite rank, it is isomorphic to $\mathcal{H}^r$ for some $r\in\NN$, hence \[0\cong H^1_\Iw({V}_\ad,{D}_\ad)_{\Gamma_0}\cong E^r.\] Thus we obtain $r=0$, i.e. \[H^1_\Iw({V}_\ad,{D}_\ad)= 0.\]
    \par By Corollary \ref{cor:pottharst-proposition-hi-ranks}, we have \[\rank_\mathcal{H}H^2_\Iw(V_\ad,D_\ad)=0=\rank_\mathcal{H}H^2_\Iw(V_1,D_1).\] This proves the torsionness of these modules. Thus, the short exact sequence in Proposition \ref{prop:Pottharst-Duality-Iwasawa-SelmerComp} (applied to the case $i=1$) gives \[H^1_\Iw(V_\ad^*(1),D_\ad^\perp)=0=H^1_\Iw(V_1^*(1),D_1^\perp).\] Hence $H^2_\Iw(V_\ad^*(1),D_\ad^\perp)$ and $H^2_\Iw(V_1,D_1)$ are also torsion.
    \par The exact triangles \eqref{eq:exact-triangle-specialization} and \eqref{eq:exact-triangle-dual-specialization} yields exact sequences
    \begin{align*}
        H^1_\Iw(V_\ad,D_\ad)\rightarrow &H^1_\Iw(V_\Ad,D_\Ad)\rightarrow H^1_\Iw(V_1,D_1),\\
        H^1_\Iw(V_1^*(1),D_1^\perp)\rightarrow  &H^1_\Iw(V_\Ad^*(1),D_\Ad^\perp) \rightarrow H^1_\Iw(V_\ad^*(1),D_\ad^\perp),
    \end{align*}
    and by previous computations leftmost and rightmost terms in these exact sequences are $0$. Thus \[H^1_\Iw(V_\Ad,D_\Ad)=0=H^1_\Iw(V_\Ad^*(1),D_\Ad^\perp),\] and we obtain \ref{enum:main-thm-spec-2}. We also conclude that $H^2_\Iw(V_\Ad,D_\Ad)$ and $H^2_\Iw(V_\Ad^*(1),D_\Ad^*(1))$ are torsion, thus by Corollary \ref{cor:algebraic-p-adic-L-function} we get
    \begin{align*}
        \det(\RGammaIw(V,D_V))&={\rm char}_\mathcal{H}H^2_\Iw(V,D_V),\\
        \det(\RGammaIw(V^*(1),D_V^\perp))&={\rm char}_\mathcal{H}H^2_\Iw(V^*(1),D_V^\perp)
    \end{align*}
    for $(V,D_V)$ as before. Finally, the factorization of determinants in \ref{enum:main-thm-spec-3} follows formally from applying the determinant functor to the exact triangles \eqref{eq:exact-triangle-specialization} and \eqref{eq:exact-triangle-dual-specialization} (see \cite[Appendix A, Proposition 31]{gelfand-det}).
\end{proof}

\par Before stating our factorization result for Coleman families, we will need a lemma, which will make use of coadmissible modules and some ideas from algebraic geometry.

\begin{lemma}\label{lemma:coherent-sheaves-zero-fibers}
    Let $M$ be a coadmissible $\mathcal{H}_A$-module, and for each $x\in U=\Spm(A)$, let $\mathcal{H}_x:=\mathcal{H}_{E_x}$, where $E_x:=A/\frakm_x$ is the residue field of $x$. Suppose that $M\hatotimes_{{\rm sp}_{x,\Iw}}\mathcal{H}_x=0$ for each $x\in U$, where ${\rm sp}_{x,\Iw}:\mathcal{H}_A\rightarrow \mathcal{H}_x$ is the specialization map. Then $M=0$.
\end{lemma}
\begin{proof}
    \par Since $M$ is a coadmissible $\mathcal{H}_A$-module, $M=\varprojlim_{n}M_n$ for finitely generated $\mathcal{H}_n\hatotimes_E A$-modules $M_n$. We will show that $M_n=0$ for each $n\in\NN$.
    \par For an arbitrary $x\in U$, let 
    \begin{align}\label{eq:frakp-xn-definition}
        \frakp_{x,n}:=\ker((\mathcal{H}_A)_n\rightarrow(\mathcal{H}_x)_n). 
    \end{align}
    Recall that \[M_n\cong M\hatotimes_{\mathcal{H}_A}(\mathcal{H}_A)_n,\] thus \[M_n\hatotimes_{(\mathcal{H}_A)_n}(\mathcal{H}_x)_n\cong (M\hatotimes_{\mathcal{H}_A}\mathcal{H}_x)\hatotimes_{\mathcal{H}_A}(\mathcal{H}_A)_n.\] By our running assumption, we have \[M_n\hatotimes_{(\mathcal{H}_A)_n}(\mathcal{H}_x)_n\cong M_n\hatotimes_{(\mathcal{H}_A)_n}((\mathcal{H}_A)_n)/\frakp_{x,n}=0.\]

    \par Let $n\in\NN$ and a maximal ideal $\frakm$ of $(\mathcal{H}_A)_n=\mathcal{H}_n\hatotimes A$ be arbitrary but fixed. We claim that there exists an $x\in U$ such that $\frakm\supseteq \frakp_{x,n}$. Set $Y:=W_n\times_E U=\Spm(\mathcal{H}_n\hatotimes_E A)$. Denote by $\pi:Y\twoheadrightarrow U$ the natural projection, which corresponds to the morphism $A\hookrightarrow\mathcal{H}_n\hatotimes A$ of affinoid algebras sending $a\in A$ to $1\otimes a$. Set $x:=\pi(\frakm)$ (where $\frakm\in Y=\Spm(\mathcal{H}_n\hatotimes_E A)$). 
    \par We claim that the fiber $Y_x:=W_n\times_E E_x\cong\pi^{-1}(x)$ of $x$ is the vanishing locus $V(\frakp_{x,n})$ of the prime ideal $\frakp_{x,n}$. Note that $Y_x=\Spm(\mathcal{H}_n\hatotimes_E E_x)$ and $\mathcal{H}_n\hatotimes_E E_x\cong ({\mathcal{H}_x})_n$. Hence \eqref{eq:frakp-xn-definition} immediately implies our claim. 
    \par By construction $\frakm\in \pi^{-1}(x)=Y_x=V(\frakp_{x,n})$, hence $\frakm\supseteq\frakp_{x,n}$.
    \par Thus for all $\frakm\in Y$ there is a point $x\in U$ with $\frakm\supseteq\frakp_{x,n}$, and recall that we already obtained \[M_n\hatotimes_{(\mathcal{H}_A)_n}((\mathcal{H}_A)_n)/\frakp_{x,n}=0.\] These two results imply \[M_n\hatotimes_{(\mathcal{H}_A)_n}((\mathcal{H}_A)_n)/\frakm=0.\] By Nakayama lemma (as $M_n$ is a finitely generated $(\mathcal{H}_A)_n$-module) $(M_n)_\frakm=0$. Since $\frakm$ was arbitrary we get $M_n=0$, thus $M=0$.
\end{proof}

\begin{theorem}[Factorization for Coleman families]\label{Main-Theorem}
Under our hypotheses in Section \ref{subsect:assumptions-main-thm}, we have:
\begin{enumerate}[(1)]
    \item For the data \[(\mathbf{V},\mathbf{D}_\mathbf{V})\in\{(\mathbf{V}_{\Ad},\mathbf{D}_{\Ad}),(\mathbf{V}_{\ad},\mathbf{D}_{\ad}),(\mathbf{V}_1,\mathbf{D}_1)\},\] the Selmer complexes $\RGammaIw(\mathbf{V},\mathbf{D}_\mathbf{V})$ and $\RGammaIw(\mathbf{V}^*(1),\mathbf{D}_\mathbf{V}^\perp)$ lie in the category $\mathcal{D}_{\rm parf}^{[1,2]}(\mathcal{H}_A)$. \label{enum:main-thm-fams-1}

    \item The cohomologies $H^i_\Iw(\mathbf{V},\mathbf{D}_\mathbf{V})$ and $H^i_\Iw(\mathbf{V}^*(1),\mathbf{D}_\mathbf{V}^\perp)$ vanish except $i\neq 2$. Therefore the exact triangles \eqref{eq:exact-triangle-families} and \eqref{eq:exact-triangle-dual-families} degenerate to the following short exact sequences:
    \begin{align}
        0 \rightarrow H^2_\Iw(\mathbf{V}_{\ad},\mathbf{D}_{\ad}) &\rightarrow H^2_\Iw(\mathbf{V}_{\Ad},\mathbf{D}_{\Ad})  \rightarrow H^2_\Iw(\mathbf{V}_1,\mathbf{D}_1) \rightarrow 0, \label{eq:SES-Main-Thm-even}\\
        0 \rightarrow H^2_\Iw(\mathbf{V}_1^*(1),\mathbf{D}_1^\perp) \rightarrow & H^2_\Iw(\mathbf{V}_{\Ad}^*(1),\mathbf{D}_{\Ad}^\perp)\rightarrow  H^2_\Iw(\mathbf{V}_{\ad}^*(1),\mathbf{D}_{\ad}^\perp)\rightarrow 0.\label{eq:SES-Main-Thm-odd}
    \end{align}\label{enum:main-thm-fams-2}
    
    \item For $(V,D_V)$ as in Theorem \ref{Main-Theorem-Specialization} - \ref{enum:main-thm-spec-1} with $\mathbf{V}_{x_0}=V$ and $(\mathbf{D}_\mathbf{V})_{x_0}=D_V$, we have
    \begin{align*}
        \det(\RGammaIw(\mathbf{V},\mathbf{D}_\mathbf{V}))\hatotimes_{\mathcal{H}_A}\mathcal{H}_A/{\frakp_{x_0}}&\cong {\rm char}_{\mathcal{H}}H^2_\Iw(V,D_V),\\
        \det(\RGammaIw(\mathbf{V}^*(1),\mathbf{D}_\mathbf{V}^\perp))\hatotimes_{\mathcal{H}_A}\mathcal{H}_A/{\frakp_{x_0}}&\cong{\rm char}_{\mathcal{H}}H^2_\Iw(V^*(1),D_V^\perp),
    \end{align*}
        where $\frakp_{x}$ is the kernel of the map ${{\rm sp}_{{x},\Iw}}:\mathcal{H}_A\rightarrow\mathcal{H}$ induced by the specialization at $x\in U=\Spm(A)$. We have the following factorizations:
    \begin{align*}
        \det(\RGammaIw(\mathbf{V}_{\Ad},\mathbf{D}_{\Ad}))&=\det(\RGammaIw(\mathbf{V}_{\ad},\mathbf{D}_{\ad}))\hatotimes_{\mathcal{H}_A}\det(\RGammaIw(\mathbf{V}_1,\mathbf{D}_1)),\\
        \det(\RGammaIw(\mathbf{V}_{\Ad}^*(1),\mathbf{D}_{\Ad}^\perp))&=\det(\RGammaIw(\mathbf{V}_{\ad}^*(1),\mathbf{D}_{\ad}^\perp))\hatotimes_{\mathcal{H}_A}\det(\RGammaIw(\mathbf{V}_1^*(1),\mathbf{D}_1^\perp)).
    \end{align*} \label{enum:main-thm-fams-3}
\end{enumerate}
\end{theorem}
\begin{proof}
    \par First note that for any $\mathbf{V}$ as above, $H^1(I_\ell,\mathbf{V})$ is projective $A$-module since $U$ is nice affinoid (hence $A$ is an Euclidean domain). This verifies the condition \ref{cond:tamagawa-proj} in Theorem \ref{thm:Perfectness-amplitude-1-2}. Let us show that 
    \begin{align} \label{eq:h3_Iw-0-fams}
        H^3_\Iw(\mathbf{V},\mathbf{D}_\mathbf{V})=0=H^3_\Iw(\mathbf{V}^*(1),\mathbf{D}_\mathbf{V}^\perp).    
    \end{align}    
    By the control theorem \eqref{eq:control-HA-H}, for any $x\in U$ we get \[H^3_\Iw(\mathbf{V},\mathbf{D}_\mathbf{V})\hatotimes_{\mathcal{H}_A}\mathcal{H}_A/\frakp_x\cong H^3_\Iw(\mathbf{V}_x,(\mathbf{D}_\mathbf{V})_x)\] and \[H^3_\Iw(\mathbf{V}^*(1),\mathbf{D}_\mathbf{V}^\perp)\hatotimes_{\mathcal{H}_A}\mathcal{H}_A/\frakp_x\cong H^3_\Iw(\mathbf{V}_x^*(1),(\mathbf{D}_\mathbf{V})_x^\perp).\] When $\mathbf{V}=\mathbf{V}_\ad$, using our running assumption \ref{hyp:parf-1-2-ad0} in Section \ref{subsect:assumptions-main-thm} and applying the arguments of Lemma \ref{lemma:vanishing-H3-specialization} we get \[H^3_\Iw(\mathbf{V}_x,(\mathbf{D}_\mathbf{V})_x)=0=H^3_\Iw(\mathbf{V}_x^*(1),(\mathbf{D}_\mathbf{V})_x^\perp).\] Since this holds for all $x\in U$, by Lemma \ref{lemma:coherent-sheaves-zero-fibers} we get \[H^3_\Iw(\mathbf{V}_\ad,\mathbf{D}_\ad)=0=H^3_\Iw(\mathbf{V}_\ad^*(1),\mathbf{D}_\ad^\perp).\] When $\mathbf{V}=\mathbf{V}_1$, for any $x\in U$ we obtain $\mathbf{V}_x\cong E_x(\psi+j)$, and we can apply Lemma \ref{lemma:vanishing-H3-specialization} directly to conclude \eqref{eq:h3_Iw-0-fams}. When $\mathbf{V}=\mathbf{V}_\Ad$, \eqref{eq:h3_Iw-0-fams} follows from applying $H^3$ to the exact triangles \eqref{eq:exact-triangle-families} and \eqref{eq:exact-triangle-dual-families}. These computations verify \ref{cond:h3Iw-vanishes} in Theorem \ref{thm:Perfectness-amplitude-1-2}.
    \par Our running assumption \ref{hyp:parf-1-2-ad0} in \ref{subsect:assumptions-main-thm} immediately implies the condition \ref{cond:h0Iw-residual-vanishes} of Theorem \ref{thm:Perfectness-amplitude-1-2} as $H_{\QQ,S}\subseteq G_{\QQ,S}$. Hence we can apply Theorem \ref{thm:Perfectness-amplitude-1-2} to conclude \ref{enum:main-thm-fams-1}.

    \par Now we will show that 
    \begin{align} \label{eq:h1_Iw-0-fams}
        H^1_\Iw(\mathbf{V},\mathbf{D}_\mathbf{V})=0=H^1_\Iw(\mathbf{V}^*(1),\mathbf{D}_\mathbf{V}^\perp).    
    \end{align}   
    By control theorem \eqref{eq:control-HA-H} for any $x\in U$ we have \[H^1_\Iw(\mathbf{V},\mathbf{D}_\mathbf{V})\hatotimes_{\mathcal{H}_A}\mathcal{H}_A/\frakp_x\subseteq H^1_\Iw(\mathbf{V}_x,(\mathbf{D}_\mathbf{V})_x).\] When $\mathbf{V}=\mathbf{V}_1$ (or $\mathbf{V}_1^*(1)$) the latter (hence the former) vanishes thanks to Theorem \ref{Main-Theorem-Specialization}. Thus by Lemma \ref{lemma:coherent-sheaves-zero-fibers} we can conclude \eqref{eq:h1_Iw-0-fams}. When $\mathbf{V}=\mathbf{V}_\ad$ (or $\mathbf{V}=\mathbf{V}_\ad^*(1)$), we should use the control theorem \eqref{eq:control-HA-A} instead. We obtain \[H^1(\mathbf{V},\mathbf{D}_\mathbf{V})_{x_0}\subseteq H^1(\mathbf{V}_{x_0},(\mathbf{D}_\mathbf{V})_{x_0})=0,\] where the latter is due to Theorem \ref{Main-Theorem-Specialization}. Note that $\RGamma(\mathbf{V},\mathbf{D}_\mathbf{V})\in\mathscr{D}_{\rm parf}^{[1,2]}(A)$, hence it is quasi isomorphic to a complex consisting of finite projective $A$-modules concentrated at degrees $[1,2]$, which implies \[\RGamma(\mathbf{V},\mathbf{D}_\mathbf{V})\simeq \left[A^{r}\rightarrow A^{r}\right],\] since every projective module over a PID is free. Then $H^1(\mathbf{V},\mathbf{D}_\mathbf{V})\subseteq A^{r}$ is also free of some finite rank, and $H^1(\mathbf{V},\mathbf{D}_\mathbf{V})_{x_0}=0$. Thus we conclude that the rank must be zero, i.e. $H^1(\mathbf{V},\mathbf{D}_\mathbf{V})=0$. Thus we obtain \[H^1_\Iw(\mathbf{V},\mathbf{D}_\mathbf{V})_{\Gamma_0}\subseteq H^1(\mathbf{V},\mathbf{D}_\mathbf{V})=0.\]
    
    \par Now let $x\in U$ be arbitrary. By the control theorem \eqref{eq:control-HA-H} we get \[H^1_\Iw(\mathbf{V},\mathbf{D}_\mathbf{V})\hatotimes_{\mathcal{H}_A}\mathcal{H}_A/\frakp_x\subseteq H^1_\Iw(\mathbf{V}_x,(\mathbf{D}_\mathbf{V})_x),\] and note that $H^1_\Iw(\mathbf{V}_x,(\mathbf{D}_\mathbf{V})_x)$ is a free $\mathcal{H}$-module (by the structure theorem for coadmissible $\mathcal{H}$-modules and the fact that $H^3_\Iw(\mathbf{V}_x^*(1),(\mathbf{D}_\mathbf{V})_x^\perp)=0$). Therefore $H^1_\Iw(\mathbf{V},\mathbf{D}_\mathbf{V})\hatotimes_{\mathcal{H}_A}\mathcal{H}_A/\frakp_x$ is also a free $\mathcal{H}$-module with some finite rank $r\geq 0$, i.e. \[H^1_\Iw(\mathbf{V},\mathbf{D}_\mathbf{V})\hatotimes_{\mathcal{H}_A}\mathcal{H}_A/\frakp_x\cong\mathcal{H}^r.\] On the other hand, \[(H^1_\Iw(\mathbf{V},\mathbf{D}_\mathbf{V})\hatotimes_{\mathcal{H}_A}\mathcal{H}_A/\frakp_x)_{\Gamma_0}\cong (H^1_\Iw(\mathbf{V},\mathbf{D}_\mathbf{V})_{\Gamma_0})/\hatotimes_{A}A/\frakm_x,\] where $\frakm_x\subseteq A$ is the maximal ideal corresponding to $x\in U$. We know from the previous paragraph that the latter vanishes, hence \[(H^1_\Iw(\mathbf{V},\mathbf{D}_\mathbf{V})\hatotimes_{\mathcal{H}_A}\mathcal{H}_A/\frakp_x)_{\Gamma_0}=0.\] This gives us that the rank of the (free) $\mathcal{H}$-module $H^1_\Iw(\mathbf{V},\mathbf{D}_\mathbf{V})\hatotimes_{\mathcal{H}_A}\mathcal{H}_A/\frakp_x$ must be zero, hence \[H^1_\Iw(\mathbf{V},\mathbf{D}_\mathbf{V})\hatotimes_{\mathcal{H}_A}\mathcal{H}_A/\frakp_x=0.\] Since $x\in U$ was arbitrary, it follows from Lemma \ref{lemma:coherent-sheaves-zero-fibers} that \[H^1_\Iw(\mathbf{V},\mathbf{D}_\mathbf{V})=0.\]

    \par For $\mathbf{V}=\mathbf{V}_\Ad$ or $\mathbf{V}=\mathbf{V}_\Ad^*(1)$, we obtain \[H^1_\Iw(\mathbf{V},\mathbf{D}_\mathbf{V})=0\] by applying $H^1$ to the exact triangles \eqref{eq:exact-triangle-families} and \eqref{eq:exact-triangle-dual-families}. 
    \par These computations show that the cohomologies of $\RGammaIw(\mathbf{V},\mathbf{D}_\mathbf{V})$ and $\RGammaIw(\mathbf{V}^*(1),\mathbf{D}_\mathbf{V}^\perp)$ vanishes except $i=2$, thus the long exact sequence induced by the same exact triangles degenerate to the short exact sequences in \ref{enum:main-thm-fams-2}.

    \par By the control theorem \eqref{eq:control-HA-H} we get 
    \begin{align*}
        \RGammaIw(\mathbf{V},\mathbf{D}_\mathbf{V})\hatotimes_{\mathcal{H}_A}\mathcal{H}_A/{\frakp_{x_0}}&\cong \RGammaIw(V,D_V),\\
        \RGammaIw(\mathbf{V}^*(1),\mathbf{D}_\mathbf{V}^\perp)\hatotimes_{\mathcal{H}_A}\mathcal{H}_A/{\frakp_{x_0}}&\cong\RGammaIw(V^*(1),D_V^\perp),
    \end{align*}
    hence applying the determinant functor and the fact that \[\det(\RGammaIw(V,D_V))={\rm char}_\mathcal{H}(H^2_\Iw(V,D_V))\] and \[\det(\RGammaIw(V^*(1),D_V^\perp))={\rm char}_\mathcal{H}(H^2_\Iw(V^*(1),D_V^\perp))\] (see the last paragraph of the proof of Theorem \ref{Main-Theorem-Specialization}), we conclude 
    \begin{align*}
        \det(\RGammaIw(\mathbf{V},\mathbf{D}_\mathbf{V}))\hatotimes_{\mathcal{H}_A}\mathcal{H}_A/{\frakp_{x_0}}&\cong {\rm char}_{\mathcal{H}}H^2_\Iw(V,D_V),\\
        \det(\RGammaIw(\mathbf{V}^*(1),\mathbf{D}_\mathbf{V}^\perp))\hatotimes_{\mathcal{H}_A}\mathcal{H}_A/{\frakp_{x_0}}&\cong{\rm char}_{\mathcal{H}}H^2_\Iw(V^*(1),D_V^\perp).
    \end{align*}
    Finally, the factorizations follow from applying the determinant functor to the exact triangles \eqref{eq:exact-triangle-families} and \eqref{eq:exact-triangle-dual-families}.
\end{proof}

\subsection{Comparison of $3$-variable and $2$-variable Algebraic $p$-adic $L$-functions}\label{sect:3-var-2-var}

\par In this subsection, we will give a partial result on the comparison of the algebraic $p$-adic $L$-functions attached to (certain twists of) $\mathbf{V}_\mathcal{F}\hatotimes_E\mathbf{V}^*$ and $\Ad(\mathbf{V}_\mathcal{F})$ (see Corollary \ref{cor:algebraic-p-adic-L-function}). For the sake of notational simplicity we will denote \[\mathbf{V}_3:=\mathbf{V}_\mathcal{F}\hatotimes_E\mathbf{V}^*_\mathcal{F}(\psi+j),\] and \[\mathbf{D}_3:=\mathbf{D}\hatotimes_{\mathcal{R}_E}\mathbf{D}_{{\rm rig},A}^\dagger(\mathbf{V}_\mathcal{F})^*.\] The notation chosen refers to the fact that it should be the algebraic counterpart of the $3$-variable $p$-adic $L$-function $L_p(\mathcal{F},\mathcal{F}^c,\psi+j+\mathbf{k}_f+1)$.

\par Let $B:=A\hatotimes_E A$. By our running assumptions \ref{hyp:parf-1-2-FxF} and \ref{hyp:tamagawa-assumption-FxF}, we conclude by Theorem \ref{thm:Perfectness-amplitude-1-2} that \[\RGammaIw(\mathbf{V}_3,\mathbf{D}_3),\ \RGammaIw(\mathbf{V}_3^*(1),\mathbf{D}_3^\perp)\in\mathscr{D}_{\rm parf}^{[1,2]}(\mathcal{H}_{B}).\] To be more precise, the assumptions \ref{cond:tamagawa-proj} and \ref{cond:h0Iw-residual-vanishes} of Theorem \ref{thm:Perfectness-amplitude-1-2} are automatically satisfied, meanwhile for the condition \ref{cond:h3Iw-vanishes}, we can argue as in the proof of Theorem \ref{Main-Theorem} and use Lemma \ref{lemma:coherent-sheaves-zero-fibers} to conclude that \[H^3_\Iw(\mathbf{V}_3,\mathbf{D}_3)=0=H^3_\Iw(\mathbf{V}_3^*(1),\mathbf{D}_3^\perp).\]

\par There exists a map $\pi:=\Delta^*:B\rightarrow A$ which induces $\widetilde{\pi}:\mathcal{H}_B\rightarrow\mathcal{H}_A$. Notice that $\mathbf{V}_3\otimes_{\pi}A\cong\mathbf{V}_\Ad$ and $\mathbf{D}_3\otimes_{\mathcal{R}_B}\mathcal{R}_A\cong\mathbf{D}_\Ad$.

\par We can use the control theorem of \cite[Theorem 1.12 - (3)]{Pottharst2013} for Selmer complexes to deduce that 
\[
\RGammaIw(\mathbf{V},\mathbf{D}_3)\otimes_{\widetilde{\pi}}^\mathbf{L}\mathcal{H}_A\simeq \RGammaIw(\mathbf{V}_\Ad,\mathbf{D}_\Ad),  
\] hence \[\det(\RGammaIw(\mathbf{V}_3,\mathbf{D}_3))\otimes_{\widetilde{\pi}}^\mathbf{L}\mathcal{H}_A\simeq \det(\RGammaIw(\mathbf{V}_\Ad,\mathbf{D}_\Ad)).\] Similarly we get \[\det(\RGammaIw(\mathbf{V}_3^*(1),\mathbf{D}_3^\perp))\otimes_{\widetilde{\pi}}^\mathbf{L}\mathcal{H}_A\simeq \det(\RGammaIw(\mathbf{V}_\Ad^*(1),\mathbf{D}_\Ad)^\perp).\]

\par Let $\frakp:=\ker(\pi)$ and $\widetilde{\frakp}:=\ker(\widetilde{\pi})$. Note that $\frakp$ and $\widetilde{\frakp}$ are height-$1$ prime ideals of $B$ and $\mathcal{H}_B$ respectively. 
\begin{theorem}\label{thm:comparison-3-Ad}
       Suppose that $H^1_\Iw(\mathbf{V}_3,\mathbf{D}_3)$ is a finitely generated $\mathcal{H}_B$-module. There exists an $f\in \mathcal{H}_B$ with $f\equiv 1\ ({\rm mod}\ \widetilde{\frakp})$ such that
   \begin{align*}
       H^1(\mathbf{V}_3^\prime,\mathbf{D}_3^\prime)=0,\\
       H^1((\mathbf{V}_3^\prime)^*(1),(\mathbf{D}_3^\prime)^\perp)=0,
   \end{align*}
   where $\mathbf{V}_3^\prime:=\mathbf{V}_3\hatotimes_B (\mathcal{H}_B)_f^\iota$ and $\mathbf{D}_3^\prime:=\mathbf{D}_3\hatotimes_{\mathcal{R}_B}{(\mathcal{H}_B)}_f^\iota$. Moreover, we have the following relationships between the algebraic $p$-adic $L$-functions:
    \begin{align*}
        \det(\RGamma(\mathbf{V}_3^\prime,\mathbf{D}_3^\prime))\otimes_{(\mathcal{H}_B)_f}\mathcal{H}_A&\cong\det(\RGammaIw(\mathbf{V}_\Ad,\mathbf{D}_\Ad)),\\
        \det(\RGamma((\mathbf{V}_3^\prime)^*(1),(\mathbf{D}_3^\prime)^\perp))\otimes_{(\mathcal{H}_B)_f}\mathcal{H}_A&\cong\det(\RGammaIw(\mathbf{V}_\Ad^*(1),\mathbf{D}_\Ad^\perp)).
    \end{align*}
\end{theorem}
\begin{proof}
    \par From now on we will prove the theorem for $\mathbf{V}_3$ and $\mathbf{V}_\Ad$, the exact same arguments will prove the results for $(\mathbf{V}_3)^*(1)$ and $\mathbf{V}_\Ad^*(1)$. By control theorem \cite[Theorem 1.12]{Pottharst2013} we have \[\RGamma(\mathbf{V}_3,\mathbf{D}_3)\otimes^\mathbf{L}_{\widetilde{\pi}} \mathcal{H}_A\cong\RGamma(\mathbf{V}_\Ad,\mathbf{D}_\Ad).\] Since $\widetilde{\frakp}$ is a height $1$ prime, the higher Tor terms $\Tor_i^{\mathcal{H}_B}(\blank,\mathcal{H}_A)$ vanishes for $i\geq 2$, and the corresponding spectral sequence degenerates into the following information:
    \begin{align*}
        &H^1_\Iw(\mathbf{V}_3,\mathbf{D}_3)[\widetilde{\frakp}]:=0,\\
        0\rightarrow H^1_\Iw(\mathbf{V}_3,\mathbf{D}_3)\hatotimes_{\mathcal{H}_B}\mathcal{H}_B/\widetilde{\frakp}\rightarrow &H^1_\Iw(\mathbf{V}_\Ad,\mathbf{D}_\Ad)\rightarrow H^2_\Iw(\mathbf{V}_3,\mathbf{D}_3)[\widetilde{\frakp}] \rightarrow 0,\\
        H^2_\Iw(\mathbf{V}_3,\mathbf{D}_3)\hatotimes_{\mathcal{H}_B}\mathcal{H}_B/\widetilde{\frakp}&\cong H^2_\Iw(\mathbf{V}_\Ad,\mathbf{D}_\Ad).
    \end{align*}
    \par Thus we have $H^1_\Iw(\mathbf{V}_3,\mathbf{D}_3)\hatotimes_{\mathcal{H}_B}\mathcal{H}_B/\widetilde{\frakp}=0$. By our finiteness assumption we can use Nakayama lemma, thus there exists an $f\in \mathcal{H}_B$, $f\in 1+\widetilde{\frakp}$ such that $fH^1_\Iw(\mathbf{V}_3,\mathbf{D}_3)=0$. This implies that \[H^1_\Iw(\mathbf{V}_3,\mathbf{D}_3)\hatotimes_{\mathcal{H}_B} (\mathcal{H}_B)_f=0.\] Since $(\mathcal{H}_B)_f$ is flat over $\mathcal{H}_B$, we have \[0=H^1_\Iw(\mathbf{V}_3,\mathbf{D}_3)\hatotimes_{\mathcal{H}_B} (\mathcal{H}_B)_f\cong H^1(\mathbf{V}_3^\prime,\mathbf{D}_3^\prime).\]
    \par Note that there exists a map $(\mathcal{H}_B)_f\rightarrow\mathcal{H}_A$ which makes the following diagram commutative:
    \[ \begin{tikzcd}
        & \mathcal{H}_B \arrow[r, hookrightarrow] \arrow[d, "\widetilde{\pi}"] 
        & (\mathcal{H}_B)_f \arrow[ld, dashed] \\
          & \mathcal{H}_A &
    \end{tikzcd}\]
    
    This follows from the fact that the localization functor is exact, thus we get an exact sequence 
    \[
    0\rightarrow \widetilde{\frakp}_f\rightarrow(\mathcal{H}_B)_f \xrightarrow{{\widetilde{\pi}}_f} (\mathcal{H}_A)_{\overline{f}}\rightarrow 0,    
    \]
    where $\overline{f}:=f\ ({\rm mod}\ \widetilde{\frakp})=1$. Thus $(\mathcal{H}_A)_{\overline{f}}\cong\mathcal{H}_A$. Note that ${\widetilde{\pi}}_f$ is a $(\mathcal{H}_B)_f$-algebra map.
    \par Hence by the control theorem, we have 
    \[
    \RGammaIw(\mathbf{V}_\Ad,\mathbf{D}_\Ad)\cong\RGammaIw(\mathbf{V}_3,\mathbf{D}_3)\otimes^{\mathbf{L}}_{\mathcal{H}_B} \mathcal{H}_A \cong  \RGamma(\mathbf{V}_3^\prime,\mathbf{D}_3^\prime) \otimes^{\mathbf{L}}_{(\mathcal{H}_B)_f} \mathcal{H}_A, 
    \]
    and the result follows from the functoriality of taking determinants.
\end{proof}
\begin{remark}
    \par Geometrically this means that there is a neighborhood $U_f$ of $\Delta(\Spm(A))\times_E W$ in $\Spm(B)\times_E W$ such that the coherent sheaf $H^1_\Iw(\mathbf{V}_3,\mathbf{D}_3)$ (resp. $H^1_\Iw(\mathbf{V}_3^*(1),\mathbf{D}_3^\perp)$) restricted to $U_f$ vanishes (where $W$ is the component of the weight space $\mathcal{W}$ that we have fixed). Even though this neighborhood is not of the form $\widetilde{U}\times_E W$ for some neighborhood $\widetilde{U}$ of $\Delta(\Spm(A))$ in $\Spm(B)$, this is sufficient for our purposes. Unfortunately we cannot conclude that $H^1_\Iw(\mathbf{V}_3,\mathbf{D}_3)$ is finitely generated, since submodules of finitely generated modules are not finitely generated in general (which is however true for modules over Noetherian rings), we hope to be able to remove this assumption in the future.
\end{remark}

\bibliographystyle{amsalpha}
\bibliography{references}

\end{document}